\documentclass[10pt]{amsart}
\usepackage{graphicx}
 \usepackage{caption}
\usepackage{psfrag}
\usepackage{float}
\usepackage{subfigure}
\usepackage{psfrag}
\usepackage{epsfig}
\usepackage{epstopdf}
\usepackage{amsmath}
\usepackage{amscd}

\usepackage{amsthm,amsfonts,amsxtra,amssymb,amscd}
\usepackage[all]{xy}
\xyoption{2cell}
\usepackage{enumerate}

\usepackage{pdfsync}
\def\pd#1#2{\dfrac{\partial#1}{\partial#2}}
\def\ro{{\varepsilon^2}}

 \usepackage[unicode]{hyperref}

\theoremstyle{plain}
\theoremstyle{plain}
\newtheorem*{lemma*}{Lemma}
\newtheorem{lemma}{Lemma}
\newtheorem*{theorem*}{Theorem}
\newtheorem{theorem}{Theorem}
\newtheorem*{proposition*}{Proposition}
\newtheorem{proposition}{Proposition}
\newtheorem*{corollary*}{Corollary}

\newtheorem{corollary}{Corollary}
\theoremstyle{definition}
\newtheorem*{definition*}{Definition}
\newtheorem{constraint}{Constraint}
\newtheorem{definition}{Definition}
\newtheorem*{example*}{Example}
\newtheorem{example}{Example}
\theoremstyle{remark}
\newtheorem*{remark*}{Remark}\newtheorem*{proof*}{Proof}
\newtheorem{remark}{Remark}

\def\e{{\varepsilon}}

\def\A{{\mathcal A}}
\def\Gama{{\Gamma_S}}
\def\Gamati{{\tilde\Gamma_S}}
\def\GamaG{{\Gamma_S^{geo}}}
\def\Z{{\mathbb Z}}

\def\R{{\mathbb R}}
\def\C{{\mathbb C}}
\def\ome{{\omega}}
\def\Ome{{\Omega}}
\def\GA{{\mathcal A}}
\def\BU{{ \mathcal B}}
\def\N{{\mathbb N}}

\newcommand{\T}{\mathbb{T}}
\newcommand{\ii}{{\rm i}}
 \title[  Normal Forms for  the NLS ]{ A Normal Form for   the Schr\"odinger  equation with analytic non--linearities}
\author{ M. Procesi*, \and 
 C. Procesi{**}. }
 \thanks{ *Universit\`a Federico II  Napoli, supported by ERC grant ''New connections between dynamical systems and Hamiltonian PDEs '' and partially  by  the PRIN2009 grant "Critical Point Theory and Perturbative Methods for Nonlinear Differential Equations", {**}  Universit\`a di Roma, La Sapienza }

\begin{document}
\begin{abstract}
We discuss a class of normal forms of the completely resonant non--linear Schr\"odinger  equation on a torus.  We stress the geometric and combinatorial constructions arising from this study.
 \end{abstract}\maketitle
 
\tableofcontents
\section{Introduction} In this paper  we  exhibit a  normal form, with remarkable integrability properties, for the  completely resonant non--linear Schr\"odinger  equation on  the torus $\T^n,\ n\in\N$ (NLS for brevity):
\begin{equation}\label{main}-iu_t+\Delta u=\kappa |u|^{2q}u +\partial_{\bar u}G(|u|^2),\quad  q\geq 1\in\N .\end{equation}
Where $u:= u(t,\varphi)$, $\varphi\in \T^n$ and  $G(a)$ is a real analytic function whose Taylor series starts from degree $q+2$. The case $q=1$ is of particular interest and is usually referred to as the {\em cubic NLS}. 
\smallskip

It is well known that equation  \ref{main},  the NLS,   can be written as an infinite dimensional Hamiltonian dynamical system.
 
  It has the {\em energy}  $H=\int_{\T^n}(|\nabla(u)|^2+\kappa (q+1)^{-1}|u|^{2(q+1)}+G(|u|^2))\frac{d\phi}{(2\pi)^n}$,   the {\em momentum}  $M=\int_{\T^n} \bar u(\varphi) \nabla u(\varphi)\frac{d\phi}{(2\pi)^n}$  and  the {\em the mass} $L= \int_{\T^n} |u(\varphi)|^2\frac{d\phi}{(2\pi)^n}$,  as integrals of motion.
 
 Passing to the Fourier representation
\begin{equation}
u(t,\varphi):= \sum_{k\in \Z^n} u_k(t) e^{\ii (k, \varphi)}\ 
\end{equation} 
we have, up to a rescaling of $u$ and of time, 
 in coordinates:
\begin{equation}\label{Ham}H:=\sum_{k\in \Z^n}|k|^2 u_k \bar u_k \pm  \sum_{k_i\in \Z^n: \sum_{i=1}^{2q+2}(-1)^i k_i=0}\hskip-30pt u_{k_1}\bar u_{k_2}u_{k_3}\bar u_{k_4}\ldots u_{k_{2q+1}}\bar u_{k_{2q+2}}+\int_{\T^n} G(|u|^2)\frac{d\phi}{(2\pi)^n} , \end{equation}
We fix the sign to be $+$ since in our treatment it does not play any particular role.
 \bigskip

\subsection{Preliminaries}
 By Formula \eqref{Ham}, we can write equation  \eqref{main}  as an infinite dimensional Hamiltonian dynamical system, where the quadratic term  consists of infinitely many independent oscillators with rational frequencies and hence completely resonant (all the bounded solutions are periodic).  In order to study resonant systems  a standard instrument is the ``Resonant Birkhoff  normal form''. In Formula \eqref{Ham} denote by $K:=  \sum_{k\in \Z^n}|k|^2 u_k \bar u_k$.
 
 The first step of ``resonant Birkhoff  normal form'' is the sympletic change of variables which reduces the Hamiltonian $H$ to 
$$H=H_{\rm Res}+ H^{(2q+4)};\quad H_{\rm Res}= K  + H^{(2q+2)}_{res}(u,\bar u) $$ where $H^{(2q+4)}$ is an analytic function of degree at least $2q+4$ while $H^{(2q+2)}_{res}$ is of degree $2q+2$ and consists exactly of the degree $2q+2$ terms of \eqref{Ham} which Poisson commute with $K$. 
Then one wishes to treat   the {\em truncated system}   $H_{\rm Res}=K  + H^{(2q+2)}_{res}(u,\bar u)$, as the new unperturbed Hamiltonian and $H^{(2q+4)}$ as a small perturbation.  An ideal situation is when the truncated system is integrable, 
 this is the case for the cubic NLS in dimension 1, as shown  by Kuksin and P\"oschel  in \cite{KuP}. 
However  the special degenerations of the truncated system  used by these authors are not  valid in the case of the non-cubic NLS, already in dimension one, nor for the cubic case in dimension higher than one.

Although the truncated system appears to be very complicated (see formula \eqref{Ham2}) we show that it admits infinitely many invariant subspaces (cf. \S \ref{insub}), defined by requiring $u_k=0$ for all $k\notin S$ where $S=\{v_1,\ldots,v_m\}$, {\em tangential sites}, is some (arbitrarily large) subset of $\Z^n$ satisfying the {\em completeness condition}  (see Proposition 
\ref{completa}). 

The dynamics on these subspaces depends in a subtle way on the geometric properties of $S$, we show -- in Proposition \ref{completa} {\it ii)}-- that for generic choices of $S$ the behavior is integrable and that all the $|u_{v_i}|$ are constants of motion.  Suitable non--generic choices  of $S$  lead also to interesting non--integrable behavior as for instance in the paper \cite{KTa}.

By momentum conservation, it is easily seen that for  any set $S\subset \Z^n$, the subspace $u_k=0$ for all $k\notin {\rm Span}(S)$ is invariant. We restrict to this subspace\footnote{notice that this subspace is invariant   not only for $H_{Res}$ but also for the full Hamiltonian  $H$}
and denote by   $S^c:={\rm Span}(S)\setminus S$ the {\em normal   sites}.
 We  collect in $H_{Res}$  the terms by the degree (which we denote by $\# S^c$) in the variables $u_k,\ \bar u_k, \ k\in S^c$  we have
 $$H_{Res}=H_S+H_{\# S^c= 1}+H_{\# S^c= 2}+H_{\# S^c>2},$$ by definition the completeness is equivalent to the fact that the term of degree one is zero, i.e. $H_{\# S^c= 1}=0$.
\smallskip

We show that the term $H_{\# S^c>2}$ is negligible and we give an explicit formula for $H_{\# S^c= 2}$ described by an infinite dimensional matrix (cf. Formula \eqref{laquadra})  with coefficients depending on the ``tangential angles''. 
This is done explicitly by  1) putting the tangential variables in action--angle coordinates  and then 2)  introducing parameters for the  actions and finally 3) isolating the terms of the Hamiltonian $H_{Res}$ of degree  $\leq 2$. The resulting Hamiltonian is what we call the {\em normal form}, it is quadratic and explicitly described by a matrix which  depends  on the ``tangential angles''.  Hence the dynamics of this quadratic Hamiltonian is apparently non--integrable and given by an infinite set of coupled linear equations with {\em non--constant} coefficients. 

It is natural at this point to try to reduce the normal form to constant coefficients, exploiting the fact that $H^{(2q+2)}_{res}$ is {\em smaller} than $K$. However  the quadratic term $K$ is very degenerate and does not satisfy the {\em second Melnikov condition}, hence the perturbative methods, see for instance \cite{El}, fail. In finite dimensional systems one can still approximately reduce (to constant coefficients) matrices  whose diagonal part does not satisfy the second Melnikov condition,  see \cite{Cha}. This is done via a change of variables which is not close--to--identity and hence must be constructed explicitly.  In our infinite dimensional setting however this kind of results are {\em not} applicable, since in general the change of variables suggested by  the finite dimensional analog is not analytic.

\subsection{ The object of this paper}
The main contribution of this paper is to construct, for generic $S$,   an explicit analytic symplectic change of variables which removes the dependence form the tangential angles so that the Normal form is block--diagonal (with blocks of dimension $\leq 2n$) and integrable, see Theorem \ref{teo1} for a precise statement. Notice that this symplectic transformation is {not} close--to--identity. It is given by explicit algebraic formulas (Formula \eqref{labella}) and not constructed through a recursive algorithm.  This is due to the fact that we can achieve a complete control on the diagonal blocks of the normal form. In turn this is done by codifying the corresponding matrix in terms of  graphs, see Definition \ref{iduegrafi}, and describing   the possible  blocks which may appear in the normal form,  depending on the choice of the tangential sites,  combinatorially using {\em finitely} many graphs. 
\smallskip

Then we   find    {\em optimal}  constraints on the tangential sites, given by a finite list of polynomial inequalities on the coordinates of $S$. If $S$ satisfies these inequalities we say that it is {\em generic} and then, these      constraints    make the normal form {\em as simple as possible}.   

 We  organize our constraints in 6 different requirements,  summarized in  Definition   \ref{fincon}. 
 Under these constraints the normal form is block--diagonal with blocks of dimension bounded by $n+1$, except finitely many exceptional blocks of size bounded by $2n$. The diagonal blocks are explicitly described as functions of the average tangential actions $\xi$ and angles $x$.
 \smallskip

Then, for these infinitely many choices of the tangential sites $S$,  we exhibit\footnote{  In general, in order  to construct a change of variables one  solves a Hamilton--Jacobi equation, finding a generating function for the change of variables. In our case however we do not use this procedure, indeed the change of variables was {\em  guessed} directly.} a symplectic change of variables (cf. Formula \eqref{labella}) which makes the normal form with constant coefficients and still block--diagonal.

Finally we show that, in dimension one and two, the normal form has both stable and unstable regions,
namely there are open sets for the parameters $\xi$ where the normal form is completely elliptic--hence its Hamiltonian flow is stable. For all the remaining values of the parameters $\xi$ there are a finite number of unstable directions. In the stable region one may perform a further analytic change of variables which reduces the normal form to the standard elliptic one $(\ome(\xi),y)+ \sum_k \bar\Ome_k |z_k|^2$ (cf. Corollary  \ref{ficof}).

 \subsection{Some related literature} The idea of choosing an appropriate set of tangential sites $S$  was first used by Bourgain in \cite{Bo2} in a slightly different context.  He studied the cubic NLS in dimension two and proved the existence of quasi--periodic solutions with two frequencies by using a combination  of Lyapunov-Schmidt reduction tecniques and a Nash--Moser algorithm to solve the small divisor problem (the so--called {\em Craig--Wayne--Bourgain} approach, see \cite{CW} ,\cite{Bo2} and  for a recent generalization also \cite{BBhe}). In \cite{Bo2} it is shown that, for appropriate choices of the tangential sites, one may find simple solutions for the bifurcation equation where only the Fourier indexes of the tangential sites are excited. This strategy was generalized by Wang in \cite{W09} to study the NLS on a torus $\T^n$ and prove existence of quasi periodic solutions with $n$ frequencies. A similar idea was exploited in \cite{GP2} and \cite{GP3} to look for   ``wave packet'' periodic solutions (i.e. periodic solutions which at leading order excite an arbitrarily large number of ``tangential sites'') of the cubic NLS in any dimension both in the case of periodic and Dirichlet boundary conditions.
  
  In the context of KAM theory and normal form, this idea  was used by Geng in  \cite{G} for the NLS in dimension one with the nonlinearity $|u|^4u$.  
  
 A similar strategy is used   by Geng-You and Xu in \cite{GYX}, to study the cubic NLS in dimension two. In that paper the authors show that one may give constraints on the tangential sites so that the normal form is non--integrable (i.e. it depends explicitly on the angle variables) but block diagonal with blocks of dimension $2$.    They apply this result to perform a KAM algorithm and prove existence (but not stability) of quasi--periodic solutions. 
 We also mention the paper \cite{MP}, which studies the non-local NLS and the beam equation both for periodic and Dirichlet boundary conditions. The main result  of that paper  is that,  by only requiring very simple  constraints on the tangential sites, the leading order of the normal form Hamiltonian is quadratic and block diagonal, with blocks of uniformly bounded dimension.  
 
 Finally we mention the preprints by Wang \cite{W1} and \cite{W2}, which use the {\em Craig--Wayne--Bourgain} approach to study quasi--periodic solutions for the NLS  \eqref{main} in any dimension.

\subsection{ Description of the paper:}   In Section \ref{due} we introduce some necessary Hamiltonian formalism, we perform the Birkhoff change of variables and study the {\em truncated system} $H_{Res}$. In Particular we study invariant subspaces and in Propositions \ref{completa} and \ref{compleH} we give conditions for their {\em completeness} and {\em integrability}. Finally we pass to the elliptic--action angle variables and define the functional domains in which we work.
All the results and techniques of this section are pretty standard so we try to review them concisely. 

Having introduced the relevant notations, in Section \ref{tre} we give the notion of {\em generic} tangential set $S$ and state our main results Theorem \ref{teo1} and Corollary \ref{ficof}.

In Section \ref{quattro} we impose Constraint \ref{co1} on the tangential sites $S$, this enables us to define our normal form $N$ -- see Proposition \ref{quadrat}--and prove that  $N$  satisfies  non--degeneracy in the action variables--  see Proposition \ref{nond}. Finally we discuss the perturbation $P$ and estimate its size-- see Proposition \ref{gliord}.

In Section \ref{cinque} we define two spaces $V^{0,1}$ and $F^{0,1}$ on which  we study the linear operator $ad(N):=\{N,\cdot\}$. This gives two matrix descriptions of $N$.

In Section \ref{Gra} we describe the two matrices in terms of two graphs $\Gamati$ and $\Lambda_S$ with vertices respectively the basis elements of  $V^{0,1}$ and $F^{0,1}$ and edges connecting those couples of elements which have a non--zero matrix coefficient.  This is a standard way to display infinite matrices, in particular one easily sees that the connected components of the graph correspond to block--diagonal terms in the matrix.

From these graphs we deduce a more abstract {\em geometric graph}  $\Gama$ which still contains all the information  necessary to compute the matrix entries of $ad(N)$. 

In \ref{seiuno} we define a graph $\GamaG$ with vertices on $\R^n$ which contains $\Gama$ but is easier to study. With this notations we prove-- Proposition \ref{rozzo}-- a first rough bound on the dimension of the block--diagonal blocks in $ad(N)$.

Finally in \ref{seidue}-- Theorems \ref{sunto} and \ref{oneone}-- we state our main results on the connected components of $\GamaG$ and $\Gama$, this is the core of the paper. 
It is interesting to notice that these results hold  independently of the number of tangential sites and hence remain true also if one excites infinitely many tangential sites.

It is possible that  our constraints may be   improved (the best possible result is that the geometric constraints are sufficient to bound   the dimension of all the blocks  by $n+1$). This is actually true in low dimensions $n=1,2$ for all $q$. For $q=1$ we believe it to be true in any dimension, this will be discussed in a separate paper.  

In Section \ref{sette} we formalize our graphs as subgraphs of a Cayley graph (we group the relevant  definitions and properties  in the Appendix).  This is the content of Proposition \ref{gliegg} and enables us to endow our graphs with a group action that simplifies the  combinatorial analysis.

In Section \ref{otto} we  impose Constraints \ref{co4} and \ref{co5}.  This enables us to identify the connected components of $\Lambda_S$ with those of $\Gama$-- see Proposition \ref{coqe2} and Corollary \ref{cabel}.  The isomorphism between the   connected components of the two graphs is the key point which allows us to construct the change of variables which sends $N$ to constant coefficients, as can be seen in Example \ref{redEGG}. 

In section \ref{eqgama} we define a finite set of connected graphs, the {\em possible combinatorial graphs}. To a graph $\GA$ of this set, with $k$ vertices, we associate a list of  $k-1$ linear and quadratic equations  in $n$ variables, given by   Formula \eqref{bacosG}.  Then in Proposition \ref{brik} we show that $\GA$ is isomorphic to a connected subgraph of $\Gama$ if and only if its equations have solutions in $S^c$ (the solutions are  identified with a special vertex in $\Gama$, called the root).
 This enables us to describe the infinite connected components of $\Gama$ via a finite set of graphs.

To a   possible combinatorial graph $\GA$  we associate its 
 equations  \eqref{bacosG}, which  have as coefficients linear and quadratic functions  of the tangential sites.
 If these equations do not have real solutions for generic choices of $S$ then  $\GA$ cannot be isomorphic to a connected subgraph of $\Gama$
for generic $S$. This is a geometric condition from which   one expects to be able to rule out the connected components of $\Gama$ as soon as $k-1\geq n+1$ by imposing that those  overdetermined system of equations be generically incompatible. 
However  this simple idea does not cover various pathological cases. We try to give an idea of the main problems.

Given a graph $\GA$  with $k$ vertices  its equations \eqref{bacosG} may not be of maximal rank for particular choices of $S\subset \Z^{k-1}$, this can be avoided by introducing appropriate generiticity  constraints, as Constraint \ref{co8}. 
Unfortunately it may well be, see Example \ref{pirir},  that equations \eqref{bacosG} are linearly dependent for all choices of $S$, independently of the dimension $n$ such that $S\in \Z^n$.  In this case one is faced with a compatibility problem, namely one can try to exclude this graphs by requiring that the equations are incompatible for generic choices of $S$, see Example \ref{pirir} and Constraint \ref{co6}.

 This does  not conclude the analysis since it is possible that the equations
 be always compatible, see Remark \ref{bala}.  
 So it is possible that one has a graph with $k>n+1$ vertices but still with rank $\leq n$, this is the reason of our bound $k\leq 2n$. 
 To simplify the problem we introduce the notion of {\em colored rank}, see Definition \ref{ranghi}, we have Theorem \ref{ridma}.

In section \ref{Magt}, using Theorem \ref{ridma},  we discuss {\em possible combinatorial graphs} $\GA$ with  rank $r=n+1$, when $S\subset \Z^n$.

 We  prove that if their equations are   always compatible then their (unique) solution must be a point in $S$. This means that $\GA$ cannot be isomorphic to a connected subgraph of $\Gama$ (which has vertices in $S^c$).

This enables us to prove Theorems \ref{sunto} and \ref{oneone}.

In Section \ref{lapro} we prove Theorem \ref{main} by exhibiting in Formula \ref{labella} the change of variables which reduces the normal form $N$ to constant coefficients.  We also give explicit formul\ae \ which allow to compute $N$ in this new set of variables, via the {\em combinatorial graphs}.

In section \ref{pfiu} we prove Proposition \ref{teo2} and Corollary \ref{ficof}. The most relevant  notion is that of arithmetic constraint. Roughly speaking we want to ensure that if a combinatorial graph $\GA$ is such that its equations have a {\em unique solution}  in $\R^2$, then this solution is not integer valued. 

This result enables us to prove the existence of stable regions for the parameters $\xi$, where $N$ is purely elliptic.  Corollary \ref{ficof} follows from the theory of Quadratic hamiltonians and from Proposition \ref{teo2}.
 
 \vskip5pt
{\em Acknowledgments:} We wish to thank   Nguyen Bich Van  for correcting some formulas, Dorian Goldfeld  and Jonathan Pila for introducing us to the paper \cite{BP}  and finally Luca Biasco and Massimiliano Berti for several useful suggestions and discussions.

\section{Hamiltonian formalism}\label{due}
 We work on the scale of complex Hilbert spaces 
\begin{equation}\label{scale}
{\bf{\bar \ell}}^{(a,p)}:=\{ u=\{u_k \}_{k\in \Z^n}\;\big\vert\;|u_0|^2+\sum_{k\in \Z^n} |u_k|^2e^{2 a |k|} |k|^{2p}:=||u||_{a,p}^2 < \infty \},
\end{equation}$$\ a>0,\ p>n/2\,,$$
 equipped with the symplectic structure $\ii \sum_{k\in \Z^n} d u_k\wedge d \bar u_k $.

These choices are rather standard in the literature and consist in requiring that the functions $u(\varphi)$   extend to analytic functions in    the complex domain $|Im(\varphi)|< a$, with Sobolev regularity on the boundary, the condition $p>n/2$ ensures that our function spaces are Hilbert algebras.

 \begin{remark}\label{ivt}
It is not necessary to assume that the torus $\T^n=\R^n/\Z^n$.  The theory works and in fact we shall apply it, also if $\T^n=\R^n/\Lambda$ where $\Lambda$ is a lattice generated by a not necessarily orthonormal basis. \end{remark}

 We may write, for any $d$
\begin{equation}\label{albet}
[u]^{2d}:=\sum_{k_i\in \Z^n} u_{k_1}\bar u_{k_2}u_{k_3}\bar u_{k_4}\ldots u_{k_{2d-1}}\bar u_{k_{2d}}= \sum_{\alpha,\beta\in (\Z^n)^\N: \atop {|\alpha|=|\beta|=d}} \binom{d}{\alpha}\binom{d}{\beta}u^\alpha\bar u^\beta, 
\end{equation}  where $\alpha:k\mapsto \alpha_k\in\N$ and $u^\alpha=\prod_k u_k^{\alpha_k}$, same for $\beta$.
It is easily seen that for any $d$ the function $[u]^{2d}$ is an analytic function of $u,\bar u$.
 Moreover  $[u]^{2d}$ is {\em regular}, namely its Hamiltonian vector field is an analytic function from ${\bf{\bar \ell}}^{(a,p)}\times {\bf{\bar \ell}}^{(a,p)} $ to itself. 
 
 In formula \eqref{Ham} we may expand $G$ in Taylor series obtaining a totally convergent sum of terms  $[u]^{2d}$, this shows that our Hamiltonian is analytic and regular.
\smallskip

The torus $\T^n$ acts on itself by translations leaving invariant the symplectic form, in fact it  gives rise in this way to a {\em moment map} in the sense of symplectic Geometry or a {\em momentum vector} in the language of Mechanics.  The Hamiltonian is invariant  under translation so by Noether's Theorem it Poisson commutes with momentum.  

We  thus will systematically apply the fact that our Hamiltonian $H$ (see Formula \eqref{Ham}) has  $n+1$ conserved  quantities:
the  $n$--vector {\em momentum} $M=\sum k |u_k|^2$ and the scalar {\em mass}  $L= \sum_k |u_k|^2$, with 
\begin{equation}
\{M,u_h\}=\ii h u_h,\ \{M,\bar u_h\}=-\ii h \bar u_h,\ \{L,u_h\}=\ii u_h,\ \{L,\bar u_h\}=-\ii \bar u_h.\ 
\end{equation}\smallskip
The terms in equation \eqref{albet} commute with $L$. The conservation of momentum selects the terms with $\sum_k (\alpha_k-\beta_k)k=0$.
A first useful consequence of the conservation of momentum is that given any set $S\subset \Z^n$, setting 
$$\bar\ell^{(a,p)}_S:=\{ u\in \bar\ell^{(a,p)}:\; u_k=0,\; \forall k \notin \;{\rm Span}(S)\}, $$ $\bar\ell^{(a,p)}_S\times \bar\ell^{(a,p)}_S $ is an invariant subspace for the dynamics. 

\begin{remark}
\label{idt}This has the following geometric interpretation, the lattice $\Lambda:={\rm Span}(S)\subset\Z^n$  is of some rank $k$ and it is the character group of a  torus  $\bar T=\R^k/\Lambda$  with a natural map  $\pi:T\to \bar T$.  Under this map a simple variant of the space $\bar\ell^{(a,p)}$ for the torus $\bar T$ is identified to $\bar\ell^{(a,p)}_S$. \end{remark}

 \subsection{One  step of Birkhoff normal form } A  monomial $u^\alpha\bar u^\beta$ is an eigenvector of $ad(K):=\{K,-\}$  of eigenvalue $\sum_k (\alpha_k-\beta_k)|k|^2$ where $K$  is the quadratic part \begin{equation}
K:=\sum_{k\in \Z^n}|k|^2 u_k \bar u_k \quad\text{quadratic energy}.
\end{equation}   

We apply a step of Birkhoff normal form (cf. \cite{Bo3},\cite{Bo2},\cite{bamb}),  by which we cancel all the  terms  of degree $2(q+1)$ which do not Poisson commute with   $K$.   This is done by applying  a well known analytic change of variables, with generating function $$A:=\sum_{\alpha,\beta\in (\Z^n)^\N: |\alpha|=|\beta|=q+1\atop {\sum_k (\alpha_k-\beta_k)k=0\,,\;\sum_k (\alpha_k-\beta_k)|k|^2\neq 0}} \hskip-10pt\binom{q+1}{\alpha}\binom{q+1}{\beta}\frac{u^\alpha\bar u^\beta}{\sum_k (\alpha_k-\beta_k)|k|^2}\,.$$ We denote the change of variables by $\Psi^{(1)}:=e^{ad(A)}$ and notice that it is well defined and analytic:    $   B_{\epsilon_0}\times B_{\epsilon_0} \to B_{2{\epsilon_0}} \times B_{2\epsilon_0}$,  with $\epsilon_0=  (2 c_{a,p})^{-1}$ (here $B_{r}$ denotes the open ball of radius $r$   and  $c_{a,p}$ is the algebra constant of  the space\footnote{Notice that the unperturbed Hamiltonian $K$ is completely resonant so $A$ does not have small divisors.  Since 
$\bar\ell^{(a,p)}$ is a Hilbert algebra, this implies that the change of variables does not loose regularity.} $\bar\ell^{(a,p)}$) .

   By construction $\Psi^{(1)}$  brings (\ref{Ham}) to  the form $H= H_{\rm Res} +P^{2(q+2)}$ where:
\begin{equation}\label{Ham2}H_{\rm Res}:=\sum_{k\in \Z^n}|k|^2 u_k \bar u_k +  \hskip-30pt\sum_{\alpha,\beta\in (\Z^n)^\N: |\alpha|=|\beta|=q+1\atop {\sum_k (\alpha_k-\beta_k)k=0\,,\;\sum_k (\alpha_k-\beta_k)|k|^2=0}} \hskip-10pt\binom{q+1}{\alpha}\binom{q+1}{\beta}u^\alpha\bar u^\beta\,,
\end{equation} 
$P^{ 2(q+2) }(u)$  has degree at least $2(q+2)$ in $u$, it is is analytic and regular  and satisfies the bound:
$$\sup_{ (u,\bar u)\in B_{\e}\times B_{\e}} \| X_{P^{2(q+2)}}\|_{a,p} \leq {\rm cost} \,\e^{2q+3}\,, \quad \forall \e< \epsilon_0$$
where cost denotes a universal constant (depending only on $q$, $c_{a,p}$ and the function $G$).
\begin{remark}
The three constraints in the second summand of formula \eqref{Ham2} express the conservation of $L$, $M$ and  the quadratic energy $K  $.
\end{remark}
\smallskip

  \vskip10pt

\begin{definition}\label{reson}
We say that a list $k_1,\dots, k_{2d}$ of vectors in $\Z^n$ is {\em resonant}  if,
up to reordering, we have
$$ k_1+k_3\dots+k_{2d-1}= k_{2}+\dots + k_{2d} \,,\quad |k_1|^2+\dots+|k_{2d-1}|^2= |k_{2}|^2+\dots + |k_{2d}|^2. $$
We say that the list  is {\em integrable} if furthermore, up to reordering, we have $k_{2i-1}=k_{2i},\ i=1,\ldots,d$.
\end{definition}
 
 The resonant lists with $d=q+1$ describe the  {\em resonant monomials}, that is  those monomials which Poisson  commute with $K$, which appear in $H_{\rm Res}$.  The integrable lists describe the monomials in $|u_h|^2$.
 \begin{figure}[!ht]
 
\begin{minipage}[c]{9cm}{ \begin{example}\label{conin}[$q=1$] $$k_1+k_3= k_2+k_4, \qquad  |k_1|^2+|k_3|^2=| k_2|^2+|k_4|^2$$ is equivalent to \begin{equation}\label{vinco}k_1+k_3= k_2+k_4,\qquad( k_1-k_2,k_3-k_2)=0\end{equation}\end{example}
}\end{minipage}\hskip20pt\begin{minipage}[c]{3cm}
{
\psfrag{a}{$k_1$}
\psfrag{b}{$k_2$}
\psfrag{c}{$ k_3$}
\psfrag{d}{$ k_4$}

\includegraphics[width=3cm]{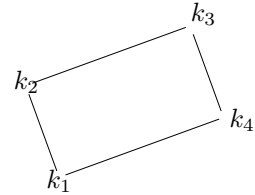}
}
\label{fig0}
\end{minipage}\caption{A resonant quadruple $k_1,k_2,k_3,k_4$}
\end{figure}

Notice that a quadruple $k_1,k_2,k_3,k_4$ is  resonant if these points are the vertices of a rectangle, it is integrable if and only if the corresponding rectangle is degenerate.
 \subsubsection{Invariant subspaces.\label{insub}}
 In view of Remark \ref{idt} we wish to study  the Hamiltonian $H_{\rm Res}$ on the invariant subspaces $\ell^{a,p}_S$ for suitable choices of $S$. We want to characterize those subsets $S\subset \Z^n$, such that the Hamiltonian vector field $X_ {H_{\rm Res}}$ is tangent to the subspace of equation
 $$ u_k=0=\bar u_k \,,\quad \forall k\in S^c:={\rm Span}(S)\setminus S, $$ this of course implies that this subspace is stable under the dynamics, a set $S$ with this property is called {\em complete}. We denote by $H_S$ the Hamiltonian $H_{\rm Res}$ restricted to such a subspace, naturally $H_S$ depends only on $u_k,\bar u_k$ with $k\in S$.
 
 The next statement follows immediately from the definitions:
 \begin{proposition}\label{completa}
\begin{enumerate}\item$S$ is complete if and only if, for any choice of $2q+1$  vectors $v_i\in S$ the following holds:

\noindent {\em If there exists a further vector $w\in \Z^n$ such that the list $v_1,\dots, v_{2q+1},w$ is resonant then $w\in S$.}

\item If  all the lists in $S$ of $2q+2$ elements  which are resonant  are also integrable,  then $H_S$ depends only on the elements $|u_h|^2$ with $h\in S$.  \end{enumerate}
\end{proposition}
\begin{remark}\label{piro}
 A sufficient condition for $S$ to be integrable is the following: set  $S= \{v_1,\dots ,v_m\} $, introduce variables $e_i$ with $i=1,\dots,m$.  For any choice of $2q+2$ elements $e_{i_1},\dots e_{i_{2q+2}}$ if the expression
  $$ e_{i_1}+ \dots+ e_{i_{q+1}}-(e_{i_{q+2}}+\dots +e_{i_{2q+2}})$$ is not zero then
  $$ v_{i_1}+ \dots+ v_{i_{q+1}}-(v_{i_{q+2}}+\dots +v_{i_{2q+2}})\neq 0.$$
  We have thus shown that completeness and integrability are {\em genericity} condition  on $S$, the first of many which we will impose.
\end{remark}
  \begin{example}
$q=1,\ n=2,m=4$  \quad Four vectors $v_1,v_2,v_3,v_4$  in the plane are not complete  if they form a picture of type
$$\xymatrix{{\circ}\, v_1&&   &{\circ}\, v_4\\&{\circ}\,v_2&   &{\circ}\,v_3 }$$ that is we have a right triangle  which is not completed to a rectangle.

The list
 $$\xymatrix{{\circ}\, v_1&   &{\circ}\, v_4\\{\circ}\,v_2&   &{\circ}\,v_3 }$$ is complete but not integrable, and finally
 $$ \xymatrix{&{\circ}\, v_1&{\circ}\, v_4&\\{\circ}\,v_2& &  &{\circ}\,v_3 }$$
is complete and integrable.
\end{example}
  
When we partition $${\rm Span}(S)= S\cup S^c,\quad S:=(v_1,\ldots,v_m),$$ where $S$ is complete, we call
the elements of  $S$   {\em tangential sites} and of $S^c$ the {\em normal sites}.   Of course the word tangential is justified by the fact that the Hamiltonian vector field is tangent to the subspace parametrized by the coordinates in $S$.\smallskip

We introduce 
\begin{equation}\label{symme}
A_r(\xi_1,\ldots,\xi_m)= \sum_{ \sum_i k_i=r} {\binom{r}{k_1,\dots, k_m}}^2 \prod_i \xi_i^{k_i}.
\end{equation}
\begin{proposition}\label{compleH}
If $S$ is complete and  integrable  the restricted Hamiltonian is
$$H_S=\sum_{i=1}^m |v_i|^{2 } |u_{v_i}|^{2 }+ A_{q+1} (|u_{v_1}|^2,\ldots,|u_{v_m}|^2)$$$$=\sum_{i=1}^m |v_i|^{2 } |u_{v_i}|^{2 }+ \sum_{ \sum_i k_i=q+1} {\binom{q+1}{k_1,\dots, k_m}}^2 \prod_i |u_{v_i}|^{2k_i}.$$
\end{proposition}
\begin{proof}
This follows immediately from Formula \eqref{Ham2}   and the definitions.
\end{proof}
 \subsection{Tangential sites in action--angle coordinates\label{cts}} 
  
We set \begin{equation}
\label{chofv}u_k:= z_k \;{\rm for}\; k\in S^c\,,\quad u_{v_i}:= \sqrt {\xi_i+y_i} e^{\ii x_i}= \sqrt {\xi_i}(1+\frac {y_i}{2 \xi_i }+\ldots  ) e^{\ii x_i}\;{\rm for}\;  i=1,\dots m,
\end{equation}  considering  the $\xi_i>0$ as parameters $|y_i|<\xi_i$ while $y,x,w:=(z,\bar z)$ are dynamical variables.  
\begin{definition}
 We  denote   by   $ \bf{\ell}^{(a,p)}$   the subspace of $\bf{\bar \ell}^{(a,p)}\times\bf{\bar \ell}^{(a,p)} $  generated by the indices in $S^c$ with coordinates  $w=(z,\bar z) $.
\end{definition}
 For  all $\e >0$ and for all 
\begin{equation}\label{Aro}
\xi \in A_{\ro}:= \{  \xi\,:\  \frac12 \ro \le \xi_i\le \ro \,\}\,,
\end{equation}
Formula \eqref{chofv} is a well known analytic and symplectic change of variables $\Psi^{(2)}_\xi$ in  the  domain
\begin{equation}\label{domain}  D_{a,p}(s,r)= D(s,r):= 
 \{   x,y,w\,:\   x\in \T^m_s\,,\  |y|\le r^2\,,\  \|w\|_{a,p}\le r\}\subset \T^m_s\times\C^m\times {\bf{\ell}}^{(a,p)}.
\end{equation} Here $\e>0$, $s>0$ and $ 0<r<\e/2$ are auxiliary parameters. $\T^m_s$ denotes the open subset of the complex torus $\T_{\C}^m:=\C^m/2\pi\Z^m$ where   $ x\in\C^m,\ |$Im$(x)|<s$. Moreover if \begin{equation}
\label{bapa}\sqrt{2 m} (\max(|v_i|)^{p} e^{ ( s+ a\max(|v_i|))} \e  < {\epsilon_0}\,,
\end{equation} the change of variables sends $D(r,s)\to B_{{\epsilon_0}}$ so we can apply it to our Hamiltonian.

We thus assume that the parameters $\e,\, r,\, s$  satisfy \eqref{bapa}.

Formula \eqref{chofv}  puts  in action angle variables  $(y;x)= (y_1,\dots, y_m;x_1,\dots, x_m) $ the tangential sites, close to the action $\xi= \xi_1,\dots, \xi_m$, which are parameters for the system.  The  symplectic form is now $ dy \wedge dx + i \sum_{k\in S^c} dz_k\wedge d \bar z_k $.   

 \medskip

Following \cite{Po} we study {\em regular} functions $F:A_{\ro }\times D_{a,p}(s,r)\to \C$, that is whose Hamiltonian vector field  $X_F$ is analytic from $D(s,r)\to \C^m\times\C^m\times\ell_{a,p}$. In the variables $\xi$ we require Lipschitz regularity.  We use the weighted norm:
\begin{equation}
\label{weno} \Vert X_F\Vert^\lambda_{s,r}= \sup_{  A_{\ro }\times D(s,r)}\|X_F\|_{s,r}+\lambda \sup_{\xi\neq \eta \in A_{\ro }\,,\;(x,y,w)\in D(s,r)}\frac{\|X_F(\eta)-X_F(\xi)\|_{s,r}}{|\eta-\xi|},  
\end{equation}  where $\lambda =  \ro $  and $$ \|X_F\|_{s,r}:= r^{-2}|\partial_x F|+s^{-1}|\partial_y F|+r^{-1}\|\partial_w F\|_{a,p}.$$
The different weights ensure that if $\Vert X_F\Vert^\lambda_{s,r}<\frac12$ then $F$ generates a close--to--identity symplectic change of variables from $D(r/2,s/2)\to D(r,s)$.
  \subsubsection{Quadratic Hamiltonians} We have the rules of Poisson bracket
\begin{equation}
\label{lzpb}\{y_i,y_j\}=\{x_i,x_j\}= 0,\ \{y_i,x_j\}=\delta^i_j,\ \{z_h,z_k\}=\{\bar z_h,\bar z_k\}=  0,\ \{\bar z_h,z_k\}=\ii \delta^h_k.
\end{equation}
If we define $w$ as the infinite row vector $w$ with coordinates $ z_h$  and then $\bar z_h$  and $J$ the standard skew symmetric matrix $J:=\begin{vmatrix}
0&-1\\1&0
\end{vmatrix}$ we have the Poisson bracket\footnote{The apex $t$ is the transpose}   $\{w^t ,w \}=\ii J $.  Thus a quadratic Hamiltonian $\mathcal Q(w)$   in the elements of $w$  represents by Poisson bracket  a linear transformation  on the space with basis $w$.  If $\mathcal Q(w)$ is real, the matrix of this linear transformation is purely imaginary,  thus it is convenient to denote it by   $\ii Q$ and write  $ad(\mathcal Q):=\{\mathcal Q(w),w^t\}= \ii Qw^t$ . The equations of motion are $\dot w =\ii w Q^t$. The matrix $Q$  is related to the quadratic expression by\footnote{The parentheses represent the scalar product in $\R$. }  
\begin{equation}\label{represQ}
\mathcal Q(w)= \frac 12 (w, wJQ^t)=-\frac1 2  wQJw^t.
\end{equation}
Quadratic Hamiltonians are closed under Poisson bracket and, by Jacobi's identity, if  $\mathcal Q_1(w),\mathcal Q_2(w)$ correspond to matrices $Q_1,Q_2$ then $\{\mathcal Q_1(w),\mathcal Q_2(w)\}$ corresponds to $[Q_1,Q_2]$. Moreover, a quadratic Hamiltonian $\mathcal Q$ has $\|X_{\mathcal Q}\|_{r,s}<\infty$ if and only if  its matrix $Q$ is such that $QJ$ is  a continuous symmetric linear operator from $\ell_{a,p}$ to itself.
  \section{Main dynamical results}  \label{tre}
\subsubsection{Generiticity conditions\label{gen}}    Our Theorems hold under some constraints on $S$ such as those of  Remark \ref{piro}. These constraints are expressed by the condition that the list of vectors $S$, thought of as a point in $\Z^{mn}$, does not lie in any of the varieties defined by a finite list of polynomial equations, called the {\em avoidable resonances}.
\smallskip

In order to explain this let us establish some simple language. \begin{definition}\label{gener} Given a list $\mathcal R:=\{P_1(\zeta),\dots, P_{U}(\zeta)\}$ of polynomials in $d$ vector variables $\zeta_i$, called {\em resonance polynomials},
we say that a list of    vectors  $S=\{v_1,\dots,v_m\}, v_i\in\C^{n }$ is  {\em generic}  relative to $\mathcal R$ if, for any list  $A=\{u_1,\ldots,u_{d}\}$   such that $u_i\in S,\ \forall i,\ u_i\neq u_j\ \text{if}\ i\neq j$, the evaluation of the resonance polynomials at $\zeta_i=u_i$ is non--zero. 
\end{definition}If $m$ is finite this condition is equivalent to requiring that $S$ (considered as a point in  $\C^{nm}$) does not belong to the algebraic variety where at least one of the resonance polynomials is zero.

In our specific case the condition of being generic for the tangential sites $S$
is expressed by a finite list of  non--zero polynomials with integer coefficients   depending on $d=  4q(n+1)$  vector variables $\zeta=(\zeta_1,\dots,\zeta_{d})$ with $\zeta_i=(\zeta_i^1,\ldots,\zeta_i^n) $. 
The explicit list of these resonances  (see Definition \ref{fincon}) depends on some  non trivial combinatorics, nevertheless it is easy to give a  (highly) redundant list of  inequalities out of which the resonances appear. There is a constant $C>0$ depending only on $q,n$ so that we can take as resonances the non--zero polynomials of the form:

\smallskip

i) {\em Linear inequalities}\quad For all non--zero  vectors $(a_1,\dots, a_{4q (n+1)})$ with $\ a_i\in\mathbb Z, \ |a_i|\leq C,\   $ we require that
$$\sum_{i=1}^{4q (n+1)}a_i\zeta_i\neq 0,$$

ii) {\em Quadratic inequalities}\quad Let $(\zeta_i,\zeta_j)=\sum_{h=1}^n\zeta_i^h\zeta_j^h$ be the {\em scalar products}. 
For all non zero matrices $\{a_{i,j}\}_{i,j=1}^{4 q (n+1)}$ with  $a_{i,j}\in\mathbb Z, \ |a_{i,j}|\leq C, $  we require
$$\sum_{i,j=1}^{4 q (n+1)}a_{i,j}(\zeta_i,\zeta_j)\neq 0.  $$

iii) {\em Determinantal inequalities}\quad   Consider $n$  linear combinations $u_h$ out of the list of elements $\mathcal L:=\{\sum_{i=1}^{4 q(n+1) }a_{h,i}\zeta_i,\  a_{h,i}\in\mathbb Z, \ |a_{h,i}|\leq C\}$. 

The determinantal resonances are contained in the list of the  formally non--zero expressions of type $\det(u_1,\ldots,u_n),\ u_i\in\mathcal L$.

\smallskip

 Given any $m\in \N,$ let  $S=\{v_1,\dots,v_m\}\in\nobreak  \Z^{nm}$ be a {\em generic} choice of the tangential sites. 
 \begin{theorem}\label{teo1}    For all $r,s,\varepsilon$ satisfying \eqref{bapa} and for all $\xi\in A_\ro$,  there exists   an analytic symplectic change of variables:
$$\Phi_\xi: ( y, x)\times (z,\bar z) \to (u, \bar u) $$   from $  D(s,r/2) \to B_{2\epsilon_0}$  such that the Hamiltonian \eqref{Ham} in the new variables  is analytic and has the form 
$$ H\circ\Phi_\xi= (\ome(\xi),y) +\sum_{k\in S^c}{\tilde\Ome}_k|z_k|^2 +\tilde{\mathcal Q}(\xi,w)+ \tilde P(\xi,y,x,w)\,,$$ where
\smallskip

\noindent i) {\bf Non--degeneracy:} \ $\ome_i(\xi)- |v_i|^2  $ is homogeneous of degree $q$.

The map $(\xi_1,\ldots,\xi_m)\mapsto (\ome_1(\xi),\ldots,\ome_m(\xi))$ is a diffeomorphism for  $\xi$ outside a real algebraic hypersurface. 
\smallskip

\noindent ii) {\bf Asymptotic of the normal frequencies:} We have ${\tilde\Ome}_k= |k|^2 +\sum_{i=1}^m |v_i|^2 L^{(i)}(k)$ where $L^{(i)}(k)\in \Z$ satisfy $|L^{(i)}(k)|\leq 4nq$.
\smallskip

\noindent iii) {\bf Reducibility:}\ The matrix      $\tilde{ Q}(\xi)$ which represents the quadratic form	$\tilde{\mathcal Q}(\xi,w)$(see formula \eqref{represQ})   depends  only on the variables $\xi$  
and all its entries are homogeneous of degree $q$ in these variables.  It is   block--diagonal  and satisfies  the following properties: 

\quad   All of the blocks except a finite number  are self adjoint and of dimension $\leq n+1$; the remaining finite number of blocks are of dimension $\leq 2n$.
 
\quad All the (infinitely many) blocks are  chosen from a finite list of matrices $\mathcal M(\xi)$.
\smallskip

\noindent iv)  {\bf Smallness:} \ If $\e^3<r<\e/2$,  the perturbation $\tilde P$ is  small, more precisely 
we have   the bounds:
\begin{equation}\label{pertu}
\Vert X_{\tilde P}\Vert^\lambda_{s,r}\leq C (\e^{2q-1} r + \e^{2q+3} r^{-1}) \,, 
\end{equation}  where $C$ is independent of $r,\e$. 
\smallskip

\end{theorem} 
\begin{proof}
See \S  \ref{reteo}.
\end{proof}
\begin{remark}
At first inspection it may seem that the estimate on $X_P$ is {\em too small} to be possible. Indeed $P$ should contain  terms from $P^{(2q+4)}$, which should contribute to $X_P$ a term of order $\e^{2q+4}r^{-2}$. In fact for a {\em generic} choice of $S$ these terms are constant so they do not enter in the vector field.
\end{remark}
\begin{remark} The list of matrices $\mathcal M(\xi)$ is constructed   in Section \ref{lapro}, cf. Definition \ref{bfnu2}.

It  contains at most $2n \cdot (2q)^{m-1}!$ matrices distributed in at most $2n \cdot (2q)^{4nq}!$ orbits under  the group of permutations of the variables $\xi_i$. 

 In Example \ref{less} we exhibit $\mathcal M(\xi)$  in the case $q=1$, $n=2$.
\end{remark}
\subsubsection{Stable regions for the normal form }  An interesting issue is  to see if one can
  use   arithmetic constraints such as those of \cite{GYX},  to simplify those matrices  in $\mathcal M$ which are not self--adjoint.
   In Proposition \ref{blocchini} we show that, for $n\leq 2$ and all $q $ it is possible to choose the tangential sites so that the matrices reduce to only  $2\times 2$ matrices independently of   $m$.    This requires a notion of generic, of arithmetic (or probabilistic) nature,  which we call the {\em x--constraints} discussed in \S  \ref{arcon}.
    One may deduce the following very useful property, proved in  \S \ref{rero}:
\begin{proposition}\label{teo2}
Under the hypotheses of Proposition \ref{blocchini}, there exists an open region $O_{\ro }\subset \A_{\ro }$    where all the  non self--adjoint matrices  in $\mathcal M$ have  real and   distinct   eigenvalues.
\end{proposition}
\begin{proof}
The region is the one where all discriminants are strictly positive, we show in  \S \ref{rero} that it is a non empty open cone. 
\end{proof}
It then follows fairly easily from this result and Theorem \ref{teo1}:

\begin{corollary}\label{ficof}
There exists an algebraic hypersurface $\A$ such that on the open region $A_{\ro }\setminus \A$ there is an analytic  symplectic change of coordinates taking $\mathcal Q$  into a diagonal form with constant coefficients plus a form $\bar{\mathcal Q}$  with constant coefficients depending only on finitely many variables  $z_k,\bar z_k,\ k\in A$.

 The change of variables does not depend on $(x,y)$, it is linear in $w$ and analytic in $\xi$. The Hamiltonian is then
 \begin{equation}\label{hfin}  H_{\rm fin}= (\ome(\xi),y) + \sum_{k\in S^c} \bar\Omega_k |z_k|^2 +\bar{\mathcal Q}+ P(\xi,x,y,w),\end{equation}
 where 
 $$ \bar\Omega_k =\begin{cases} \tilde\Ome_k + \lambda_k(\xi)\,,\quad \forall k\in S^c\setminus A,\\
\tilde\Ome_k,\ k\in A\end{cases}$$  
i) The correction $\lambda_k(\xi)$   is chosen in a finite list, say
\begin{equation}\label{autov}
\lambda_k(\xi) \in \{ \lambda^{(1)}(\xi),\dots,  \lambda^{(K)}(\xi)\}\,,\quad K:= K(n,m),
\end{equation}
of different (real) analytic functions of $\xi$.

ii)\ The functions $\lambda^{(i)}(\xi)$ are homogeneous of degree $q$ in $\xi$. Let $\mathfrak A_\ro$ be a tubular neighborhood of $\mathcal A$ with radius of order $\ro$. For $\xi\in A_{\ro}\setminus \mathfrak A_\ro$ the $\lambda^{(i)}(\xi)$   satisfy  the bounds
\begin{equation}
\label{linh} |\lambda^{(i)}(\xi)|\leq C\e^{2q} \,,\quad c \e^{2q}\leq |\lambda^{(i)}(\xi)\pm \lambda^{(j)}(\xi)|\leq C\e^{2q}\,, \quad 
|\nabla_{\xi}\lambda^{(i)}(\xi)| \leq  C\e^{2q-2}. 
\end{equation} 
iii)\  For $\xi\in A_{\ro}\setminus \mathfrak{A}_{\varepsilon}$ item { iv)} of Theorem \ref{teo1} holds. 

iv)\ $\bar{\mathcal Q}$ is a quadratic Hamiltonian with constant coefficients in finitely many of the variables $z_k,\bar z_k,\ k\in S^c$. 

v)\ 
For $n=1,2$ and all $q$  it is possible to choose the tangential sites so that  $\bar{\mathcal Q} $ is formed by $2\times 2 $  blocks which (outside  the hypersurface $\A$) are semisimple with distinct eigenvalues.  The region in which these eigenvalues are real is open non empty and on this region the real eigenvalues are given by analytic functions $\mu_k(\xi)$ so that  we may  write 
\begin{equation}\label{hfin1}  H_{\rm fin}= (\ome(\xi),y) + \sum_{k\in S^c} \bar\Omega_k |z_k|^2 + P(\xi,x,y,w),\end{equation}
with $ \bar\Omega_k =  \tilde\Ome_k + \mu_k(\xi)\,,\quad \forall k\in A$.
\end{corollary}
\begin{proof}
See Section \ref{rero}.
\end{proof} 

\section{A normal form}\label{quattro}
{\em In this section we make a preliminary study of the Hamiltonian $H_{res}$ and introduce some simple constraints on the tangential sites $S$, this enables us to define our normal form.}
\begin{definition}\label{divor}
We call $x,y,w$ {\em dynamical variables}. We give degree $0$ to the angles $x$, $2$ to $y$ and $1$ to $w$.

\end{definition}

We use the degree only for handling dynamical variables, as follows. We develop  in  Taylor expansion, in particular since $y$ is small with respect to $\xi$ we develop $\sqrt{\xi_i+y_i}= \sqrt{\xi_i}(1+\frac{y_i}{2\xi_i}+\ldots)$ as a series in $\frac{y_i}{\xi_i}$ we then develop the entire Hamiltonians $H, H_{\rm Res}$ as series in $ y,w$.  
\smallskip

\begin{definition}[Normal form]\label{norfo}
We   separate $H_{\rm Res} +P^{2(q+2)}(u)=H= N+P$  where the  {\em normal form}    $N$     collects all the terms of $H_{\rm Res}$ (as series in $ y,w$) of degree $\leq 2$ in the variables   $y,w$. 

The series $P$ collects all terms of $P^{2(q+2)}(u)$ and  all the terms of $H_{\rm Res}$   of degree $> 2$ in the variables   $y,w$.  
\end{definition}

\smallskip

 It is easily seen that $H$, hence $P$, depend analytically on all the variables $\xi,y,x,w$ in the domain $A_{\ro }\times D(r,s)$.

In the new variables:
\begin{equation}\label{moma0}
M=\sum_i\xi_i v_i+\sum_i   y_i v_i+\sum_{k\in S^c}k|z_k|^2,\quad L=\sum_i\xi_i  +\sum_i   y_i  +\sum_{k\in S^c} |z_k|^2\,,\end{equation}
$$ \sum_{k\in \Z^n}|k|^2 u_k \bar u_k=K=(\ome_0,\xi+y)+\sum_{k\in S^c} |k|^2|z_k|^2\,,\quad \ome_0=(|v_1|^2,\dots,|v_m|^2). $$
\begin{remark}
\label{rinn}The terms $\sum_i\xi_i$, $\sum_i\xi_i v_i$ and $ \sum_i\xi_i |v_i|^2$ are constant and can be dropped, renormalizing the three quantities $M,L,K$ (momentum, mass and quadratic energy).
\end{remark}

We summarize  the commutation rules:
\begin{equation}\label{moma}
\{M,y_h\}=\{L,y_h\}=\{K,y_h\}=0,\  \{M,x_h\}=v_h x_h,\  \{L,x_h\}=x_h,\   \{K,x_h\}=|v_h|^2.
\end{equation}
 $$\{M,z_k\}=\ii k z_k,\ \{L,z_k\}=\ii  z_k,\ \{K,z_k\}=\ii |k|^2 z_k,$$ $$  \{M,\bar z_k\}=-\ii k \bar z_k,\ \{L,\bar z_k\}=-\ii\bar  z_k,\ \{K,\bar z_k\}=-\ii |k|^2\bar z_k.  $$  
We formalize the momentum and mass by  two linear maps.
\begin{equation}
\label{MMM}\pi:\Z^m\to {\rm Span}(S),\ \pi(e_i)=v_i,\quad \text{momentum};\quad \eta:\Z^m\to \Z ,\ \eta(e_i)=1\quad  \text{mass}.
\end{equation}
where $e_1,\dots,e_m$ are the elements of the standard basis of the lattice $\Z^m$. 
\begin{lemma}
Each  monomial $e^{\ii (\nu,x)}y^l z^\alpha\bar z^\beta$  is an eigenvector  of the linear operators\footnote{ Given a polynomial $P$, we denote by $ad(P)$ the linear operator that associates to each polynomial $A$ the polynomial $\{P,A\}.$}  $ad(M)$ and $ad(L)$ with eigenvalues (i.e. with momentum and mass) given by \begin{equation}\label{preei}
\pi(\nu)+\sum_{k\in S^c} (\alpha_k-\beta_k)k\,,\qquad \eta(\nu)+\sum_{k\in S^c} (\alpha_k-\beta_k)\,.
\end{equation}

\end{lemma}\begin{proof}
This follows computing $\{M,e^{\ii (\nu,x)}y^l z^\alpha\bar z^\beta\}$ and $\{L,e^{\ii (\nu,x)}y^l z^\alpha\bar z^\beta\}$  using  Formulas  \eqref{moma0} and the rules of Poisson bracket.
\end{proof}

\begin{remark}
A monomial Poisson commutes with $M$ and $L$ if and only if the momentum and mass are zero, that is $\pi(\nu)=-\sum_{k\in S^c} (\alpha_k-\beta_k)k\,,\  \eta(\nu)=-\sum_{k\in S^c} (\alpha_k-\beta_k)$.
\end{remark}
\smallskip

  \subsubsection{The normal form $N$\label{thenormalform}}  Our next task is to describe the Hamiltonian $N$ of Definition \ref{norfo}, provided that $S$ satisfies some basic constraints.  This is done in Proposition \ref{quadrat}.

$N$ is described in terms of a list of vectors, called {\em edges} since they will appear as edges of  a  graph describing the non--diagonal elements in $ad(N)$.
 \begin{definition}[edges] \label{edges}
  Consider the elements \begin{equation}
\label{glied}X_q:=\{\ell:=\sum_{j=1}^{2q} \pm  e_{i_j}=\sum_{i=1}^m \ell_ie_i,\  \quad \ell\neq 0,-2e_i\; \forall i\,,\quad \eta(\ell)\in \{0,-2\}\}.
\end{equation}  Notice the {\em mass constraint}   $\sum_i\ell_i=\eta(\ell)\in \{0,-2\}$.  We call all these elements respectively the {\em black, $\eta(\ell)=0$} and {\em red $\eta(\ell)=- 2$}    {\em edges}  and denote them by $X_q^0,X_q^{-2}$ respectively.

\end{definition}\begin{example}
For $q=1$ we have $e_i-e_j,-(e_i+e_j),\ i\neq j.   $  For $q=2$  we   have all the terms for $q=1$  and 
 $e_i- e_j-e_a-e_b,, \  2e_i-2e_j,\ -3e_i+e_j,\ i\neq j,a,b.  $  
\end{example}

 We start to impose a list of linear and quadratic inequalities on the choice of $S$ which will be justified in Proposition \ref{quadrat}.
 \begin{constraint}\label{co1}\begin{enumerate}
\item We assume   that  $\sum_{j=1}^m n_j v_j \neq 0$ for all  $n_i\in\Z,\,\sum_in_i=0,\ 1<\sum_i|n_i|\leq 2q+2$. 

\item $|\sum_in_iv_i|^2-\sum_in_i|v_i|^2\neq 0$ when $n_i\in\Z,\,\sum_in_i=1,\ 1<\sum_i|n_i|\leq 2q+1$.

\item We assume that  $\sum_{j=1}^{m}\ell_j v_j\neq 0 $, when $u:=\sum_{j=1}^{m}\ell_j e_j$ is either an edge or a sum or difference  of two distinct edges.

\item  $2\sum_{j=1}^{m}\ell_j |v_j|^2+|\sum_{j=1}^{m}\ell_j v_j|^2\neq 0$ for all edges  $\ell=\sum_{j=1}^{m}\ell_j e_j$ in $X_q^{-2}$.
\end{enumerate}
\end{constraint}
\begin{lemma}
Constraint \ref{co1}i) is an {\em integrability} constraint. Constraint \ref{co1}ii) is the {\em completeness} constraint. 
Constraint iii) means that an edge $\ell=\sum_{j=1}^{m}\ell_j e_j$ is determined by the associated vector $\pi(\ell)=\sum_{j=1}^{m}\ell_j v_j$. \end{lemma}
\begin{proof} i) The first statement follows  from Remark \ref{piro}. 

ii) Using   Proposition \ref{completa} it is enough to show that, under Constraint ii),  we cannot find  $2q+1$ elements $u_j=v_{i_j}$  for which there is a further vector $w$ in $\Z^m$ with $u_1,\ldots,u_{2q+1},w$ resonant.  Otherwise $w=\sum_in_iv_i$ is a linear combination with $\pm1$ coefficients of the $v_i$ hence it is a vector satisfying the hypotheses of item ii), but the quadratic condition in the same item implies that the list is non resonant.  iii) is clear.
\end{proof}  Constraint iv)  will be used in the next proposition, we shall see that it excludes quadratic terms of type $z_h^2$ or $\bar z_h^2$  in $H_{Res}$.

For $q=1$ this constraint means only that $ -2|v_i|^2- 2|v_j|^2+ |v_i+v_j|^2\neq 0, i\neq j$ and this just means $v_i\neq v_j$.
\begin{proposition}
\label{quadrat} Under all the previous constraints we have \begin{equation}\label{hama}N:= (\ome(\xi),y)+ \sum_{k\in S^c}  |k|^2 |z_k|^2 +   \mathcal Q(x,w) \end{equation}  where (cf. Formula \eqref{moma0}) the coefficient \begin{equation}
\label{laaa}\ome(\xi)= \ome_0+\nabla_\xi  A_{q+1}(\xi)- (q+1)^2 A_q(\xi)\underline 1,    
\end{equation}   does not depend on the dynamical variables.  
Here $\underline 1\in\N^m$ denotes the  vector with all coordinates equal to 1.

The term $\mathcal Q(x,w)$  is quadratic:
\begin{equation}\label{laquadra}
\mathcal Q= \sum_{\ell\in X_q^0} c(\ell)e^{ \ii(\ell,x)}\sum_{(h,k)\in \mathcal P_\ell }z_h\bar z_k
+ \sum_{\ell\in X_q^{-2}}c(\ell) \sum_{\{h,k\}\in \mathcal P_\ell }[ e^{ \ii(\ell,x)}z_h  z_k
 +  e^{ -\ii(\ell,x)} \bar z_h\bar z_k].
\end{equation} 
 where,
given  an edge  $\ell$, we set $\ell=\ell^+-\ell^- $ and define:
\begin{equation}\label{cl}
c_q(\ell)\equiv c(\ell):=\left\{\begin{array}{ll}\displaystyle (q+1)^2\xi^{\frac{\ell^++\ell^-}{2}}\sum_{\alpha\in\N^m\atop{|\alpha+\ell^+|_1=q}}\binom{q}{ \ell^++\alpha}\binom{q}{ \ell^-+\alpha }    \xi_i^\alpha  & \ell\in X_q^0\\\displaystyle
(q+1)q\xi^{\frac{\ell^++\ell^-}{2}}\sum_{\alpha\in\N^m\atop{|\alpha+\ell^+|_1=q-1}}\binom{q+1}{\ell^-+\alpha}\binom{q-1}{ \ell^++\alpha}  \xi_i^\alpha &\ell\in X_q^{-2}  
\\c_q( \ell)=c_q(-\ell)\qquad\quad  &\ell\in X_q^{ 2} 
 \end{array}\right.
\end{equation}  for the definition of $\mathcal P_\ell$ see Definition \ref{pl}.
\end{proposition}
\begin{example}
  Let us discuss $q=1$, {\em the cubic NLS}.  We have
    \begin{equation}\label{gliome}\ome_i(\xi):  = |v_i|^2 -2\xi_i,
\end{equation}
finally the quadratic form is \begin{equation}
\label{quafo}\mathcal Q(w)= 4\sum^*_{  1\leq i\neq j\leq m    \atop h,k \in S^c}\sqrt{\xi_{i}\xi_{j}}e^{\ii(x_{i}-x_{j})}z_{h}\bar z_{k } + 
\end{equation} $$  2\sum^{**}_{ 1\leq i< j\leq m   \atop h,k \in S^c }\sqrt{\xi_{i}\xi_{j}}e^{-\ii(x_{i}+x_{j})}z_{h} z_{k } + 
  2\sum^{**}_{ 1 \leq i<j\leq m    \atop h,k \in S^c }\sqrt{\xi_{i}\xi_{j}}e^{\ii(x_{i}+x_{j})}\bar z_{h}\bar  z_{k }. $$  Notice that in the sums  $  \sum^{**}$ each term appears twice. \vskip10pt  
 
Here $\sum^*$ denotes that $(h,k,v_i,v_j)$ satisfy: 
 $$ \{(h,k,v_i,v_j)\,|\,  {h+v_i= k+v_j},\  { |h|^2+|v_i|^2=| k|^2+|v_j|^2}\}.$$
  and $\sum^{**}$, that  $(h,v_i, k,v_j)$ satisfy:
   $$ \{(h,v_i,k,v_j)\,|\,  {h+k= v_i+v_j},\  { |h|^2+|k|^2=| v_i|^2+|v_j|^2}\}.$$

\end{example}
\begin{proof}[Proof of Proposition \ref{quadrat}] By definition  the normal form     $N$     collects all the terms of $H_{\rm Res}$ (as series in $y,w$) of degree $\leq 2$ in the variables   $y,w$. In turn $H_{\rm Res}$ is the sum of the quadratic term $K= \sum_k  |k|^2 |u_k|^2$ and of the terms of degree $2q+2$ in the original variables $u,\bar u$.

From Remark \ref{rinn}  the quadratic term
$K$ contributes to $N$ the terms $$(\ome_0,y)+\sum_{k\in S^c}|k|^2|z_k|^2 .$$
  The remaining terms  
 $  u_{k_1}\bar u_{k_2}u_{k_3}\bar u_{k_4}\ldots u_{k_{2q+1}}\bar u_{k_{2q+2}}$ satisfy the constraints  \begin{equation}\label{basco}
\sum_i(-1)^i k_i=0,\quad  \sum_i(-1)^i |k_i|^2=0  .
\end{equation}   These terms may contribute to terms of $N$ only if they are of total degree $\leq 2$ in $y,w$.
 
 We  analyze the three  possible cases, of degree $0,1,2$ in $w$. 
 \begin{itemize}
 \item  {\em degree 0}\quad If all the $k_i$ are in $S$ the  momentum $ \sum_i(-1)^i k_i$ is a linear expression $\sum_jm_jv_j$.  From momentum conservation  and  Constraint \ref{co1} i) we must have $m_j=0,\ \forall j$.  This implies that we can pair the even and odd $k's$  and, as shown in Proposition  \ref{compleH}, this gives a contribution $A_{q+1}(\xi+y)$. In this expression the terms of degree $\leq 2$ give a constant (which we ignore) and   the   term $(\nabla_\xi  A_{q+1}(\xi),y)$.\smallskip
 
\item {\em degree 1}\quad One and only one of the $k_i=k\in S^c$. Formula \eqref{basco} becomes
$$ k-\sum_i n_i v_i=0\,,\quad |k|^2-\sum_i n_i |v_i|^2=0$$ where $\sum_i n_i v_i$ satisfies the hypotheses of 
Constraint \ref{co1} ii). Thus these terms do not occur and $S$ is complete. \smallskip

\item {\em degree 2}\quad  Given  $h,k\in S^c$ we compute the coefficients of $z_h \bar z_k$ or $ z_h z_k$  or $\bar z_h \bar z_k$. These terms are obtained when all  but two of the $k_i$  are in $S$.  Each $k_i$ in $S$  contributes $\sqrt{\xi_i+y_i}e^{\pm x_i}$, giving a coefficient $\sqrt{\prod_{j=1}^{m } \xi_{j}^{\ell_j }}e^{\ii (\ell,x)}$, whenever 
\begin{equation}\label{fico}{(z_h\bar z_k)}:\quad\quad\;
\sum_{j=1}^{m}\ell_j v_j+h-k=0,\quad  \sum_{j=1}^{m}\ell_j |v_j|^2+|h|^2-|k|^2=0\ \quad \phantom{-}\ell\in X_q^0\quad
\end{equation}\begin{equation}\label{fico1} {( z_h z_k)}:\quad\quad
\sum_{j=1}^{m}\ell_j v_j+h+k=0,\  \quad \sum_{j=1}^{m}\ell_j |v_j|^2+|h|^2+|k|^2=0\quad \phantom{-} \ell\in X_q^{-2}
\end{equation}\begin{equation}\label{fico2}{( \bar z_h \bar z_k)}:\quad\quad\;
\sum_{j=1}^{m}\ell_j v_j-h-k=0,\ \quad  \sum_{j=1}^{m}\ell_j |v_j|^2-|h|^2-|k|^2=0\quad \phantom{-} \ell\in X_q^{ 2}
\end{equation} \end{itemize}
\begin{definition}\label{pl}
Given $\ell \in X_q^{(0)}$  denote by $\mathcal P_\ell$  the set of pairs $h,k$ satisfying   Formula \eqref{fico}.   If $\ell\in X_q^{(-2)}$  we denote by $\mathcal P_\ell$  the set of unordered pairs $\{h,k\}$ satisfying Formula  \eqref{fico1}.  \end{definition}
 \smallskip

Constraint \ref{co1} iii), where $u$ is the sum or difference  of two edges, implies that $h,k$ fix $\ell$ uniquely.
 In Formulas  \eqref{fico1},  \eqref{fico2}  we see that we cannot have $\ell=\mp 2e_i$ since the equations in these Formulas  have the only solution $h=k=v_i\in S$. This explains why in Definition \ref{edges} we have excluded $\pm 2e_i$ as edges.
Constraint \ref{co1} iv) implies that $h\neq k$ in Formulas  \eqref{fico1},  \eqref{fico2}.
By Constraint \ref{co1} iii) where $u$ is an edge, in \eqref{fico} $k=h$ implies $\ell=0$. This contributes a term $(q+1)^2A_{q }(\xi)\sum_{k\in S^c}|z_k|^2$.
 It is convenient to  write $$\sum_k (q+1)^2A_{q }(\xi)|z_k|^2=(q+1)^2A_{q }(\xi)(\sum_k |z_k|^2+\sum_iy_i)-   (q+1)^2A_{q }(\xi)(\sum_iy_i),$$ and notice that   $ (q+1)^2A_{q }(\xi)(\sum_k |z_k|^2+\sum_iy_i)$ is a mass term (hence a constant of motion for the whole Hamiltonian) and can be dropped from the Hamiltonian, so we change $N$ into
\begin{equation}
\label{laff}N=K+(\nabla_\xi  A_{q+1}(\xi)-(q+1)^2A_{q }(\xi)\underline 1,y)+\mathcal Q(x,w),\quad K=( \ome_0,y)+\sum_k |k|^2 |z_k|^2.
\end{equation}
Recall that $\underline 1\in\N^m$ denotes the  vector with all coordinates equal to 1.

Let us now compute $\mathcal Q(x,w)$,
 given  an edge  $\ell$ set $\ell=\ell^+-\ell^- $  formula \eqref{cl} comes from the expansion
 $$ c_q(\ell):=\left\{\begin{array}{ll}\displaystyle (q+1)^2\sum_{e_{h_1}-e_{k_1}+e_{h_2}+\ldots + e_{h_q}-e_{k_q}=\ell}\prod_{i=1}^q (\xi_{h_i}\xi_{k_i})^{1/2} & \ell\in X_q^0\\\displaystyle
(q+1)q\sum_{e_{h_1}-e_{k_1}+e_{h_2}+\ldots + e_{h_{q-1}}-e_{k_{q-1}}-e_{h_q}-e_{k_q}=\ell}\prod_{i=1}^q (\xi_{h_i}\xi_{k_i})^{1/2} &\ell\in X_q^{-2} \\c_q(-\ell)=c_q(\ell)
 \end{array}\right.$$

\end{proof}

It is interesting to notice  a point essential for the KAM algorithm,   since it gives a locally invertible change  of coordinates $\omega_i=\omega_i(\xi)$  expressing$$(\nabla_\xi  A_{q+1}(\xi)-(q+1)^2A_{q }(\xi)\underline 1,y)=\sum_{i=1}^m\omega_iy_i. $$

  \begin{proposition}\label{nond}
 For every $r$, the Hessian  of $A_r(\xi)$ is a non degenerate matrix as polynomial in $\xi$.\end{proposition}

\begin{proof}
Let $r=p^st$ with $p$ prime and $p\not | t$. It is well known and elementary that,   if $p$ does not divide   $\binom{r}{\ell}$, then  $p^s$ divides  the vector $\ell$.   
The coefficients of  $\partial_{\xi_1}\partial_{\xi_2} A_r(\xi)$  are 
$$\ell_1 \ell_2 \binom{r}{\ell }^2=r(r-1) \binom{r-2}{\ell_1-1,\ell_2-1,\dots,\ell_m}\binom{r}{\ell }. $$ We claim that they are divisible by $p^s r(r-1)$. Indeed  if $p$ does not divide  $\binom{r}{\ell}$   we have seen that $p^{2s}$ divides $\ell_1\ell_2$ while $p^{s+1}$ does not divide $r(r-1)$.
The coefficients of  $\partial^2_{\xi_1} A_r(\xi)$  are  $$\ell_1(\ell_1-1)  \binom{r}{\ell }=r(r-1) \binom{r-2}{\ell_1-2,\ell_2 ,\dots,\ell_m}.$$
 
 It follows that the Hessian is divisible by $r(r-1)$,  the off diagonal terms are divisible by   $p^{s }r(r-1)$  while the diagonal contains the term $r(r-1)$diag$(\xi_i^{r-2})$. Therefore,   once we divide  by  $r(r-1)$  we have a matrix which, modulo $p$, is diagonal with non zero entries. 
 \end{proof}

 From Proposition \ref{nond} we have the coordinate change:
 \begin{corollary}\label{diffeo}
The map $\xi\to \nabla_\xi  A_{q+1}(\xi)-(q+1)^2A_{q }(\xi)\underline 1$ is a local diffeomorphism outside a real algebraic hypersurface.
\end{corollary}
 \begin{proof}
The Jacobian of this map is a matrix with entries polynomials in $\xi$ with integer coefficients. Reasoning as in Proposition \ref{nond} we see that this Jacobian matrix is  of the form $J=q(q+1)A-(q+1)^2B$ where $q(q+1)A$ is the Hessian of $A_{q+1}(\xi)$ while $B$  has as entries    the derivatives of $A_{q }(\xi)$, therefore $B=qC$ has   coefficients divisible by $q$. Thus when we divide $J$ by $q(q+1)$  we have a matrix $A-(q+1)C$ with entries polynomials with integer coefficients. Modulo a prime $p$ dividing $q+1$ we have only the contribution from $A$ which gives the diagonal matrix with non zero entries discussed in the proof of Proposition \ref{nond}. It follows that the determinant of the Jacobian $J$  is a non--zero polynomial.
\end{proof}
 \subsubsection{The perturbation $P$\label{lestpe}}
  
Remark that  $P(x,y,w)$ is {\em regular} in the sense of \S\ref{cts}. Indeed in    \eqref{Ham} all the terms of degree $>2$ are regular and the Birkhoff normal form and elliptic-action angle variables preserve this property by the chain rule. 

We say that $P$  is {\em of order} $\e^a r^b$ for some integers $a,b$ if its norm is smaller than $C \e^a r^b$ for some $\e,r$ independent constant $C$. This implies, since $P$ is regular, that  $X_P$ is of order $\e^a r^{b-2}$ (i.e. $\|X_P\|^\lambda_{s,r}$ is bounded by   $ C' \e^a r^{b-2}$).

According to Definition \ref{norfo}, $P$ comes from two types of terms. In a term-- denote it by $P^{(3)}$-- we collect all the terms of degree $2i+j>2$ coming from the resonant  terms  $\prod_{i=1}^{q+1}u_{k_i}\bar u_{h_i}$ (of $H_{\rm Res}$). 
 In $P^{(2q+4)}$ we collect all the terms coming from products $\prod_{i=1}^{d}u_{k_i}\bar u_{h_i}$, with $d\geq q+2$ of $P^{2(q+2)}$. 
 
 Recall that $u_{v_i}= \sqrt{\xi_i+y_i}e^{\ii x_i}=\sqrt{\xi_i}(1+\frac{y_i}{2\xi_i}+\dots)e^{\ii x_i}$  is of order $\e$, while $z_k$ is of order $r$.  Then the dominant term in $P^{(3)}$ is given by the dominant terms of the monomials of degree $2i+j=3$.  Hence all the other $2q-1$ variables are tangential and computed at $y=0$.  The order is hence $r^3 \e^{2q-1}$ 
  
   The order of $P^{(2q+4)}$  is clearly $\e^{2q+4}$ and comes from a term depending only on $\xi$ and possibly on $x$. However, by hypothesis  all its Fourier coefficients must respect momentum conservation. Reasoning as in Proposition \ref{quadrat}, by Constraint \ref{co1} iii) such term is necessarily constant in the dynamical variables, hence we drop it, since it does not contribute to the vector field.
   Hence the order is either $\e^{2q+6}$  or $\e^{2q+3}r$.
    For $r>\e^3$, the leading term is $\e^{2q+3}r$. Passing to the vector field, under our constraints:
   \begin{proposition}\label{gliord} The order of $X_{P^{(3)}}$ is   $r \e^{2q-1}.$
   
     The order of $X_{P^{(2q+4)}}$  is $\e^{2q+3}r^{-1}$.
     \end{proposition}
 \begin{remark}
Is is possible to improve the estimate $r>\e^3$  to  $r> \e^{2q+1}$ by noticing that with one step of  Birkhoff  normal form one can  remove all the non--resonant terms in $H$ of degree $< 4q+2$, then we repeat the analysis as above. This procedure only changes $\ome$ and $\mathcal Q$ in a trivial way. 
\end{remark}

 \section{Matrix description of $ad(N)$}  \label{cinque}

 \subsection{The spaces  $V^{i,j}$ and $F^{0,1}$}  \begin{definition}
We denote by $V^{i,j}$ the space of functions spanned by elements of total degree $i$ in $y$ and $j$ in $w$  and $V^h=\sum_{i+j=h} V^{i,j}, V^\infty =\sum_{i,j } V^{i,j}.$
\end{definition}
  The  space     $V^{0,1}$ has a basis over $\C$  given by the elements $\{ e^{\ii \sum_j\nu_jx_j}z_k ,\quad   e^{-\ii \sum_j\nu_jx_j}\bar z_k\},$ where ${\nu\in \Z^m\,,k \in S^c}$.

  It can be also viewed as a {\em free module}\footnote{ A free module is the vector space of linear combinations of a basis with coefficients in an algebra.} with basis the elements $z_k,\bar z_k$,  over the algebra $\mathcal F$ of finite Fourier series. This  is useful since, by Formula \eqref{laquadra}, the function $\mathcal Q$ commutes with  $\mathcal F$  and thus it can be described by a matrix, with entries in  $\mathcal F$, in the basis $z_k,\bar z_k$. 
  \smallskip
  
  We now impose the restrictions of momentum and mass conservation.  
   Denote by $F^{0,1}$  the subspace of   $V^{0,1}$ commuting with momentum.  By formula \eqref{moma} we see that $F^{0,1}$  has as basis,  which we call {\em  frequency basis}, the set $F_B$ of elements (cf. \eqref{MMM})
 \begin{equation}\label{llab}
 F_B=\{ e^{\ii \sum_j\nu_jx_j}z_k ,\quad   e^{-\ii \sum_j\nu_jx_j}\bar z_k\};\quad  \sum_j\nu_jv_j+k=\pi(\nu)+k=0\,,\quad k  \in S^c.\end{equation}  
An element of $F_B$  is completely determined by the value of $\nu$  and the fact that the $z$ variable may or may not be conjugated, thus sometimes we refer to $e^{\ii \sum_j\nu_jx_j}z_{-\pi(\nu)}  $ as $(\nu,+)$ and to $e^{-\ii \sum_j\nu_jx_j}\bar z_{-\pi(\nu)}$ as $(\nu,-)$.  By construction $\nu\in \Z^m_c$ where
 \begin{equation}\label{zmc}
\Z^m_c:=\{\mu\in\Z^m\,|\, -\pi(\mu)\in S^c\}\,,
\end{equation}
and we may identify $F_B$ with $\Z^m_c\times\Z/(2)$.

We can further decompose the space  $F^{0,1}=\oplus F^{0,1}_\ell$ by the eigenspaces of the mass operator $ad(L)$. Notice that   the {\em mass} of $e^{\ii \sum_j\nu_jx_j}z_k$ is $\ell=\sum_i\nu_i+1$, thus on the subspace commuting with $L$ we have  $-1=\sum_i\nu_i$  for $(\nu,\pm)$. 

 \subsubsection{The action of $ad(N)$\label{theaction}}  In order to study the  action of $ad(N)$ on the two spaces $F^{0,1}$ and
     $V^{0,1}$ we
 notice that:
\begin{remark}
\begin{enumerate}
\item The terms $ \sum_k|k|^2   |z_k|^2 +    \mathcal Q(x,w)$ Poisson commute with the algebra $\mathcal F$ of Fourier series in $x$.  
\item $ad(\sum_k|k|^2  |z_k|^2) $ is a diagonal matrix in the geometric basis  $z_k,\bar z_k$.

\item $ad((\ome(\xi),y)+ \sum_k |k|^2  |z_k|^2)$  is a diagonal matrix in the frequency basis $F_B$.\end{enumerate}
\end{remark}
 Hence, in order to describe $ad(N)$, we only need to understand the action of $ad(\mathcal Q)$   on the two spaces $F^{0,1}$ and
     $V^{0,1}$.
We then have two matrix   descriptions. One, denoted $\ii Q(x)$,  with respect to the basis $w$ and with finite Fourier series as entries  the other $\ii  Q$ with respect to the frequency basis  and with constant coefficients. Of course  each  can be deduced from the other in a simple way. 
\section{Graph representation}\label{Gra}
  A matrix $(a_{i,j})$  over a basis indexed by a set $I$ is conveniently  displayed graphically by a graph with vertices the elements of $I$.  Two  vertices $i,j$ are joined by an  edge if $a_{i,j}\neq 0$, in this case it is also convenient to orient the edge and mark it with the entry of the matrix as
$$\xymatrix{i  \ar@{<-}[r]^{a_{i,j}} &j}.$$  The usefulness of this construction lies in the fact that the connected components of the graph  correspond to the diagonal indecomposable blocks into which the matrix can be decomposed.

We thus associate to   the matrices $Q(x),Q$    two graphs    $
\Gamati,{\Lambda_S}$   encoding the information of the non--zero off diagonal entries in the respective bases.  \begin{definition}\label{iduegrafi}   The  graph $\Gamati$ has as vertices the geometric basis, i.e.  the variables $z_k,\bar z_h$, and edges corresponding to the nonzero entries of the matrix $Q(x)$  in this basis. 

The  graph ${\Lambda_S}$ has as vertices the elements of $F_B=\Z^m_c\times\Z/(2)$, and edges corresponding to the nonzero entries of the matrix $Q$  in this frequency basis.
\end{definition} 
\begin{remark}\label{ilgrang}
We could also introduce the graph describing the matrix $Q$  on the entire space  $V^{0,1}$  in its corresponding basis.  It is just obtained as infinitely many copies of   $\Lambda_S$ (for all values of momentum) by multiplying with all possible exponentials $e^{\ii\sum_j\nu_jx_j}$.  
\end{remark}

\subsubsection{The rules\label{Trules}}The rules to construct the graph are the Formulas \eqref{fico}, \eqref{fico1}, \eqref{fico2}. 

To be explicit in our case, if $a_{i,j}\neq 0$ also $a_{j,i}\neq 0$ so we should connect each connected pair of vertices  with two edges.  In fact it is clear that both edges and their markings are uniquely determined by a single edge $\ell$. We  discuss the simple choices that we make in order to be explicit.

In case of an ordered pair $(h,k)$ satisfying \eqref{fico} for the edge $\ell\in X_q^0$,  we display:
$$ \xy (0,0)*+{z_h}="a", (20,0)*+{z_k}="b" \ar @{<-}^{ c(\ell) e^{\ii \ell\cdot x}} "a";"b"< 2pt>
\ar @{->}_{ c(\ell) e^{-\ii \ell\cdot x}}"a";"b" < -2pt> \endxy \;\equiv\;  \xy (0,0)*+{z_h}="a", (20,0)*+{z_k}="b" \ar @{<-}^{ \ell} "a";"b"< 2pt>
\ar @{->}_{ -\ell}"a";"b" < -2pt> \endxy\,,\quad  \xy (0,0)*+{\bar z_h}="a", (20,0)*+{\bar z_k}="b" \ar @{<-}^{ - c(\ell) e^{-\ii \ell\cdot x}} "a";"b"< 2pt>
\ar @{->}_{- c(\ell) e^{\ii \ell\cdot x}}"a";"b" < -2pt> \endxy \;\equiv\; \xy (0,0)*+{\bar z_h}="a", (20,0)*+{\bar z_k}="b" \ar @{<-}^{ -\ell} "a";"b"< 2pt>
\ar @{->}_{\ell}"a";"b" < -2pt> \endxy\,. $$ 
Of course, if $\ell\in X_q^0$ then also $-\ell \in X_q^0$. We choose one representative of the pair $\ell,-\ell$ (for instance using lexicographic ordering) and drop one of the arrows.

Similarly for $(a,\sigma)\,,(b,\rho)\in \Z^m\times\Z/(2)$ such that $b= a+\ell$, 
$h=-\pi(b)\,,k=-\pi(a)\,,$ $\sigma=\rho$ (and $h,k$ as above) we have
 $$ \xy (0,0)*+{(b,+)}="a", (20,0)*+{(a,+)}="b" \ar @{<-}^{ c(\ell) } "a";"b"< 2pt>
\ar @{->}_{ c(\ell) }"a";"b" < -2pt> \endxy \equiv \xymatrix{ (b,+) \ar@{<-}[r]^{\ell}& (a,+)  },\quad  \xy (0,0)*+{(b,-)}="a", (20,0)*+{(a,-)}="b" \ar @{<-}^{ - c(\ell) } "a";"b"< 2pt>
\ar @{->}_{- c(\ell) }"a";"b" < -2pt> \endxy \equiv \xymatrix{(b,-) \ar@{<-}[r]^{\ell}& (a,-) }. $$
Notice that our convention in describing the basis $F_B$, implies that the arrow joining $(a,-)$ to $(b,-)$ has the opposite direction to that joining $\bar z_{h}$ to $\bar z_k$.

In case of an unordered pair $(h,k)$ satisfying \eqref{fico1} for the edge $\ell\in X_q^{-2}$  we display:
$$ \xy (0,0)*+{z_h}="a", (20,0)*+{\bar z_k}="b" \ar @{<-}^{ -c(\ell) e^{\ii \ell\cdot x}} "a";"b"< 2pt>
\ar @{->}_{ c(\ell) e^{-\ii \ell\cdot x}}"a";"b" < -2pt> \endxy \;\equiv\; \xymatrix{ z_h \ar@{-}[r]^{\ell}& \bar z_k  }\,,\quad  \xy (0,0)*+{(a,+)}="a", (20,0)*+{(b,-)}="b" \ar @{<-}^{ -c(\ell) } "a";"b"< 2pt>
\ar @{->}_{ c(\ell) }"a";"b" < -2pt> \endxy \;\equiv\; \xymatrix{ (a,+) \ar@{-}[r]^{\ell}& (b,-)  }, $$
where $-\pi(a)=h$, $-\pi(b)=k$ and $a+b=\ell$.
 \begin{remark}\label{Kcos}
Since $K$ commutes with $\mathcal Q$, a block for  $\mathcal Q$ is  contained in an eigenspace of $K$ with fixed eigenvalue $\kappa$. We have 
\begin{equation} 
\label{eiK}\{K, e ^{\ii \mu.x}z_k\}\!=\! \ii (\sum_i\mu_i|v_i|^2\!+|k|^2)e ^{\ii \mu.x}z_k,    \{K, e ^{-\ii \mu.x}\bar z_k\}\!= -\ii (\sum_i\mu_i|v_i|^2\!+|k|^2) e ^{-\ii \mu.x}\bar  z_k .\end{equation}
 The eigenspace of $K$ where $\sum_i\mu_i|v_i|^2\!+|k|^2=\kappa$ in general is an infinite block which has to be further reduced, by the explicit description of $\mathcal Q$, into the direct sum of infinitely many finite blocks.
\end{remark}
While the graph $\Gamati$  appears naturally in the description of $Q(x)$, we find it convenient to forget the conjugate variables getting a purely {\em geometric graph}  $\Gama$  with vertices in $S^c$ and {\em colored edges}.

\begin{definition}\label{geogra}
Two points $h,k\in S^c$ are connected by a {\em black edge} if $z_h,z_k$ are connected in  $\Gamati$, the edge has the same orientation  as the one joining $z_h,z_k$ and mark the edge by $-\pi(\ell)$.    Similarly, $h,k\in S^c$ are connected by a {\em red edge} if $z_h,\bar z_k$ are connected in  $\Gamati$,  the  marking is again $-\pi(\ell)$.
\end{definition} 
 
\begin{example}[$q=1$]\label{EGG}
Suppose we have in $\Gamati$ the  connected component containing $z_{k_1}$: 
\begin{equation}
\label{tildegamma}\tilde A_{k_1,+}=\xymatrix{ & \ar@{-}[d]^{-e_2-e_1}\bar z_{k_4} &&\\ z_{k_1}\ar@{->}[r]^{e_3-e_1\quad}& z_{k_2}  \ar@{->}[r]^{e_3-e_2\quad}&  z_{k_3}}\, {\rm where} \quad \left\{\begin{array}{l} k_2- k_1+v_3-v_1 =0   \\  |k_2|^2-| k_1|^2+|v_3|^2-|v_1|^2 =0 \\ k_3- k_2+v_3-v_2 =0   \\  |k_3|^2-| k_2|^2+|v_3|^2-|v_2|^2 =0\\  k_4+ k_2-v_2-v_1 =0   \\  |k_4|^2+| k_2|^2-|v_2|^2-|v_1|^2 =0\end{array}\right.\end{equation} 
with $k_1\neq k_2\neq k_3\neq k_4\in S^c$. Then the block of the matrix   $Q(x)$  corresponding to this  graph is
   $$4\begin{pmatrix}
 0 & \sqrt{\xi_1 \xi_3}e^{-\ii(x_3-x_1)} & 0 &0 \\ \sqrt{\xi_1 \xi_3}e^{\ii(x_3-x_1)} &0& \sqrt{\xi_2\xi_3} e^{-\ii(x_3-x_2)} & -\sqrt{\xi_1\xi_2}e^{-\ii(x_1+x_2)} \\ 0 &\sqrt{\xi_2\xi_3} e^{\ii(x_3-x_2)}  &0 & 0 \\ 0 & \sqrt{\xi_1\xi_2}e^{\ii(x_1+x_2)}&0 &0 
\end{pmatrix} $$ 
where we have arbitrarily chosen the ordering $z_{k_1},z_{k_2}, z_{k_3}, \bar z_{k_4}$.

By the reality condition we also have the connected component:$$ \tilde A_{k_1,-}=\xymatrix{ & \ar@{-}[d]^{-e_2-e_1} z_{k_4} &&\\ \bar z_{k_1}\ar@{<-}[r]^{e_3-e_1\quad}& \bar z_{k_2}  \ar@{<-}[r]^{e_3-e_2\quad}& \bar z_{k_3}}, $$ which we think of as a {\em conjugated block}.

In conclusion the contribution of these two components to  $\mathcal Q$ is:
\begin{equation}\label{hamex}
4\sqrt{\xi_1\xi_3} e^{\ii(x_1-x_3)}z_{k_1}\bar z_{k_2}+  4\sqrt{\xi_2\xi_3} e^{\ii(x_2-x_3)} z_{k_2} \bar z_{k_3}+ 4\sqrt{\xi_1\xi_2} e^{-\ii(x_1+x_2)} z_{k_2}   z_{k_4}+ $$ $$ 4\sqrt{\xi_1\xi_3} e^{-\ii(x_1-x_3)}\bar z_{k_1} z_{k_2}+  4\sqrt{\xi_2\xi_3} e^{-\ii(x_2-x_3)} \bar z_{k_2} z_{k_3}+ 4\sqrt{\xi_1\xi_2} e^{ \ii(x_1+x_2)} \bar z_{k_2} \bar  z_{k_4} 
\end{equation} 
Let now $a\in \Z^m$ be any vector such that $ -\pi(a)=k_1$, then  the graph $\Lambda_S$ has the two components

\begin{equation}\label{EGA}{\mathcal A}_{(a,+)}= \xymatrix{ & \ar@{-}[d]^{-e_2-e_1}(-e_2-e_3-a,-)&&\\ (a,+)\ar@{->}[r]^{e_3-e_1\quad}& (a+ e_3- e_1,+)  \ar@{->}[r]^{e_3-e_2}&  (a  -e_1-e_2+2e_3,+)  }, \end{equation} $${{\mathcal A}_{(a,-)}}= \xymatrix{ & \ar@{-}[d]^{-e_2-e_1}(-e_2-e_3-a,+)&&\\ (a,-)\ar@{->}[r]^{e_3-e_1\quad}& (a+ e_3- e_1,-)  \ar@{->}[r]^{e_3-e_2\quad}&  (a  -e_1-e_2+2e_3,-)  }. 
$$

Finally the  geometric graph  corresponding to \eqref{tildegamma} is \footnote{we shall denote by $ \xymatrix{ & \ar@{=}[r]&}$ a red edge.}  \begin{equation}
\label{EGA1}{A_{k_1}}= \xymatrix{ & \ar@{=}[d]^{v_2+v_1}k_4&&\\ k_1\ar@{->}[r]^{ v_1-v_3\quad}& k_2  \ar@{->}[r]^{v_2-v_3\quad}&  k_3}, 
\end{equation} 

The vectors appearing as vertices must satisfy the  linear and quadratic constraints  appearing in \eqref{tildegamma}. Notice that we can deduce the list of equations associated to a geometric graph by looking at its vertices, indeed if $k_i,k_j$ are connected by an edge then this arises from an $\ell$ (see Formulas \eqref{fico}--\eqref{fico2}) which is uniquely determined.\medskip
  \end{example}
\begin{remark}\label{pato}
All the connected components which we have described in this simple example  are isomorphic (as marked graphs), this is a fundamental issue since it enables us to define the change of variables  which reduces the Hamiltonian to constant coefficients.  The geometric graph  probably gives the clearer picture since it encodes in the simplest way the list of equations which the $k_i$ must fulfill. 

 It is useful to notice that, as soon as $m>n$, corresponding to the components  $ \tilde A_{k_1,\pm}\in \Gamati,$ there are infinitely many components $\mathcal A_{(a,\pm)} \in\Lambda_S$, one for  each of the the points $(a,\pm)$ such that $-\pi(a)=k$. These points are infinitely many if $m>n$, since the vectors $v_i$ cannot be independent and so $\pi$ must have a kernel $\ker(\pi)$, these components are obtained from a given one by translations by elements of $\ker(\pi)$.
\end{remark}

\subsection{Geometric graph $\GamaG$}\label{seiuno} We define a graph on $\R^n$ using the formulas \eqref{fico} and \eqref{fico1}.

 \begin{definition}\label{lasfe}  An edge $\ell\in X^{-2}_q$ defines a sphere $S_\ell$ through the relation:

\begin{equation}\label{sfera}|x|^2+(x,\sum_i\ell_i v_i)=-\frac{1}{2}( | \sum_i\ell_i v_i|^2+\sum_i \ell_i |v_i|^2 ),
\end{equation}

An edge $\ell\in X^0_q$ defines a plane $H_{\ell}$ through the relation \begin{equation}
\label{iperp} (x,\sum_i\ell_i v_i)=\frac{1}{2}( | \sum_i\ell_i v_i|^2+\sum_i\ell_i |v_i|^2).
\end{equation}
 \end{definition}
\begin{figure}[!ht]
\centering
\begin{minipage}[b]{11cm}
\centering
{
\psfrag{a}{$v_j$}
\psfrag{b}{$v_i$}
\psfrag{c}{$ h_2$}
\psfrag{d}{$ k_2$}
\psfrag{e}{$ h_1$}
\psfrag{f}{$ k_1$}
\psfrag{H}{$H_{\ell}$}
\psfrag{S}{$S_{\ell}$}
\psfrag{m}{$ v_i-v_j$}
\psfrag{l}{$ v_i+v_j$}
\includegraphics[width=11cm]{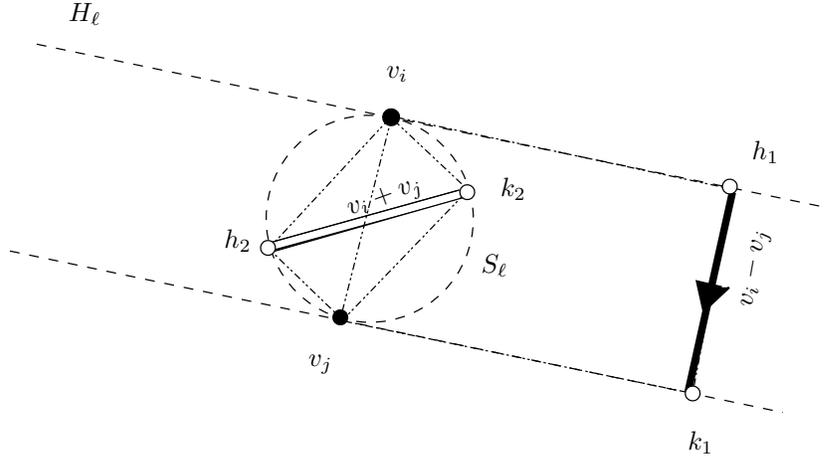}
}
\caption{\footnotesize{The plane $H_{\ell}$ with $\ell= e_j-e_i$ and the sphere $S_{\ell}$  with $\ell= -e_i-e_j$ . The points $h_1,k_1,v_j,v_i$ form  the vertices of a rectangle. Same for  the points $h_2 ,v_i,k_2,v_j$}}\label{fig1}
\end{minipage}
\end{figure}

\begin{definition}\label{ilggg}Each $x\in S_\ell$ is joined by a {\em red} unoriented edge to $-x-\sum_i\ell_iv_i\in S_\ell$.
Each $x\in H_\ell$ is joined by a {\em black} oriented edge to $x-\sum_i\ell_iv_i\in H_{-\ell}$.
We construct the graph  $\GamaG$  with vertices all the points of $\mathbb R^n$ and  edges the black and red edges described.
\end{definition}
It is convenient to mark each edge of the graph with the element $-\pi(\ell)$ from which it comes from.  Notice that Constraint  \ref{co1}     implies that the edge $\ell$ is uniquely determined by the vector $\pi(\ell)$.
\begin{remark}
The points in $H_\ell$ are the initial vertices of an edge    $\ell\in X_q^{(0)}$ ending in the parallel hyperplane $ H_\ell-\sum_i\ell_iv_i= H_{-\ell}$.

The points in $S_\ell$ are the initial vertices of an edge of type $\ell\in X_q^{(2)}$ which is a diameter of the sphere.

\end{remark}

\begin{remark}\label{spco} The completeness Constraint \ref{co1} {\it ii)} on $S$ implies that
the vectors $v_1,\ldots,v_m$ form a component of the graph  $\GamaG$. In this component every two vertices are joined by a red and by a black edge marked respectively $v_i+v_j$ and $v_i-v_j$.

This is independent of the choices of $q,m,n$.
\end{remark}
\begin{definition}
The component $v_1,\ldots,v_m$ is called the {\em special component} of the graph  $\GamaG$. 
\end{definition}
 We want to understand the other connected components of the graph  $\GamaG$, which contain a purely geometric description of the possible components of $\Gama$. Naturally most of the components of the graph  $\GamaG$ will not be formed by integral vectors.
 \smallskip
 
 \subsubsection{A rough estimate}\quad Before we start a fine analysis we may recall first a simple result, which is proved in \cite{MP}:
 \begin{lemma}\label{banal}
 i) The number of  vertices which may be adjacent to a red edge is finite and bounded  by some constant $N(q,n, \max_{i=1}^m(|v_i|)$.
 
 ii) For generic choices of $S$  each path in  $P\in \GamaG$  containing only black edges  cannot contain  two distinct   edges marked with the same  $\ell\in X_q^{(0)}$.
\end{lemma}
\begin{proof}
i) Each $\ell \in X_q^{(-2)}$ defines a sphere and on each sphere there are only a finite number of points, at most $R_\ell^{n-1}$, where $R_\ell$ is the radius of the sphere . If a vertex $k$ is adjacent to $\ell$ by definition $k\in S_\ell$; the statement follows.

ii) In a minimal counterexample we suppose that an edge $\ell$ appears twice and all others appear at most once. Let $x_1,x_2$ be the two distinct vertices out of which $\ell$ exits and consider a path $P(x_1,x_2)$ joining them. By applying the linear equations in \eqref{fico}  to the vertices in $P(x_1,x_2)$ one may conclude that
$x_2 =  x_1 + \sum n_i v_i$, where the $n_i$ are integers which depend on $P(x_1,x_2)$. Since  $P(x_1,x_2)$ does not contain any other edge more then once, then $|n_i|\leq (2q)^{m+1}$.  We now write the condition that $x_1,x_2\in H_\ell$, using \eqref{iperp}:
$$ 2(x_1,\sum_i \ell_i v_i)=  \sum_i \ell_i |v_i|^2 +|\sum_i \ell_i v_i|^2\,,\quad 2( x_1+ \sum n_i v_i,\sum_i \ell_i v_i)=  \sum_i \ell_i |v_i|^2 +|\sum_i \ell_i v_i|^2$$
and this may be excluded by requiring $$(\sum n_i v_i,\sum_i \ell_i v_i)\neq  0,\quad \forall \sum n_i v_i\neq 0, |n_i|\leq (2q)^{m+1}
$$ which is a generiticity condition.
\end{proof}
This Lemma immediately implies:
\begin{proposition}\label{rozzo}
For generic choices of $S$ there is a uniform bound on the number of vertices in each connected component of $\GamaG$.
\end{proposition}
\begin{proof}
By Lemma \ref{banal} {\it  ii)} a path made of black edges has a finite (and uniformly bounded) length since each edge may appear at most once.
So the connected components containing black edges have a uniform bound on the number of vertices.
By Lemma \ref{banal} {\it  i)} we may  form a finite block with all the points adjacent to a red edge and all the vertices connected to them. Indeed if a vertex is connected to a vertex in a sphere by a path made of black edges then this path has finite length. 
\end{proof}

This bound is clearly very rough, however to prove  optimal bounds one must work much harder and this we shall do in the rest of the paper.

\subsection{Geometric results\label{georesu}} \label{seidue}
{\bf Our goal}\quad We want to decompose the graph $\GamaG$ into simple blocks, as for instance that of \eqref{EGA1}.  The fact that this may be possible with blocks of at most $n+1$ vertices is suggested by a simple count of parameters, indeed one sees in \eqref{tildegamma} that a tree with $e$ edges occurs when the $e+1$ vertices  (corresponding to $(e+1)n$ incognit\ae) satisfy a set of $ e(n+1)  $ equations.
  
 Indeed, this bound can be achieved for all blocks consisting only of black edges  under all geometric   constraints.

The core of the paper is to prove Theorems \ref{sunto} and \ref{oneone} by imposing a finite number of non-zero polynomial constraints on $S$, 
 Constraints \ref{co1}  are the beginning of this analysis. The  full list  of the explicit geometric constraints is summarized in   Definition \ref{fincon}. 
 \begin{theorem}\label{sunto} For a generic choice of the $v_i$ as in  \ref{fincon} we have:
\begin{enumerate} \item All connected components of the graph $\GamaG$    consisting only of black edges    have at most $n+1$ vertices.

\item There are finitely many components in $\Gama$  containing red edges, each can contain at most $2n$  vertices.
\item The connected components of  $\GamaG$  consisting only of black edges  are divided into a finite number of families.
 \item  Each family in $\GamaG$  is formed by all the graphs obtained from a given one $G$, with $k+1$ affinely independent vertices, under translation by  all the points of  the  $n-k$  dimensional subspace  orthogonal to the span of $G$, minus a union $A$ of a finite number of  lower dimensional affine subspaces.
 
 The translates  $G+a,\ a\in A$  are contained in strictly larger  connected components of $\GamaG$.\smallskip

Moreover
\smallskip

\item All connected components of the graph ${\Lambda_S}$   have
at most $2n$ vertices. The vertices with the same sign are affinely independent.   There may be complicated dependencies between vertices with different signs.

\end{enumerate}
\end{theorem}
   \begin{proof}
See \S \ref{suntoS}.
\end{proof}

The next result relates the three graphs  ${\Lambda_S},\Gamati,\Gama$.  Take a frequency $\mu\in\Z^m_c$,  let $\A_{(\mu,+)}$ be the   component in ${\Lambda_S}$ of $(\mu,+)$ and set $k=-\pi(\mu)$. From Formula \eqref{llab} the associated component     in $\Gamati$ is the one of the element $z_k$ and will be denoted by  $\tilde A_{(k,+)}$. Finally in the geometric graph $ \Gama$ we have the  component of  the element $ k$ which will be denoted by  $  A_k$.  Similar  description for  $(\mu,-)$.   \begin{theorem}\label{oneone} For a generic choice of $S$ the map $-\pi$ establishes a graph isomorphism between $\A_{(\mu,\pm)}$ and $\tilde A_{(k,\pm)}$, which is also mapped  isomorphically to $A_k$. All these maps are compatible with the markings.

 \end{theorem}   \begin{proof}
See \S \ref{oneoneS}.
\end{proof}
 In particular the space  spanned by all transforms of  $e^{\ii \mu.x}z_k$   applying the operator $ ad(N)$ has a basis extracted from the frequency basis in  correspondence, under $-\pi$,   with  the vertices of  $A_k$. 

    All other   connected components of ${\Lambda_S}$ lying over $A_k$     are obtained from $\A_{(\mu,\pm)}$  by adding all the elements $ \nu $ such that $\pi(\nu)=0$.

 \smallskip

\section{A formalization of the graphs.}  \label{sette}
 The rules \eqref{fico},\ \eqref{fico1}, \eqref{fico2}  determine the edges of the three  graphs   ${\Lambda_S},\Gamati,\Gama$ that we have introduced in \S \ref{georesu}.  These rules consist of a linear and a quadratic constraint which encode respectively  the conservation of momentum and of quadratic energy (i.e. the fact that we have kept only resonant terms).  We want to see first  that, if we implement only the linear rules,  the graphs we construct are contained in some Cayley graphs (see the Appendix for the relevant definitions). Next  we show that the quadratic rules   select, inside the large Cayley graphs, the graphs of our interest.
 \subsection{The linear momentum constraints}
  Denote by  ${\Z^m}:=\{\sum_{i=1}^ma_ie_i,\  a_i\in\Z\}$  the lattice  with basis the elements $e_i$.

We consider  the group $G:={\Z^m}\rtimes\Z/(2)$\footnote{The notation $\rtimes$ stands for {\em semidirect product}} of couples $(a,\sigma)$ with $a\in \Z^m$ $\sigma=\pm$. The group structure is given by the rules
$$ (a,+)(b,+)=(a+b,+) \,,\quad (a,-) (b,+)= (a-b,-)\,,$$ $$(a,+)(b,-)= (a+b,-)\,,\quad (a,-)(b,-)=(a-b,+).$$
It will be notationally convenient to identify by $a$ the operator of left multiplication by $(a,+)$ and by $\tau$ the operator of left multiplication by $(0,-)$ so that 
$$ (a,+)= a(0,+)\,,\quad (a,-)= a \tau (0,+).$$   Note the commutation rules $a\tau= \tau (-a)$.
 Sometimes we refer to the elements $a=(a,+)$  as {\em black} and $a\tau=(a,-)$  as {\em red}.  

  Recall we defined  the  mass  in Formula \eqref{MMM} by  $ \eta:{\Z^m}\to \Z,\  \eta(e_i):=1$.
 If $p\in\Z$ it is easily seen that the set  $G_p:= \{a\,:\ \eta(a)=0,\   a\tau\,:\, \eta(a)=p\}$ form a subgroup. In particular $G_{-2}$ is  generated by the elements $e_i-e_j,(-e_i-e_j)\tau$.

 The group $G$ has also a simple geometric interpretation, for any $a\in{\Z^m}$ the element $a $ acts  on ${\Z^m}$    as the translation  $t_a:x\mapsto x+a$, while the element $\tau$ is the sign change $\tau:x\to -x$, so $ a\tau$ acts by $x\mapsto a-x$.

\begin{remark}
In our dynamical setting, we have chosen a list of vectors $v_i$ and  defined (cf. Formula \eqref{MMM})  $\pi:\Z^m\to \R^n$ by $\pi: e_i\mapsto v_i$.
 
We can think of $G$ also as linear operators on $\R^n$ by setting
\begin{equation}\label{azione}
a  k:=  -\pi(a)+ k,\ k\in\R^n,\ a\in \Z^m\,,\quad \tau k:= - k
\end{equation} 
\end{remark}

For each $q=1,2,\ldots$ we  consider the Cayley graphs in   $G, \Z^m,\R^n$ associated to the set $ X=\{ X_q^{ 0 }=(X_q^{ 0 },+),\ X_q^{-2}\tau=(X_q^{-2},-)\}$ (cf. Formula \eqref{glied}).  
Notice that, for all $a\in {\Z^m}$,  we have $(a,-)^2=(0,+)=Id$, the identity of $G$. In particular $X=X^{-1}$. \smallskip

 We take two elements $ (a, \sigma ),\  (b, \rho)   \in G$ ($\sigma,\rho=\pm$). We  have \begin{equation}
 \label{pasl}(b,\rho) (a,\sigma)^{-1}=b (0,\rho)(0,\sigma)(- a) =\begin{cases}
 b-a \quad \text{if}\ \rho=\sigma \\ (a+b)\tau \ \text{if}\ \rho\neq \sigma 
\end{cases}.
\end{equation}
Therefore  $ (a ,\sigma) ,\  (b ,\rho) $ are joined by an oriented edge marked with $\ell\in X^0$  if $\sigma=\rho$ and $b-a=\ell$, while $  (a ,\sigma) ,\  (b ,\rho) $ are joined by an edge marked with $ \ell\tau,\ \ell\in X^{-2}$  if $\sigma=-\rho$ and $a+b=\ell$.
Graphically 
 $$  \xymatrix{ (b,+) \ar@{<-}[r]^{\ell}& (a,+)  },\quad  \xymatrix{(b,-) \ar@{<-}[r]^{\ell}& (a,-) }\,,\quad \xymatrix{ (a,+) \ar@{-}[r]^{\ell\tau}& (b,-)  }. $$
We have obtained the same picture as in \S \ref{Trules}, only now we are not imposing the {\em quadratic} constraint $h=-\pi(a)$, $k=-\pi(b)$ where $(h,k)$ satisfy \eqref{fico} or \eqref{fico1}. We usually drop the $\tau$ in the marking of the unoriented edges associated to $\ell\in X_q^{-2}$.

The significance of this choice is: 
\begin{proposition}
\begin{enumerate}
\item The elements $  X_q^{ 0 },\ X_q^{-2}\tau$ generate $G_{-2}$.

 \item The Cayley graph $\R^n_X$ contains the geometric graph $\GamaG$ of Definition \ref{geogra}.
\medskip

{\em  
 We  identify the  basis $F_B$ of Formula \eqref{llab}  with   $\Z^m_c\times\Z/(2)\subset G$ then:} 
 \smallskip
 
   \item The graph $\Lambda_S$ (cf. \ref{iduegrafi}) is a subgraph of the Cayley graph $G_X$.

\item Each connected component of $\Lambda_S$   is a full\footnote{a full subgraph of a graph  $\Gamma$ consists of a subset of the vertices and all the edges in $\Gamma$ between these chosen vertices} subgraph of the Cayley graph $G_X$.
\end{enumerate}
\end{proposition}
In view of the previous Proposition we set a:
\begin{definition}\label{CMG}
A {\em complete marked graph}, on a set $A\subset    {\Z^m}\rtimes\Z/(2)$ is the full sub--graph generated by the vertices   in $A$.
\end{definition}
In fact using conservation of mass and the action of $G$ on $\Z^m$, it is even better to consider  $\Lambda_S$  lying in the orbit of $G_{-2}$ in $\Z^m$  formed of elements $a\in\Z^m\,|\,\eta(a)=0,-2. $ This identification is not canonical but depends on  the choice of a {\em root} $r\in\Lambda_S$  that  corresponds to 0.
\smallskip

There are symmetries in the graph. The symmetric group $S_m$ of the $m!$ permutations of the elements $e_i$ preserves the graph. By Lemma \ref{propp}  we have the right action of $G$, on  the graph:
\begin{equation}
\label{sitr}(b,\sigma)\mapsto (b,\sigma)\tau=b \sigma\tau,\quad  (b,\sigma)\mapsto (b,\sigma) a= (b+\sigma a,\sigma),\ \forall a,b\in{\Z^m}. 
\end{equation}  
 Up to the $G$ action  any subgraph can be translated to one containing $(0,+)$ in this way we keep only the combinatorial information.

\subsection{The quadratic energy constraints\label{gerea}}

We consider $\R^m$   with the standard scalar product.  

Given a list $S$ of  $m$ vectors $v_i\in\R^n$, we have defined the  linear map $\pi:e_i\mapsto v_i $. 

Let    $S^2[{\Z^m}]:=\{\sum_{i,j=1}^ma_{i,j}e_ie_j\},\  a_{i,j}\in\Z$ be   the  polynomials of degree $2$  in the $e_i$ with integer coefficients. We extend  the map $\pi$ and introduce a linear map $L^{(2)}: a\mapsto a^{(2)}$ as:
 $$ \pi(e_i) = v_i,\quad 
\pi(e_ie_j):=(v_i,v_j),  \quad L^{(2)}:{\Z^m}\to S^2({\Z^m}),\  \ a= \sum a_i e_i \mapsto a^{(2)}:= \sum a_i e_i^2 .$$ 

We have $   \pi(AB)=(\pi(A),\pi(B)),\forall A,B\in{\Z^m}.$
 
\begin{remark}
Notice that we have   $a^{(2)}=a^2$ if and only if  $a$ equals 0 or one of the variables $e_i$.
\end{remark} 
 
\begin{definition}
Given an element $u= (a,\sigma)=(\sum_ia_ie_i,\sigma) \in G $ set 
 
\begin{equation}\label{ricon} 
 C(u):= \frac{\sigma }{2}(a^2+a^{(2)})  ,\quad   \frac{1 }{2} K(u): =\pi(C(u) )=\frac{\sigma }{2} (|\sum_ia_iv_i|^2+\sum_ia_i|v_i|^2).
\end{equation}We call $K(u)$ the {\em energy} of $u$,    this is exactly the eigenvalue of $K$  given by Formula \eqref{eiK}.
 \smallskip
 
  \end{definition}
  Notice that $C(u)$ has integer coefficients.

For   $u= (a,\sigma) $ and 
   $ g=(\sum_in_ie_i,  \rho)$ consider $gu=(b, \sigma\rho),\ b=\sum_in_ie_i+\rho a$. We have 
\begin{equation}\label{vee}
C(gu)=\sigma   C(g)+ C(u)+ (\rho-1) \frac{\sigma }{2}  a^2+ \sigma   (\sum_in_ie_i)a.
\end{equation}
 From \eqref{vee} we see that $K(gu)=  K(u)$ if and only if:
\begin{equation}\label{vee1}
0=  K(g)+ (\rho-1)   |\pi(a)|^2+    2 (\sum_in_iv_i,\pi(a)).
\end{equation}
\begin{definition}
Given an edge $\xymatrix{u\ar@{->}[r]^{x} &v    },$ $u=(a,\sigma),v=(b,\rho)=xu,\ x\in X_q$,    we say that the edge is {\em compatible} with $S$ or $\pi$ if
$K(u) =K(v) $. 
\end{definition}
As in the previous section we identify the  basis $F_B$ of Formula \eqref{llab}  with   $\Z^m_c\rtimes\Z/(2)\subset G$.  
\begin{proposition}\label{gliegg} The  graph $\Lambda_S$  of Definition \ref{iduegrafi} is  the subgraph of $G_X$ in which we only keep the compatible edges and the vertices in $\Z^m_c\rtimes\Z/(2)$.
  \end{proposition}
\begin{proof}
{\it i)}\quad If we have $a\in\Z^m$  and $b=(\ell,1)a=\ell+a$, set $k:=-\pi(b),\ h:= -\pi(a)$,  we have $k+\pi(\ell)=h$. The condition $K(a)=K(b)$  is given by formula \eqref{vee1} with $g=(\ell,1)$. This in turn gives formula \eqref{iperp} with $x= h$, i.e. implies the fact that $h\in H_\ell$ or equivalently that $h,k$ satisfy the equations \eqref{fico}.

Similarly if $b=(\ell,\tau)a=(\ell-a,\tau)$ we have $\pi(\ell)+h+k=0$, the condition $K(a)=K(b)$  is given by \eqref{vee1} with $g=(\ell,\tau)$, this gives formula \eqref{sfera} with $x=h$ or $x=k$, i.e. implies the fact that $h,k\in S_\ell$ or equivalently that $h,k$ satisfy the equations \eqref{fico1}. 
\end{proof}
 \begin{example}
In our Example \ref{EGG} consider the component $\A_{(a,+)}$ in \eqref{EGA}. By construction the  edges appear in the Cayley graph, moreover the condition that all the vertices have the same energy are the equations in \eqref{tildegamma}. The projection of  $\A_{(a,+)}$ with the map $-\pi$ gives $A_{k_1}$ in \eqref{EGA1}.
\end{example}
    This Proposition is the combinatorial counterpart of  {\em conservation of the quadratic energy $K$}  computed in Formula \eqref{eiK} and summarized as:
  \begin{itemize} 
\item If $u,v $ are in the same connected component of ${\Lambda_S}$ we have  $K(u) =K(v )$.
 \item Under the map $-\pi$, the component $A$ maps to a connected component $C$ of $\Gama$
 \end{itemize}  \begin{corollary}\label{coqe}

 A connected component $A$ of ${\Lambda_S}$ is a complete subgraph (cf. Definition \ref{CMG}) of $G_X$.
  \end{corollary}
 \begin{proof}
Fix an element $u$ of which we want to find the component. Consider the set of all elements $v$  with the same energy as  $u$.  They determine a complete (or full) subgraph of the graph $G_X$, and an edge in this subgraph is by construction compatible, thus the component passing through $u$ of this graph is  the required one.
\end{proof}  
 \section{Graph isomorphism}\label{otto}
 We wish to identify the connected components of $\Lambda_S$ with those of $\Gama$.   
  \begin{proposition}\label{coqe2}
i)\quad Under the map $(a,\sigma)\mapsto -\pi(a)$  the graph $\Lambda_S$ maps surjectively to the geometric graph $\Gama$. The image of an edge in $\Lambda_S$  is an edge in $\Gama$.
 
 ii)\quad The preimage of an   edge in $\Gama$  is a set of  edges  in $\Lambda_S$  which are simply permuted by right translations under the subgroup $\ker(\pi)$ of $\Z^m$. 
\end{proposition}
  \begin{proof}
{\it i)}\quad this is just the definition of $\Gama$ since we have shown that a compatible edge $ \xymatrix{ (a,\pm) \ar@{<-}[r]^{\ell}& (b,\pm)}$is such that setting $h=-\pi(a)$, $k:=-\pi(b)$ one has that $h,k$ respect \eqref{fico}.

{\it ii)}\quad Given  a compatible edge 
$ \xymatrix{ (a,\pm) \ar@{<-}[r]^{\ell}& (b,\pm)}$  let $h=-\pi(a)$ , $ k=-\pi(b)  $.  Consider now $a'$ such that $a-a' \in\ker(\pi)$ and set  $b' = a' +\ell$ so that 
by definition $(a',\pm)$ is connected to $(b',\pm)$ in $G_X$.  We notice that $\pi(a)=\pi(a')$ so that $K(a)=K(a')$, the same holds for $b'$ and we may conclude that $K(a')=K(b')$. This shows that $ \xymatrix{ (a',\pm) \ar@{<-}[r]^{\ell}& (b',\pm)}$ is a compatible edge. We follow the same reasoning in the case of a red edge $(\ell,\tau)$.
\end{proof}

 We now want to invert Proposition \ref{coqe2} and thus lift a connected component $C$  of $\Gama$ to a connected component of $\Lambda_S$. In our example \ref{EGG},  one can easily see that $A_{k_1}$ is isomorphic to $\A_{a,\pm}$ and hence can be lifted. However this is not always the case unless we impose some further constraints. 
Indeed consider a connected graph in $\A\in \Lambda_S$ and  let $A$ be its projection on $\Gama$.  As seen in Corollary \ref{respro}, the two graphs are isomorphic if and only if every circuit in $A$ is also a circuit in $\A$.  

 There can be two cases: 1.\quad  the circuit in $A$ contains an even number of {\em red edges}.  2. \quad the circuit in $A$ contains an odd number of {\em red edges}.

\begin{example*}[Case 1]
\quad suppose that the geometric graph contains a component
 $$ \xymatrix{  &\ar@{->}[dl]_{v_2-v_4}k_3\ar@{<-}[dr]^{ v_2-v_3\quad}& \\ k_1\ar@{->}[rr]^{ v_2-v_1\quad}& & k_2  },$$ which is the case provided that
 $$ v_1-3v_2+v_3+v_4=0 \,, \quad \left\{ \begin{aligned} 2(k_1, v_2-v_1)= |v_2-v_1|^2 + |v_2|^2-|v_1|^2 \\
 2(k_1, v_4-v_2)= |v_4-v_2|^2 +|v_4|^2 -|v_2|^2\end{aligned}\right.$$ Let us choose any $a\in \Z^m$ such that $-\pi(a)=k_1$.  We easily verify  that  the tree
 $$\xymatrix{   (a,+)\ar@{->}[r]^{e_1-e_2\quad}& (a+ e_1- e_2,+)  \ar@{->}[r]^{e_3-e_2}&  (a  +e_1-2e_2+e_3,+)\ar@{->}[r]^{e_4-e_2}&  (a  +e_1-3e_2+e_3+e_4,+) } $$ is contained in $\Lambda_S$.  Let us call $v= e_1-3e_2+e_3+e_4$, by hypothesis 
 $ \pi(v)=0$,  so that we have $-\pi(a+ \alpha v )=k_1$ for all integer $\alpha$. This implies that the connected component of $(a,+)$ has infinitely many vertices:
  $$\xymatrix{   (a,+)\ar@{->}[r]^{e_1-e_2}& (a+ e_1- e_2,+)  \ar@{->}[r]^{e_3-e_2}&  (a  +e_1-2e_2+e_3,+)\ar@{->}[r]^{e_4-e_2}&  (a  +v,+) \ar@{->}[d]^{e_1-e_2}\\ \cdots \ar@{<-}[r]^{e_1-e_2}& (a+  2v,+)  \ar@{<-}[r]^{e_4-e_2\qquad}&  (a  + v+ e_1-2e_2+e_3,+)\ar@{<-}[r]^{\quad e_3-e_2}&  (a  +v+e_1-e_2,+) } $$ 
To avoid this pathology we simply require that $v_1-3v_2+v_3+v_4\neq 0$ so that this geometric graph does not have a realization.
\end{example*}  
\begin{example*}[Case 2] Suppose that the geometric graph contains a component
 $$ \xymatrix{  &\ar@{->}[dl]_{v_2-v_4}k_3\ar@{=}[dr]^{ v_2+v_3\quad}& \\ k_1\ar@{->}[rr]^{ v_2-v_1\quad}& & k_2  },$$ which is the case provided that
 $$k_2+k_3=k_1+v_2-v_1+k_1+v_4-v_2= v_2+v_3 ,  $$ $$ 2k_1= v_1+v_2+v_3-v_4 \,, \quad \left\{ \begin{aligned} 2(k_1, v_2-v_1)= |v_2-v_1|^2 + |v_2|^2-|v_1|^2 \\
 2(k_1, v_4-v_2)= |v_4-v_2|^2 +|v_4|^2 -|v_2|^2\end{aligned}\right.$$
we substitute $2k_1$ in one of the linear equations and obtain that this geometric graph does not have realization if
$$  (v_1+v_2+v_3-v_4, v_4-v_2)\neq  |v_4-v_2|^2 +|v_4|^2 -|v_2|^2.$$ 
\end{example*}

To repeat this reasonings in the general case we need the following trivial fact:
 \begin{lemma}
If $a=\sum_in_ie_i\in{\Z^m}$ resp. $(a,\tau)$ is a product of $d$  elements in $X_q$  we have that $\sum_i|n_i|\leq 2d q$.
\end{lemma}
It should be clear at this point that in order to {\em lift} the components of  $\Gama$ with at most $d$ vertices we must impose as many linear/quadratic inequalities on $S$ as the number of loops which may appear in a component. Thus if we wish to impose only a finite number of constraints we cannot lift arbitrarily large components. Our strategy is the following: first we fix $d= 2n+2$ and impose constraints to ensure that all components with at most $d$ vertices can be lifted. Then we show that there are no compatible graphs in $\GamaG$ with $d$ vertices, this excludes the existence of graphs $C$ in $\Gama$ with $d$ or more vertices. Otherwise we would be able to lift some subgraph of $C$ with $d$ vertices  to a compatible graph in $\Lambda_S$.  This means that the mapping $-\pi$ gives an isomorphism from each connected component of $\Lambda_S$ to its image in  $\Gama$.

\bigskip

 We impose
 \begin{constraint}\label{co4}
We assume $\sum_i\ell_iv_i\neq 0$, for all choices of the $\ell_i$ such that $\sum_i\ell_i=0,\ \sum_i|\ell_i|\leq 4q(n+1)$ and  $\sum_i\ell_ie_i\neq 0$.
\end{constraint} 

Under this constraint take an element  $g=\sum_in_ie_i $  which is a product of $d\leq 2n+2$  elements in $X$. We have then  $\sum_i|n_i|\leq 4 q (n+1)$ so if $g\neq 0$ we have $\pi(g)=\sum_in_iv_i\neq 0$.   Then for all $k\in \Z^n$ $g k=\pi(g)+k \neq k,\ \forall k$, hence case 1. may not occur.  

\smallskip

For case 2.  let $ g= (a,\tau),\ a=\sum_in_ie_i,\ \sum_in_i=-2$ be such that 
 $g k=k$ for some $k\in \Z^n$. This is possible  if and only if $\pi(a)=\sum_in_iv_i=2k$.
 Since  we are assuming that there is a non trivial odd loop starting from $k$, changing if necessary the starting point, the first step of the loop   tells us that
 $k$ lies in   a sphere $S_{\ell}$ for some initial edge $\ell \in X$.

This implies that $k=-1/2\sum_in_iv_i$ satisfies a  relation of type 
\begin{equation}\label{pririn}  |\sum_in_iv_i|^2-2( \sum_in_iv_i,\pi(\ell))=2 K(\ell).
\end{equation}  Where  $\ell=(\sum_i\ell_ie_i)\in X^{(-2)}_q$.   This formula  vanishes identically     if   $ a ^2-2  a \ell  =2 C(\ell )=-2(\ell ^2+\ell^{(2)})$.  Thus  $$(a-\ell)^2=-\ell^2-2\ell^{(2)}. $$

This implies that all  coefficients of $\ell$ must be $-1$, and $\ell=-e_i-e_j$.

This implies $a-\ell=\pm (e_i-e_j)$ hence $a=-2e_i, -2e_j$ and $k=v_i,v_j$.

We impose
 \begin{constraint}\label{co5}
We assume that   for all choices of the $n_i$ such that $\sum_in_i=-2,\  \sum_i|n_i|\leq 4q(n+1)$ all   equations \eqref{pririn} are not satisfied. 

\end{constraint}

If $C$ is any marked graph which has at most  $d$  vertices, a minimal loop in $C$ has at most $d$ edges, thus:
\begin{corollary}\label{cabel}
Under the previous constraints  if $C\subset \Gama$ is a connected   graph  with at most $2n+2$ vertices then $C$ can be lifted.
\end{corollary}
\begin{proof}
By Corollary \ref{respro} we only need to prove that, under the previous hypotheses,  it is not possible that a non trivial element  $g$, which is a product of at most  $2n+2$ elements of $X$, fixes an element $k\in C$.

By the constraints that we have imposed this may happen if and only if  this element generates a trivial constraint, that is an identity for all choices of $v_i$.  If $g=a\in\Z^m$  this is excluded by Constraint \ref{co4}  and for  $g=a\tau$   it is excluded by Constraint \ref{co5}.\end{proof}
   \section{The equations defining a connected component of $\Gama$.\label{eqgama}}  
   
   Take   a  connected subgraph $A$ of $\Gama$ which can be lifted (in particular this will be the case if $A$ has at most $2n+2$ vertices by the previous constraints). Choose a root $x\in A $, we lift $x=-\pi(a)$, this  lifts $A$ to the component $\mathcal A_{(a,+)}$ through $a$ in $\Lambda_S$. For  $h\in A$ we have an element $g_h \in G$  obtained by lifting a path in $A$ from $x$ to $h$ and such that $h=g_h x$. We set   
   \begin{equation}\label{lasig}
g_h :=( L(h),\sigma(h)),\quad L(h)\in\Z^m,\ \sigma(h)\in\{1,\tau\}\implies h=-\pi(L(h))+\sigma(h)x.
\end{equation}

We then can deduce that:
\begin{lemma}
 For each $h\in A$  we  have:
\begin{equation}\label{bacos}
\begin{cases}
  (x,\pi(g_h))= \frac{1}{2}K(g_h )\quad \text{if}\  \sigma(h)=1\\
 |x|^2+(x,\pi(g_h))= \frac{1}{2}K(g_h )\quad \text{if}\  \sigma(h)=\tau\end{cases}
.\end{equation} 
 \end{lemma}
 \begin{proof} 
We use Formula \eqref{vee1} which implies that.
\begin{equation} 
0=   K(g_h ) + (\sigma(h)-1)   |x|^2-  2(\pi(g_h),x).
\end{equation} To be explicit if $L(h)=\sum_im_ie_i$ by \eqref{ricon}:
\begin{equation}  
\pi(g_h)= \sum_im_iv_i,\quad  K(g_h )=   \sigma(h) (|\sum_im_iv_i|^2+\sum_im_i|v_i|^2).
\end{equation} 
\end{proof}
   
The equations on $x$ given in Formula \eqref{bacos} are a complete set of conditions  for the existence of a graph $A$  inside some connected component (which could also properly contain $A$) of $\GamaG$. The reader should notice that these equations are completely analogous to the ones of Definition \ref{lasfe}, given only for edges.
 \begin{definition}\label{cogrA}
Let $\GA\subset G_X$ be the  graph with  vertices the elements $g_h $ (and $g_x=(0,+)=Id$), this is called the {\em combinatorial graph} associated to $A$ and the {\em root} $x$. \end{definition} 
\begin{example} \label{EGG2} We explicitly compute the combinatorial graph associated to $A_{k_1}$ of Example \ref{EGG}. We choose $k_1$ as the root.
\begin{equation}\label{EGA2}{\GA} = \xymatrix{ & \ar@{-}[d]^{-e_2-e_1}(-e_2-e_3,-)&&\\ (0,+)\ar@{->}[r]^{e_3-e_1\quad}& ( e_3- e_1,+)  \ar@{->}[r]^{e_3-e_2}&  (  -e_1-e_2+2e_3,+)  }  \end{equation}  

The system of equations associated to this graph is
\begin{equation}\label{EGB2}
\left\{\begin{array}
{l} (x,v_3-v_1)= |v_3|^2-(v_1,v_3) \\ (x, -v_1-v_2+2v_3)= 3|v_3|^2-2(v_1+v_2,v_3)+(v_1,v_2) \\ |x|^2-(x,v_2+v_3)= - (v_2,v_3)
\end{array}\right.
\end{equation} Notice that this graph {\bf does not} belong to $\Lambda_S$ but if $a\in \Z^m$ is such that  $-\pi(a)=k_1$ then the right translation translation by $(a,+)$, i.e.  $\GA (a,+)$ gives $\mathcal A_{(a,+)}\in \Lambda_S$.  
\end{example}
\begin{remark}
Notice that the map which associates to each $h\in A$ the element $g_h=(L(h),\sigma(h))$ is well defined only if $A$ can be lifted. The construction of the $L(k)$ is in turn the key to the reducibility as can be seen in Example \ref{redEGG}
\end{remark}
Consider  now a complete  subgraph (cf. Definition \ref{CMG}) of $G_X$ which contains $(0,+)$. We associate to each vertex $g\neq (0,+)$ of the graph an equation:
\begin{equation}\label{bacosG}
\begin{cases}
  (x,\pi(a))= \frac{1}{2}K(a )\quad \text{if}\  g=(a,+)\\
 |x|^2+(x,\pi(a))= \frac{1}{2}K(a )\quad \text{if}\  g=(a,-)\end{cases}
.\end{equation} 
  We think of this system of equations as   associated to the graph.

\begin{definition}
We call the set of complete  subgraphs of $G_X$ which contain $(0,+)$ and have at most $2n+2$ vertices the set of {\em possible combinatorial graphs}.  We say that a possible combinatorial graph $\GA$ has a geometric realization (in $\GamaG$ ) if  the equations associated to the graph have real solutions. 
\end{definition} 
The following statement holds.
\begin{proposition}\label{brik}
A possible combinatorial graph $\GA$ is a combinatorial graph if and only if the equations \eqref{bacosG} have  solutions in $S^c$.
\end{proposition} 

\begin{remark}
Notice that in a possible combinatorial graph one may deduce the color of each vertex by computing its mass. Indeed all vertices $(a,+)$ must have $\eta(a)=0$ while $(a,-)$ corresponds to $\eta(a)=-2$.
\end{remark}
We have reduced our problem to that of understanding which possible combinatorial graphs have a geometric realization. Naturally for given $S$ this amounts to checking wether the equations associated to the graph are independent and-- if they are not--  to verify their compatibility. 

\begin{definition}\label{isotras}
We say that two possible combinatorial graphs are equivalent if one is obtained from the other by right translation by an element of $G$ (see formula \eqref{sitr}).
\end{definition}
\begin{remark}
It should be clear that if $\GA$ has a geometric realization then so has any other equivalent possible combinatorial graph. Moreover the two identify the same components of $\GamaG$ with a different choice of the root.
\end{remark}
\begin{example}\label{EGG3}
The following combinatorial graph is equivalent to $\GA$ of example \ref{EGG2}:
\begin{equation}\label{EGA3}{\GA'} = \xymatrix{ & \ar@{-}[d]^{-e_2-e_1}(-e_2-e_1,-)&&\\ (e_1-e_3,+)\ar@{->}[r]^{e_3-e_1\quad}& ( 0,+)  \ar@{->}[r]^{e_3-e_2}&  (  e_3-e_2,+)  }  \end{equation}   Indeed it is obtained by right translation with the element $(e_1-e_3 ,+)$.
The equations are
\begin{equation}\label{EGB3}
\left\{\begin{array}
{l} (x,v_1-v_3)= |v_1|^2-(v_1,v_3) \\ (x, v_3-v_2)= |v_3|^2-(v_2,v_3) \\ |x|^2-(x,v_2+v_1)= - (v_2,v_1)
\end{array}\right. 
\end{equation}  it is easily seen that these equations are equivalent to the system given by formula \eqref{EGB2}, they still identify the geometric graph $A_{k_1}$ of example \ref{EGG} only now the root is in $k_2$.
\end{example}
 \subsection{Relations}\label{noveuno}

Take a {\em possible combinatorial graph} $\GA$ 
 \begin{definition}\label{ranghi}\begin{itemize}
\item If  $\GA$ has $k+1$ vertices is  said to be of {\em dimension} $k$.

\item  The dimension of the  lattice    generated by the vertices of $\GA$ is
the {\em rank}, ${\rm rk}\, \GA$,  of the graph $\GA$. The dimension of the lattice generated by the black vertices $(a,+)$ (resp. red)  is called the black (resp. red) rank of $\GA$.
\item If the rank of $\GA$ is strictly less than the dimension of $\GA$ we say that $\GA$ {\em is degenerate}.
\end{itemize}
\end{definition}
Take a connected component $A$ of $\Gama$ and choose  a root $ x \in A.$  Assume that $A$ can be lifted. Let $\GA=\{g_a ,\ a\in A\}$ be the combinatorial  graph of which $A$ is a geometric realization. 
\begin{lemma}
The rank of $\GA$ does not depend on the choice of the root but only on $A$. 
\end{lemma}
\begin{proof} If we change the root from $x$ to another $y$  we can stress in the notation $g_{a,x } =(L_x(a),\sigma_x(a))$ and  have 
\begin{equation}
\label{chro}g_{ a,x}=g_{ a,y}g_{y,x},\implies L_x(a)=L_y(a)+\sigma_y(a)L_x(y),\  \sigma_x(a)=\sigma_y(a)\sigma_x(y).  
\end{equation}  This shows that   the notion of rank is independent of the root.
\end{proof}

Notice that when we change the root in $A$ we have a simple way of changing the colors and the ranks of the vertices of $\GA$ that we leave to the reader.

If $\GA$ is degenerate  then there are non trivial relations,   $\sum_an_a a=0,\ n_a\in\Z$  where the sum runs among the vertices $a\in \GA$.
\begin{remark}
\label{maxt} It is also useful to choose a maximal tree $T$ in $\GA$. There is   a triangular change of coordinates from the vertices to the markings of $T$. Hence the relation can be also expressed as a relation between these markings.
\end{remark}

 We must have by linearity, for every relation  $\sum_an_a a=0,\ n_a\in\Z$ that $0=\sum_an_a a^{(2)},$ where we recall that if $a=\sum a_i e_i$ we have that $a^{(2)}= \sum a_i e_i^2$.  Finally we have $ \ 0=\sum_an_a \pi(a)$ and $\sum_an_a \eta(a)=0$. 
 
 Recalling that $\eta(a)=0,-2$ (resp. if $a$ is black or red), we have :\begin{equation}
0=\sum_{a\, |\, \eta(a)=-2}n_a.
\end{equation}
Applying Formula \eqref{bacos}  we deduce that, in order to ensure that the equations of $\GA$ are compatible,  we must have
\begin{equation}\label{riso}
\sum_an_aK(a )=2(x,\sum_an_a \pi(a))+[\sum_{a \, |\, \eta(a)=-2}n_a](x)^2  =2(x,\sum_an_a \pi(a))=0.
\end{equation}   
 The expression $\sum_an_a\frac{K(a)}{2} =\pi(\sum_an_aC(a) )$ is a linear combination with integer coefficients of the scalar products $(v_i,v_j)$.
 
 Given a {\em possible combinatorial graph}   $\GA$ with a relation: 
  \begin{definition}
 If $\sum_an_aC(a) \neq 0$     we say that the graph has an  {\em avoidable resonance}.  
 \end{definition}

\begin{lemma}
A degenerate possible combinatorial graph $\GA$  with an avoidable resonance has no geometric realization for a generic choice of the $S:=\{v_i\}$.
\end{lemma}
\begin{proof}
The graph has a realization only if $\sum_an_aK(a )=0$ but this polynomial, by definition, is not identically zero.
\end{proof}
\begin{example}\label{pirir}
Consider the possible degenerate combinatorial graph
 $$ 
 \GA= \xymatrix{e_1-e_2 \ar@{<-}[r]^{e_1-e_2}& 0     \ar@{=}[d]_{-e_1-e_2}
  \ar@{=}[r]^{-e_1-e_3}
  &
  {-e_1-e_3}\ar@{->}[r]^{e_1-e_3} & {-2e_3} \\&{-e_1-e_2}&& } 
 $$
The relation is $  (e_1-e_2)+2(-e_1-e_3)-(-2e_3)-(-e_1-e_2)=0  $.

We may write the value of $C(a)$
 of each vertex $a$, we get
 $$  \xymatrix{e_1^2-e_1e_2 \ar@{-}[r]& 0     \ar@{=}[d]\ar@{=}[r]&{e_1e_3}\ar@{-}[r] & {-e_3^2} \\&{e_1e_2}&& }$$
we have
$$ \sum_a n_a C(a)=
\qquad e_1^2-e_1e_2 +2e_1e_3+e_3^2-e_1e_2 
$$
so the equations of this graph are incompatible if $\pi( e_1^2-e_1e_2 +2e_1e_3+e_3^2-e_1e_2)\neq 0$, this is a generiticity condition.
\end{example}

We arrive now at the main Theorem of the section:
\begin{theorem}\label{ridma}
Given  a possible combinatorial graph  of rank $k$ for a  given color, then either it has  exactly $k $  vertices of that color or it produces an avoidable resonance.  \end{theorem}
\begin{proof}    
Assume by contradiction that  we can choose   $k+1$ vertices $(a_0,a_1,\ldots,a_k) $, different from the root of the given color  so that we have a non trivial relation   $\sum_in_i  a_i =0 $ with $n_0\neq 0$ and the vertices $a_i,\ i=1,\ldots,k$ are linearly independent.  
We compute the resonance relation
$$2\sum_in_i  C(a_i) = \sum_in_i  \sigma(a_i)(a_i^2+ a_i^{(2)}) = \pm \sum_in_i  (a_i^2+ a_i^{(2)}),$$ 
since all the vertices $a_i$  have the same  color. By linearity we have $ \sum_in_i a_i^{(2)}=0$.
We deduce that  $\sum_in_i  C(a_i)= \pm \sum_in_ia^2_i$.

 We consider the elements $a_i$ with $i=1,\ldots,k$ as {\em independent variables}  and  write  the relations        as
$$0=n_{0}a_{0} +\sum_{i=1}^kn_ia_i,\implies (\sum_{i=1}^k n_ia_i)^2+ n_{0}\sum_{i=1}^kn_ia_i^2=0 .$$
Now   $\sum_{i=1}^kn_ia_i^2 $ does not contain any mixed terms $a_ha_k,\ h\neq k$  therefore  this equation can be verified if and only if  the sum $\sum_{i=1}^kn_ia_i$ is reduced to a single term   $n_ia_i$,  and then we have $n_0=-n_{i}$ and $a_0=a_i$,  a contradiction. \end{proof}

 \begin{constraint}\label{co6}  We impose that the  vectors $v_i$ are generic for all resonances arising from degenerate possible combinatorial graphs with  at most $n+1$  elements of a given color.
\end{constraint}
\begin{remark}\label{bala}
 It is essential that we introduce the notion of colored rank, otherwise our statement is false as can be seen with the following graph:
\begin{equation}
\label{mig} \xymatrix{ & \ar@{-}[d]\ar@{=}[r] (   -e_2+e_1)&  ( -2e_1 ) \ar@{-}[d]&\\ & \ar@{=}[r]0&   ( -e_2-e_1 ) &  }
\end{equation}   
Relation is $(   -e_2+e_1)-  ( -e_2-e_1 ) +( -2e_1 )=0$, we have 
$$ C(   -e_2+e_1)=e_1^2-e_1e_2,\quad  C( -e_2-e_1 ) =-e_1e_2,\quad C( -2e_1 )=-e_1^2$$ $$ e_1^2-e_1e_2-(-e_1e_2)-e_1^2=0.$$
Actually this graph does not really pose any problem since its only geometric realization is in $S$ (hence it is {\bf not} a true combinatorial graph).
However we are not able to exclude the existence of more complicated graphs of this form which may have realization in $S^c$.
\end{remark}
We are  reduced to considering possible combinatorial graphs with at most $2n+2$ vertices and such that the vertices of the same color are linearly independent. We call these graphs colored--non--degenerate.

 We now look at the equations \eqref{bacosG} associated to the graph by a choice of $S$. 
  Consider a possible combinatorial graph $\GA$ of  black rank $h$ and red rank $k$. If $h\leq n$ then we can require that the images of the black vertices $a\in \GA$  through the map $-\pi$ are  independent. Then the linear equations \eqref{bacosG} associated to these vertices are independent and have solutions. 
The same holds for the red vertices, only the equations \eqref{bacosG} associated to these vertices are quadratic and so the solutions need not be real.

Given a colored--non--degenerate possible combinatorial graph $\GA$   with ranks $h,k\leq n$ for dimension $n$ we associate to it the $n\times h$ matrix $M^+(\GA)$ with columns  the vectors $\pi(a)$ where   $a$ runs over the black vertices. Same for the  e $n\times k$ matrix $M^-(\GA)$.
 \begin{constraint}\label{co8} For any colored--non--degenerate possible combinatorial graph $\GA$ with red and black rank $\leq n$  we require that:\footnote{If $A$ is a $a\times b$ matrix and $h\leq \min(a,b)$ we denote by $\wedge^hA$ the matrix withh  entries the determinants of the $h\times h$ minors.}
 $$ \wedge^h (M^+(\GA))\neq 0\,,\quad  \wedge^k (M^-(\GA))\neq 0$$
\end{constraint}

If one of the colored ranks is $k=n+1$, then any choice of $S$ must lead to a relation between the vectors $\pi(a_i)$ where the $a_i$ are the vertices of the same color. We will use this to show that either the equations are generically incompatible or they give a solution in the special component. This is the content of the next section.
 \section{Geometric realization \label{Magt}}
  Consider a possible--combinatorial  graph $\GA$  with $\leq 2n+2$ vertices and suppose that  it has rank $n+1$.
  By Theorem \ref{ridma}, the vertices of each color  are linearly independent. We want to study its geometric realizations in dimension exactly $n$.  For this we can consider the variety  $R_{\GA}$  {\em of realizations of the graph} i.e. the set of points $(x,v_1,\dots,v_m)\in \C^{(m+1)n}$ which satisfy the equations \eqref{bacosG} associated to $\GA$. 
  
  Call $\theta:R_\GA\to \C^{mn}$ the projection map $(x,v_1,\dots,v_m) \to (v_1,\dots,v_m)$. We say that a graph is {\em not} realizable for generic $v_i$  if  $\theta(R_\GA)$ is an algebraic variety of codimension  at least one.\footnote{In this discussion we ignore the delicate issues  of wether a realization may be integral, real or imaginary.}
 
  Suppose that we have $n+1$ black vertices (different from the root). If we choose $n$ of them (discarding say $a_1$) by Theorem \ref{ridma}, we can require that for generic $S$ the $\pi(a_i)$ with $i=2,\dots, n+1$, are independent. This we do by choosing $S$ so that the determinant of the matrix $M_1$ having $\pi(a_i)$ as rows is non--zero.
Then the system of equations is incompatible if the $n+1\times n+1$ matrix obtained by adding the row $\pi(a_1)$ and the column of inhomogeneous terms has non--zero determinant.
We compute this determinant which is a polynomial in the $v_i$ and {\em if it is not identically zero} we impose it as a generiticity constraint and $\GA$ is not generically realizable.

If it is identically zero then the equations have a solution, which we can compute by Cramer's rule by discarding the first equation. Hence $\GA$ is generically realizable.
\vskip10pt

In the same way suppose we have $n+1$ red vertices. We choose one of them, say $a_1$, and subtract the equation for $a_1$ to the remaining equations \eqref{bacosG}. We obtain a system of $n$ linear equations $M_1 x= b$ which, by Theorem \ref{ridma}, are generically independent. We impose as a genericity constraint $\det(M_1)\neq 0$ and solve the equations by  Cramers rule.  We obtain a solution $x_{\GA}$ which is a rational function in the $v_i$ with $\det (M_1)$ at the denominator.  The graph has realization for all the $S$  for which $x_{\GA}$ solves the quadratic equation associated to $a_1$. We substitute $x_\GA$ and rationalize. If the numerator is a non--zero polynomial then we impose it as a generiticity constraint and $\GA$ is not generically realizable.
Summarizing, we impose  
\begin{constraint}\label{co7}   For any colored--non--degenerate possible combinatorial graph $\GA$ with red and/or black rank $ n+1$, we impose that the  vectors $v_i$ are generic for all resonances described above.
 \end{constraint}
\begin{example}
We consider the combinatorial graph of example \ref{EGG3} in dimension $n=2$. We impose $$d=(v_{1,1}-v_{3,1})(v_{3,2}-v_{2,2})-( v_{1,2}-v_{3,2})(v_{3,1}-v_{2,1}) \neq 0$$solve the first two equations \eqref{EGB3} by Cramer's rule and obtain the solution $x= (x_1,x_2)$: 
$$ x_1=(|v_1|^2-(v_1,v_3))(v_{3,2}-v_{2,2})-( v_{1,2}-v_{3,2})(|v_2|^2-(v_2,v_3))/d\,,$$ $$x_2=(v_{1,1}-v_{3,1})(|v_2|^2-(v_2,v_3))-(|v_1|^2-(v_1,v_3))(v_{3,1}-v_{2,1})/d .$$ We substitute in the last equation, rationalize and obtain that a realization exists only if

$$ \left( (v_1,v_2)-(v_1,v_3)+|v_3|^2-(v_2,v_3) \right)\cdot \,\left( {v_{1,1}}^3\,v_{2,1} + v_{1,1}\,{v_{1,2}}^2\,v_{2,1} + 
    {v_{1,2}}^2\,{v_{2,1}}^2 +\right.$$   $$ \left.{v_{1,1}}^2\,v_{1,2}\,v_{2,2} + {v_{1,2}}^3\,v_{2,2} - 
    2\,v_{1,1}\,v_{1,2}\,v_{2,1}\,v_{2,2} +\right.$$ $$\left. {v_{1,1}}^2\,{v_{2,2}}^2 - {v_{1,1}}^3\,v_{3,1} - 
    v_{1,1}\,{v_{1,2}}^2\,v_{3,1} - 3\,{v_{1,1}}^2\,v_{2,1}\,v_{3,1} - 3\,{v_{1,2}}^2\,v_{2,1}\,v_{3,1} + \right.$$ $$\left.
    2\,v_{1,2}\,v_{2,1}\,v_{2,2}\,v_{3,1} - 2\,v_{1,1}\,{v_{2,2}}^2\,v_{3,1} + 3\,{v_{1,1}}^2\,{v_{3,1}}^2 + 
    2\,{v_{1,2}}^2\,{v_{3,1}}^2 + 3\,v_{1,1}\,v_{2,1}\,{v_{3,1}}^2 - \right.$$ $$\left. v_{1,2}\,v_{2,2}\,{v_{3,1}}^2 + 
    {v_{2,2}}^2\,{v_{3,1}}^2 - 3\,v_{1,1}\,{v_{3,1}}^3 -  v_{2,1}\,{v_{3,1}}^3 + {v_{3,1}}^4 - \right.$$ $$\left.
    {v_{1,1}}^2\,v_{1,2}\,v_{3,2} - {v_{1,2}}^3\,v_{3,2} - 2\,v_{1,2}\,{v_{2,1}}^2\,v_{3,2} - 
    3\,{v_{1,1}}^2\,v_{2,2}\,v_{3,2} - 3\,{v_{1,2}}^2\,v_{2,2}\,v_{3,2} + \right.$$ $$\left. 2\,v_{1,1}\,v_{2,1}\,v_{2,2}\,v_{3,2} + 
    2\,v_{1,1}\,v_{1,2}\,v_{3,1}\,v_{3,2} + 4\,v_{1,2}\,v_{2,1}\,v_{3,1}\,v_{3,2} +  
    4\,v_{1,1}\,v_{2,2}\,v_{3,1}\,v_{3,2} \right.$$ $$\left.- 2\,v_{2,1}\,v_{2,2}\,v_{3,1}\,v_{3,2} - 
    3\,v_{1,2}\,{v_{3,1}}^2\,v_{3,2} - v_{2,2}\,{v_{3,1}}^2\,v_{3,2} + \right.$$ $$\left.2\,{v_{1,1}}^2\,{v_{3,2}}^2 + 
    3\,{v_{1,2}}^2\,{v_{3,2}}^2 - v_{1,1}\,v_{2,1}\,{v_{3,2}}^2 + {v_{2,1}}^2\,{v_{3,2}}^2 + 
    3\,v_{1,2}\,v_{2,2}\,{v_{3,2}}^2 - \right.$$ $$\left. 3\,v_{1,1}\,v_{3,1}\,{v_{3,2}}^2 - v_{2,1}\,v_{3,1}\,{v_{3,2}}^2 + 
    2\,{v_{3,1}}^2\,{v_{3,2}}^2 - 3\,v_{1,2}\,{v_{3,2}}^3 - v_{2,2}\,{v_{3,2}}^3 + {v_{3,2}}^4 \right)=0$$
\end{example}

 We thus have the final definition of generic for tangential sites $S$.
\begin{definition}\label{fincon}
We say that the tangential sites are {\em generic} if they do not  vanish for any of the polynomials given by Constraints \ref{co1} through \ref{co7}.
\end{definition}
\begin{remark}
Each of the constraints involves at most $2n+2$ edges, thus at most $4q(n+1)$ indices which have to be taken up  to symmetry by $S_m$ hence can be taken in correspondence with the vector variables $y_1,\ldots,y_{4q(n+1)}$.
\end{remark}
  We have ensured that for generic choices of $S$ only those graphs   which are generically realizable are realized.
  \begin{example}
Consider the possible combinatorial graph:
  $$   \xymatrix{ & \ar@{-}[d]^{-e_3-e_4}(-e_3-e_4,-)\ar@{-}[dl]_{-e_1-e_4}\ar@{-}[dr]^{-e_2-e_4}&&\\ (e_3-e_1,+)\ar@{<-}[r]^{e_3-e_1\quad}& ( 0,+)  \ar@{->}[r]^{e_3-e_2}&  (  e_3-e_2,+)  }\,,$$  It is easily seen that in dimension $n=2$ this graph is generically realizable, and its equations  have the unique solution $x=v_3$.

\end{example}

We now want to study those graphs of rank $n+1$  which are generically realizable in dimension $n$. As we have seen, on a Zariski open set of the space $v_1,\ldots,v_m$  we have a unique realization given by solving a system of linear equations and thus given by a vector $x$ whoose coordinates are rational functions in the vectors $v_i$. We call this function  the {\em generic realization}.
  \begin{theorem}\label{aMT}
If $\GA$ is a  possible combinatorial graph   of rank   $n+1$   which has a realization   for generic  $v_i$'s, then   its generic realization is in the special component (the solution $x$ belongs to the set $S$).
\end{theorem}
The proof is based on two  points.  A graph  which has a generic geometric realization in the special component is called {\em special}. It is easy to describe the special graphs, up to translation they correspond to combinatorial graphs with vertices in the set  $-e_i, -e_j\tau,i,j=1,\ldots,m$.
\begin{lemma}\label{spegra}
If for a non degenerate graph of dimension $n>1$ the solution to the associated system, in dimension $n$, is given by a polynomial then the graph is special and the polynomial is of the form $v_i$ for some $i$.

For $n=1$ the same result is true for a nondegenerate graph with 2 edges.\end{lemma}
\begin{proof}See appendix \ref{prspegra}.
\end{proof}

Let $\GA$ be a graph of rank $\geq n+1$, consider as before the variety  $R_{\GA}$    of realizations of the graph, with its map $\theta:R_{\GA}\to \C^{mn}$.
 
  \begin{proposition}\label{codim}  There is an irreducible hypersurface   $W$ of $\C^{mn}$    such that the map $\theta$ has an inverse on $\C^{mn}\setminus W$. The inverse is a polynomial map.
\end{proposition} 
\begin{proof}
{\it Black edges}
\smallskip

We have $n+1$ linear equations $ (x,\pi(a_i))= b_i$ which are generically compatible. We  solve them by Cramer's rule  choosing an index $j$ and   discarding the equation \eqref{bacosG} associated to a vertex $a_{j}$. Since the equations are always compatible we must obtain, generically,   the same solution for all choices of $a_j$. Consider  the matrix $M_{j}$ with rows the $\pi(a_i)$, $ i=1,\dots,n+1$ $i\neq j$. The solution  is  a rational function of the $v_i$ having as denominator the determinant of    $M_{j}$.  This reasoning defines  the solution of our equations for all $S$ for which there exists a $j$ such that $\det(M_j)\neq 0$.   
We claim that each   of these determinants is an irreducible polynomial so it defines an irreducible hypersurface $H_j$.

In fact  a choice of $n$ rows  gives by assumption a surjective linear map $\C^{mn}\to \C^{n^2}$. Any surjective linear map can be considered (in appropriate coordinates) as a projection on the first $mn$ coordinates. Hence an irreducible polynomial remains irreducible by composition. The claim follows since it is well known that the determinant is an irreducible polynomial of the matrix elements. 

 We claim that these hypersurfaces are not all equal. 
  By hypothesis the matrix $B=(a_{ij})$ has rank $n+1$. All the matrices obtained by $B$ dropping one row define the various determinantal varieties, $H_j$.  The fact that these varieties are not equal is discussed in Appendix \ref{detvar}.  It depends on the fact that the matrices cannot have all the same kernel (otherwise the rank of $B$ is $\leq n$). Then the result follows by Proposition \ref{kke}.

  Hence  our solution is well defined outside a subvariety of codimension at least $2$.  This implies immediately that it is given by a polynomial using  the following standard  fact (which   follows immediately from the unique factorization property of polynomial algebras): {\it Let $W$ be a subvariety  of $\mathbb C^N$ of codimension $\geq 2$, let $F$ be a rational function on  $\mathbb C^N$ which is holomorphic on $\mathbb C^N\setminus W$,  then $F$ is a polynomial. }

\smallskip

{\it Red edges}
\smallskip

When we also have red edges we select $n+1$ linear and quadratic equations  associated to the $n+1$ vertices which are formally independent. We see that the equations \eqref{bacosG}  (for these vertices) are clearly equivalent to a system on $n$ linear equations associated to formally linearly independent markings, plus a quadratic equation chosen arbitrarily among the ones appearing in \eqref{bacos}.   Thus a realization is obtained by solving this system and, by hypothesis, such solution satisfies  the quadratic equation identically.   
 
 Let $P$ be the space of functions $\sum_{i=1}^mc_iv_i,\ c_i\in\mathbb R$ and $(P,P)$ their scalar products.
By assumption we have a list of $n$ equations  $\sum_{j=1}^ma_{ij}(x,v_j)=(x,t_i)=b_i$  with the $t_i=\sum_{j=1}^ma_{ij} v_j$ linearly independent  in the space $P$ and $b_i=\sum_{h,k}a^i_{h,k}(v_h,v_k)\in (P,P)$.

Solve these equations  by Cramer's rule    considering the $v_i$ as parameters. Write  $ x_i=f_i/d $, where  $d(v):=\det(A(v))$ is the determinant of the matrix $A(v)$ with rows $t_i$,     $f_i(v)$ is also a determinant of another matrix $B_i(v)$ both depending polynomially on the $v_i$. We have thus expressed  the coordinates $x_i$ as rational functions  of the coordinates of the $v_i$. The denominator is an irreducible polynomial  vanishing exactly on the determinantal variety of the $v_i$ for which the matrix of rows $t_j,\ j=1,\ldots,n$ is degenerate.
 \begin{lemma}
\label{aMT1} Given $x=(x_1,\ldots,x_n)=(f_1/d,\ldots,f_n/d)$, let $(x)^2=\sum_ix_i^2$. Assume there are two elements $a\in P, b\in(P,P)$ such that $(x)^2+(x,a)+b=0 $ holds identically (in the parameters $v_i$);
then $x$ is a polynomial in the $v_i$.  \end{lemma}
\begin{proof} Substitute  $x_i=f_i/d$ in the quadratic equation and get 
$$d^{-2}(\sum_if_i^2)+d^{-1}\sum_if_ia_i+b=0,\implies \sum_if_i^2 +d \sum_if_ia_i+d^{ 2} b=0.$$
Since $d=d(v)=\det(A(v))$ is irreducible this implies that $d$ divides $\ \sum_if_i^2.$

Since the  $f_i$ are real, for those $v:=(v_1,\ldots,v_m)\in\R^{mn}$ for which   $d(A(v))=0$,  we have $f_i(v)=0,\forall i$; so $f_i$ vanishes on all real solutions of $d(A(v))=0$.  These solutions are Zariski dense, by Lemma \ref{zade},  in the determinantal variety  $d(A(v))=0$. In other words  $f_i(v)$ vanishes on all the  $v$ solutions of $d(A(v))=0$ and thus $d(v)$  divides $f_i(v)$ for all $i$, hence  $x $ is a polynomial.
\end{proof}

\end{proof}

\begin{proof}[Proof of Theorem \ref{aMT}]   Once we fix a root we have that  the variety $R_A$ is the set of solutions of a  system of $\geq n+1$ linear and quadratic equations in the variables $x,v_i$. We are assuming, by Proposition \ref{codim},  that  we have a solution $x=F(v)$ which is a   polynomial in $v_1,\ldots,v_m$.  We now can apply Lemma \ref{spegra}.\end{proof}

\subsection{Proof of Theorem \ref{sunto}\label{suntoS}} 
\begin{enumerate}
\item  Assume by contradiction that there is  a connected subgraph $A$ of the graph $\GamaG$  with $n+2$  vertices and all black edges.  
Then $A$ is the geometric realization of a possible combinatorial graph $\GA$ with $n+1$ non--zero black vertices which, by Theorem \ref{ridma} must be independent. By  Theorem  \ref{aMT}  we have that $A$ is contained in the special component and we have a contradiction.
\item Such a component  must contain an integral point in one of the spheres $S_{\ell}.$  Suppose by contradiction that there is   a connected subgraph of this component with $2n+1$ vertices, denote it by $A$. By hypothesis $A$ can be lifted, let $\GA$ be its combinatorial graph. 
We can have at most $n$ black and $n$ red vertices  otherwise the graph has rank $\geq n+1$, but if we have exactly $n$ black vertices and at least one red vertex  we also have rank $\geq n+1$  since the red vectors are linearly independent of the black similarly for red.  So we can have at most $n-1$  black and $n$ red vertices  for a total (including the root) of $2n$ vertices.

\item We put in the same family two components whose combinatorial graphs are equivalent.  

There are  only finitely many possible combinatorial  graphs with at most $2n$ vertices. A family is formed by  the geometric realization of one representative for each equivalence class of equivalent combinatorial  graphs. This automatically chooses a root.  

\item Take a marked graph $\GA$ with $k+1< n+1$  vertices and all black edges. 
Call $A$  a realization of  $\GA$ and let $x$ be the root.  Any other realization  $A'$  has a corresponding root $x'$  such that $x-x'$  is a vector orthogonal to all the $\pi(a_i)$, where  $a_i$ are the non--zero vertices.  By Costraint \ref{co8},  the $\pi(a_i)$ are all independent. Conversely if $x'$ is a  point in $\R^n$ such that $x'-x$ is a vector orthogonal to all the $\pi(a_i)$, then it solves the same equations as $x$. Hence it is a vertex of a  connected component which contains a translate of $A$ and can only be bigger.  If the component corresponds to a bigger graph, the point  $x'$   solves some further independent linear equations, with respect to $x$, and hence $x'$ belongs to a lower dimensional affine subspace determined by the  bigger graph, since these graphs are finitely many this completes the picture.
\item This is the content of Theorems \ref{ridma} and \ref{aMT}.\end{enumerate}

\begin{definition}
Let $\BU_n=\BU_{n,m,q}$ be the set of non-equivalent combinatorial graphs for a given dimension $n$.
\end{definition}
 Each $\GA\in \BU_{n}$ has realizations in $\Gama$ and the choice of a representative in the equivalence class fixes a root in each component of $\Gama$.

\begin{example}\label{apd}The set $\BU_1$ is simply the set of graphs with a single edge: $(0,+), (\ell, 1+\eta(\ell))$.

For $n=2$ we have $\BU_1$  and  the complete graphs with three vertices $(0,+), (\ell, 1+\eta(\ell)),(\ell', 1+\eta(\ell'))$.  We should consider graphs with up to $2n=4$ vertices which are of rank $\leq 2$ and such that the non--zero vertices  of different colors are dependent. A direct computation shows that no such graphs exist. 
\end{example}
\begin{remark}
One might think at this point that for any $n$ the set $\BU_n$ is only made of graphs with at most $n+1$ vertices which are affinely independent. However this conjecture seems quite hard to prove, it is true but requires a lengthy proof  for $q=1$, for general $q$ it seems quite hard to verify even in dimension $n=3$ and indeed it may not be true.

\end{remark}

\subsection{Proof of Theorem \ref{oneone}\label{oneoneS}}   Once we have ensured that no graphs with more than $2n$ vertices exist we can apply Proposition \ref{coqe2} and Corollary \ref{cabel}. This gives an isomorphism between the components of $\Gama$ and those of $\Lambda_S$.  More precisely for each  family of components we choose one   $A\in\Gama$ and we also choose a root.

By translation this also determines a root for all other components in the same family.  With these choices   we can associate to $A$ the combinatorial graph $\GA\in \BU_n$ of which it is a realization, see Definition \ref{cogrA}. Let $x$ be the vertex  associated to $(0,+)$. Then we obtain all the components in $\Lambda_S$ over $A$ by right translation with all the elements $(a,\pm)$ such that $x=-\pi(a)$.

To establish the isomorphism with the components of $\Gamati$ we make sure   that two conjugate blocks are disjoint, i.e. that a pair $z_k,\bar z_k$ is never in the same block of $ad(N)$.  This would correspond to a loop in the geometric graph which is not a loop in $\Lambda_S$, which is excluded by Constraint \ref{co5}. 
\begin{corollary}  If the $v_i$ are generic,
in the projection map $\Gamati\to \Gama$  the preimage of a connected component of  $ \Gama$ is the union of two disjoint and conjugate connected components of $\Gamati$.
\end{corollary}
 Each $\GA\in \BU_{n}$ has realizations in $\Gama$ and the choice of a representative in the equivalence class fixes a root in each component of $\Gama$.
  For all  $k\in S^c$ set $x(k):=x(A)$ to be the root of the component $A$  of $\Gama$ to which $k$ belongs. By Corollary \ref{cabel} and Formula \eqref{lasig}:
 \begin{lemma}\label{illift}
Each component $A$ can be lifted defining in a compatible way elements $g(k)$  so that  $k=g(k)x(A),\ g(k)=(L(k),\sigma(k))$ and if $k_1,k_2$ are joined by an edge marked $\ell\in G$ we have $g(k_2)=\ell g(k_1)$.
\end{lemma}   
Clearly if $A$ is a realization of $\GA$ then $(L(k),\sigma(k))$ is just  the vertex of $\GA$ which is identified with  $k$ in the isomorphism between $A$ and $\GA$.
\begin{example}
We consider the component $A_{k_1}$ of Example \ref{EGG} (which exists provided that $n> 2$).  This component is the realization of  the combinatorial graph $\GA$ of Example \ref{EGG2}. Hence:
\begin{equation}\label{glielle}
g_{k_1}=(0,+)\,,\quad g_{k_2}=(e_3-e_1,+)\,,\quad g_{k_3}=(-e_1-e_2-2e_3,+)\,,\quad g_{k_4}=(-e_1-e_2,-)\,.
\end{equation} 
\end{example}

\section{Proof of Theorem   \ref{teo1} \label{lapro}}

\subsubsection{Reduction to constant coefficients\label{recoc}}
We think of $y=(y_1,\ldots,y_m),\ x=(x_1,\ldots,x_m)$ as vectors so that, given $a=\sum_in_ie_i\in{\Z^m}$ we have $a\cdot x:=\sum_in_ix_i, \, a\cdot dx:=\sum_in_idx_i=d(a\cdot x)$.  Furthermore $dy\wedge dx=\sum_idy_i\wedge dx_i$.

Before proving the Theorem in general we show how to reduce to constant coefficient a single block. As  usual we use the graphs in Example \ref{EGG}. 
\begin{example}\label{redEGG}
Consider for $q=1$, the Hamiltonian:
$$ (\ome(\xi),y)+ \sum_{i=1}^4 |k_i|^2 |z_{k_i}|^2+4\sqrt{\xi_1\xi_3} e^{\ii(x_1-x_3)}z_{k_1}\bar z_{k_2}+  4\sqrt{\xi_2\xi_3} e^{\ii(x_2-x_3)} z_{k_2} \bar z_{k_3}+ $$ $$4\sqrt{\xi_1\xi_2} e^{-\ii(x_1+x_2)} z_{k_2}   z_{k_4}+  4\sqrt{\xi_1\xi_3} e^{-\ii(x_1-x_3)}\bar z_{k_1} z_{k_2}+  4\sqrt{\xi_2\xi_3} e^{-\ii(x_2-x_3)} \bar z_{k_2} z_{k_3}+ 4\sqrt{\xi_1\xi_2} e^{ \ii(x_1+x_2)} \bar z_{k_2} \bar  z_{k_4},$$
the terms depending on $z,\bar z$ are those of Formula \eqref{hamex}. We apply the change of variables:
$$z_{k_i}= e^{-\ii L(k_i).x}z_{k_i}' ,\ y=y'+\sum_{i=1}^4  L(k_i)  |z_{k_i}'|^2,\ x=x', $$ where $L(k_i)$ are defined in Lemma \ref{illift} and given in formula \eqref{glielle}:
$L(k_1)=0,$ $L(k_2)= e_3-e_1$,  $L(k_3)= -e_1-e_2-2e_3$, $L(k_4)=-e_1-e_2$. A direct check shows that this change of variables is symplectic and that the Hamiltonian in the new variables is:
\begin{equation}\label{redH1}
(\ome(\xi),y' )+ \sum_{i=1}^4 \Big(\ome_0,L(k_i))+|k_i|^2 \Big)|z_{k_i}'|^2 +\tilde{\mathcal Q}
\end{equation} where $\ome(\xi)= \ome_0-2\xi$, and:
 $$\tilde{\mathcal Q}=-2\sum_{i=1}^4(\xi,L(k_i))|z_{k_i}'|^2+  4\sqrt{\xi_1\xi_3} z_{k_1}'\bar z_{k_2}'+  4\sqrt{\xi_2\xi_3} z_{k_2}' \bar z_{k_3}'+ $$ $$4\sqrt{\xi_1\xi_2}  z_{k_2}'   z_{k_4}'+  4\sqrt{\xi_1\xi_3} \bar z_{k_1}' z_{k_2}'+  4\sqrt{\xi_2\xi_3}  \bar z_{k_2}' z_{k_3}'+ 4\sqrt{\xi_1\xi_2} \bar z_{k_2}' \bar  z_{k_4}' .$$

\end{example}

  Theorem \ref{teo1} is contained in the following, more precise, proposition:
  \begin{proposition}\label{reteo} i)\  The equations 
\begin{equation}\label{labella}
  z_k= e^{-\ii L(k).x}z_k' ,\ y=y'+\sum_{k\in S^c}  L(k)  |z_k'|^2,\ x=x'.
\end{equation} define  a  symplectic change  of variables $\Psi^{(3)}:\,D(s,r/2)\to D(s,r)$,  which preserves the spaces $V^{i}$.

{\em We denote by  $W=$ diag$(\{e^{\ii L(k).x}\}_{k\in S^c},\{ e^{-\ii L(k).x}\}_{k\in S^c})$,  the matrix of $\Psi^{(3)}$ on   $w$.} \smallskip 

ii)\  The Hamiltonian $H$ in the new variables is
$N'+P \circ \Psi^{(3)}$
where 
\begin{equation}\label{Nc}
N'=N\circ\Psi^{(3)}:= (\omega(\xi),y')+\sum_{k\in S^c}\Big(|k|^2+\big(\ome(\xi),L(k)\big) \Big)|z'_k|^2+\mathcal Q'(w'),
\end{equation}

and  \begin{equation}\label{comfor1}
\mathcal Q'(w'):= \mathcal Q(x,w' W)\equiv \mathcal Q\circ \Psi^{(3)}= \ \sum_{\ell\in X_q^0} \!\!c(\ell)\!\!\sum_{(h,k)\in \mathcal P_\ell }\!\!\!\!z'_h\bar z'_k
+ \sum_{\ell\in X_q^{-2}}\!\!\!\!c(\ell) \!\!\!\! \sum_{\{h,k\}\in \mathcal P_\ell }\!\!\!\![z'_h  z'_k
 + \bar z'_h\bar z'_k], \end{equation} is independent of $x$.\end{proposition}

iii)   $\tilde P:= P \circ \Psi^{(3)}$ satisfies the bounds of Theorem \ref{teo1}, {  iv)}.
 \begin{proof} {\it i)}
By definition $|L(k)|\leq 4nq$ for all $k$. 
Since $$ \sup_{D(s,r/2)}| w'|_{a,p}\le e^{ C s}|w|_{a,p}\le e^{ C s}r/2 \le r$$ for $s$ small enough the transformation is well defined from $D(s,r/2)$ to $D(s,r)$. It is symplectic because: 
$$dy\wedge dx +\ii dz\wedge d\bar z=dy'\wedge dx'- \sum_k(L(k)\cdot dx' )\wedge d(|z_k|^2) +$$ $$\ii dz'\wedge d\bar z'  + \sum_kd(L(k). x' )\wedge(z'_k d\bar z'_k+\bar z'_k d z'_k)=dy'\wedge dx'+\ii dz'\wedge d\bar z'.$$
Finally it preserves the spaces $V^{i }$ since it is linear in the variables $w$ which have degree $1$ and in $y,|z_k|^2$ of degree 2. In fact it maps a space $V^{i ,j}$ into $\sum_{h=0}^iV^{i-h ,j+2h}$. \smallskip

 {\it ii)}\  We   simply substitute the new variables in the Hamiltonian, we obtain that  \begin{equation}
\label{oprim}(\ome(\xi),y)+\sum_k|k|^2|z_k|^2=(\ome(\xi),y')+\sum_k(\ome(\xi),L(k))  |z_k'|^2+\sum_k|k|^2|z_k'|^2.
\end{equation} 
   By definition of $W$ we have $\mathcal Q(x,w)\circ \Psi^{(3)}=\mathcal Q(x,w' W)$.
   Formula \eqref{comfor1}  follows from Lemma \ref{illift}. In fact we   substitute in Formula \eqref{laquadra} $  z_k= e^{-\ii L(k).x}z_k'$ then if $\ell\in X_q^0,\ (h,k)\in \mathcal P_\ell $ we have $$e^{ \ii(\ell,x)} z_h\bar z_k=e^{ \ii(\ell,x)}  e^{-\ii L(h).x}z_h' e^{ \ii L(k).x} \bar z_k'$$ and, by Formula \eqref{pasl}  we have $\ell-L(h)+L(k)=0$.  Similarly when 
   $\ell\in X_q^{-2},\ \{h,k\}\in \mathcal P_\ell $ we have $$e^{ \ii(\ell,x)} z_h  z_k=e^{ \ii(\ell,x)}  e^{-\ii L(h).x}z_h' e^{- \ii L(k).x}  z_k'$$ and, by Formula \eqref{pasl}  we have $\ell-L(h)-L(k)=0.$
 
 {\it iii)}\  Let us  prove the bounds. We notice that  the total degree $2i+j$ is the same in the two sets of variables.  Moreover $\Psi^{(3)}$ is independent of $\xi$, hence $P\circ  \Psi^{(3)}$ has the same order as $P$, see section \ref{lestpe}, and the bounds follow by Proposition \ref{gliord}.
     \end{proof}
\begin{remark} It is possible to choose also infinite sets of $v_i$  so that the  change of variables is still convergent in a ball. For this it is enough  to impose a reasonable growth  to $|v_i|$ as $i\to\infty$.
\end{remark} 

\subsection{Definitions of $\tilde\Ome_k,\ \tilde{\mathcal Q}$}  From Formula \eqref{laff} we have \begin{equation}
\label{laff1} N'=K'+(\nabla_\xi  A_{q+1}(\xi)-(q+1)^2A_{q }(\xi)\underline 1,y'+\sum_k L(k)|z'_k|^2)+\mathcal Q'(w'),\end{equation}$$  K'=( \ome_0,y')+\sum_k(( \ome_0,L(k))+ |k|^2) |z'_k|^2.
$$
We  set
\begin{equation}\label{tilq} \tilde\Ome_k= ( \ome_0,L(k))+ |k|^2\,,\quad \tilde{\mathcal Q}:=\mathcal Q'(w' )+\sum_k(\nabla_\xi  A_{q+1}(\xi)-(q+1)^2A_{q }(\xi)\underline 1,  L(k))|z'_k|^2 
\end{equation} and remark that the $\tilde\Ome_k$ are integers and the coefficients of the quadratic form $\tilde{\mathcal Q}$ are homogeneous in the variables $\xi$ of degree $q$.

We can group  $\tilde{\mathcal Q}=\sum_A \tilde{\mathcal Q}_A$ where the sum runs over all blocks $A\in \Gama$ and $\tilde{\mathcal Q}_A$ involves only  the variables $z'_k,\bar z'_k$ with $k$  appearing in the block. We now use the graph language.  Having made the change of variables we should really introduce a new graph  $\Gamati'$ expressing  the non--zero entries of $Q$ in the basis $z'$. In fact  by Remark \ref{ilgrang} this is just a subgraph of  that larger graph but it is also clearly isomorphic to  $\Gamati$ although the matrix entries have changed, so by abuse of notation we still denote it by the same symbol $\Gamati$.  

Take a block $A\in \Gama$  and let $\tilde A_{\pm}$ be the corresponding disjoint conjugate components in $\Gamati$ (by convention, in $\tilde A_+$ the root $x$ corresponds to $z_x$, while in $\tilde A_-$ the root $x$ corresponds to $\bar z_x$).  
\begin{remark}\label{scala}
 1.\quad $ \tilde{\mathcal Q}_A$ is a Hamiltonian on the symplectic space  $W_A$ with basis  $(z'_k,\bar z'_k)$, $k$ running over the vertices of $A$.
  
  2. \quad We denote the vertices in each $\tilde A_\pm$ by $Z_A$ and $\bar Z_A$  resp. . The variables 
  $Z_A$ and $\bar Z_A$ form the bases of two Lagrangian subspaces\footnote{Notice that the two bases  $\ii Z_A$ and $\bar Z_A$ are  dual bases only when $A$ does not contain red edges. Indeed, in general, the duality matrix is diagonal with elements $\pm 1$.}  decomposing $W_A$.
  
  3.\quad Both $K'$ and $(\nabla_\xi  A_{q+1}(\xi)-(q+1)^2A_{q }(\xi)\underline 1,y')$  in $N'$ commute  with $\mathcal Q(x, w' W)$ hence $ad(K'+(\nabla_\xi  A_{q+1}(\xi)-(q+1)^2A_{q }(\xi)\underline 1,y'))$  takes a scalar value on any given block  $ Z_A$.
\end{remark}
\subsubsection{The matrix blocks of $\tilde{\mathcal Q}$}
Set  $ \ii Q'_A $ to be the matrix of  $ad({\mathcal Q}'_A)$ and $ \ii D'_A $ to be the (diagonal) matrix of  $$ad(\sum_k(\nabla_\xi  A_{q+1}(\xi)-(q+1)^2A_{q }(\xi)\underline 1,  L(k))|z'_k|^2)$$ in the geometric basis $z'_k,\bar z'_k$ with $k\in S^c$.  Clearly the matrix representation of $\tilde{\mathcal Q}_A$  is $\tilde{ Q}_A=  Q'_A +D'_A$. Moreover, by definition of $\Gamati'$,  we have $Q'_A= Q'_{\tilde A_+}\oplus Q'_{\tilde A_-}$.

Given two vertices  $a\neq b\in \tilde A_+$  we have, from Formula \eqref{laquadra}, that the matrix element $Q'_{a,b}$ is non zero if and only if they are joined by an edge $\ell$ and then $Q'_{a,b}=c(\ell)$ if $b=z'_k$ or $Q'_{a,b}=-c(\ell)$ if $b=\bar z'_k$. The element $c(\ell)$  is   described in Formula   \eqref{laquadra}. 
On the diagonal we have $Q'_{a,a}=0$ while the $D'_{a,a}= (\nabla_\xi  A_{q+1}(\xi)- (q+1)^2A_{q }(\xi)\underline 1,L(k))$ if $a= z'_k$ or   $- (\nabla_\xi  A_{q+1}(\xi)- (q+1)^2A_{q }(\xi)\underline 1,L(k))$ if $a= \bar z'_k$ (cf. \eqref{lzpb}).\marginpar{cambiati 2 segni}  The same rules hold for vertices $a\neq b\in \tilde A_-$ and one easily verifies that \begin{equation}
\label{lopp}\tilde Q_{\tilde A_-}=- \tilde Q_{\tilde A_+}.
\end{equation}

By definition $L(k)$ depends only on the combinatorial graph $\GA$ of which $A$ is a realization, therefore the matrix $\tilde{ Q}_A$ depends only on the combinatorial block $\GA$.
 
In order to stress this point we write $\tilde Q_A\equiv C_\GA= C_{\GA,+}\oplus C_{\GA,-}$, with $C_{\GA,-}=-C_{\GA,+}$.\begin{lemma}\label{autoag}
For all combinatorial blocks $\GA$ which do not contain red edges, the matrix $C_{\GA,+}$ is self--adjoint for all $\xi\in A_\ro$. If $\GA$ contains red edges then each vertex has a sign. This defines a diagonal matrix of signs $\Sigma_\GA$ and  $C_{\GA,+}$ is self--adjoint with respect to the indefinite form defined by $\Sigma_\GA$.
\end{lemma}
\begin{proof}
This is essentially a consequence of the fact that $H$ is real, but it follows immediately from the definition of $c(\ell)$ and the explicit formulas for $C_{\GA,+}$ given above. \end{proof}
\begin{definition}\label{bfnu2}
We denote by $\mathcal M $ the finite list of matrices $\{C_\GA\}_{\GA}$ where $\GA$ runs over all the non--equivalent combinatorial blocks   with chosen root at (0,+) which contain only the indices $1,\ldots, 4qn$.  We denote by $\mathcal M(\xi) $ the finite (and independent of $m$) list of matrices $\{C_\GA\}_{\GA}$ where $\GA$ runs over all the non--equivalent combinatorial blocks   with chosen root at 0.
\end{definition}
\begin{corollary}\label{bfnu}
The matrices $\mathcal M(\xi)$ are obtained from  the matrices $\mathcal M$  by permuting the variables   $\{\xi_1,\dots,\xi_m\}$.
\end{corollary}\begin{proof}
  In fact we have finitely many graphs with at most $2n$ vertices, the indices appearing in the edges are the ones appearing    on a maximal tree with at most $2n$ edges. On  each edge they involve at most $2q$ indices and so we have the a priori bound number $ 4qn$ of indices which, up to symmetry, can be taken to be $1,\ldots, 4qn$.
\end{proof}
  \begin{example}\label{less} We describe   the block  $C_{\GA,+}$ for the graph $\GA$ consisting of a unique  edge $\ell$.  Recall that  $  \eta(\ell)=\sum_i \ell_i $ is either $0$ or $-2$ so $1+\eta(\ell)=1, -1$ respectively. Set $$c(\ell)=(q+1)a(\ell) \,, 
\quad \nabla_\xi  A_{q+1}(\xi).\ell -(q+1)^2A_{q }(\xi)\eta(\ell)=  (q+1)(1+\eta(\ell)) b(\ell)\,, $$ we then have:
  \begin{equation}\label{24}C_{\GA,+}=(q+1)\begin{vmatrix}
\ \ 0&  (1+\eta(\ell))a(\ell)\\&& \\ a(\ell)&  \ \ b(\ell) 
\end{vmatrix}  \end{equation} 
 For $q=1$ one gets
$$  
\GA_1= \xymatrix{  (0,+)\ar@{->}[r]^{1,2\quad }&  (e_2- e_1,+)  &  \GA_2=\ (0,+)  \ar@{=}[r]^{1,2\quad \quad } &  (-e_1-e_2,-)}, 
$$  
  \begin{equation}\label{22}C_{\GA_1,+}=2\begin{vmatrix}
0& 2\sqrt{\xi_1\xi_2}\\&&&\\ 2\sqrt{\xi_1\xi_2}& \xi_1- \xi_2  
\end{vmatrix},\quad C_{\GA_2,+}=2\begin{vmatrix}
0&  -2\sqrt{\xi_1\xi_2}\\&&&\\  2\sqrt{\xi_1\xi_2}& -\xi_1-\xi_2  
\end{vmatrix}  \end{equation}

For $q=1$ consider the component $\GA$ of Formula \ref{EGA2} we obtain
$$C_{\GA,+}=2\begin{pmatrix}
 0 & 2\sqrt{\xi_1 \xi_3} & 0 &0 \\ 2\sqrt{\xi_1 \xi_3} &\xi_1-\xi_3& 2\sqrt{\xi_2\xi_3} & -2\sqrt{\xi_1\xi_2} \\ 0 &2\sqrt{\xi_2\xi_3}   &  \xi_1+\xi_2-2\xi_3 &0 \\ 0 & 2\sqrt{\xi_1\xi_2}&0 &-\xi_1-\xi_2
\end{pmatrix} $$\marginpar{cambiati 3 segni}the reader can easily verify that in the Hamiltonian \eqref{redH1} $\tilde{\mathcal Q}$ is represented by the matrix above.
\end{example}

\smallskip

 \begin{proof}[Proof of Theorem \ref{teo1}] The change of variables $\Phi_\xi= \Psi^{(1)}\circ \Psi^{(2)}\circ \Psi^{(3)}$.
Item i) follows from Corollary \ref{diffeo}. Item ii) follows from the corresponding item of Proposition \ref{reteo}. Item iii) also follows by item ii) of \ref{reteo}. The set of matrices $\mathcal M$ is defined in Definition \ref{bfnu2}. iv)  follows from the same statement for $P$.\end{proof}
\section{Proof of Proposition \ref{teo2} and Corollary \ref{ficof}}\label{pfiu}

\subsection{The arithmetic constraints\label{arcon}}   We want to show now that in some special cases, on the integer points of the geometric graph we may impose much stronger conditions.
\begin{proposition}\label{ilrido} (i) For $n=1$ and for generic choices of $S$, all the connected components of $\Gama$ are either vertices or single edges.

(ii) For $n=2$,  and  for every $m$  there exist infinitely many choices of generic tangential sites $S=\{v_1,\ldots,v_m\}$ such that, if $A$ is a connect component of the geometric graph  $\Gama$, then $A$ is either a vertex or a single  edge.
\end{proposition}
\begin{proof}
(i) It is proven in Example \ref{apd}.

(ii)This statement  is proved in \cite{GYX} for  $q=1$ by a very direct and lengthy computation. Here we give a more conceptual proof based on estimates of integral points on algebraic curves, valid for all $q$.

The simplest  of such estimates is   that, for all $0<\delta<1$  one can estimate the number of integral points  of  a circle of radius $R$  by $\ll R^\delta$ as $R\to\infty$.

In general Bombieri and Pila prove, in \cite{BP} Theorem 5, that if $C$  is a real absolutely irreducible algebraic curve of degree  $d$ and if    $N>e^d$  the number of integral points in $C$  in the square $[0,N]\times [0,N]$ is bounded by 
$$N^{1/d}\exp(12\sqrt{\log(N)\log \log (N)}).$$
In particular for any $\delta>0$  and $N$ large we have a bound $N^{1/d+\delta}$.

We need a less fine estimate, if the curve is not necessarily absolutely irreducible but contains no lines we still get, by looking at its irreducible factors an estimate of type $N^{1/2+\delta}$ for $N$ large.  We want to use these bounds  for our estimates.

Let us first characterize the sets $x,v_1,\dots,v_m$ such that there is an edge marked $\ell$ with vertex in $x$. We can interpret Formulas \eqref{iperp}--\eqref{sfera} by saying that  $x,v_1,\dots,v_m$ satisfy an equation which is the equation for a sphere in either $x$ (red edge) or one of the $v_j$'s-- here we consider the other variables as parameters.
Suppose now that  $x$ is a vertex of  a graph $U$ with two different edges $\ell_1,\ell_2$. Hence $x$   satisfies the two equations given by these edges.  
\smallskip

{\bf Case $q=1$.} We know that there is an index $i=1,\dots,m$ such that $e_i$ appears in $\ell_1$ but not in $\ell_2$ (otherwise we would have $\ell_1= e_i-e_j$ and $\ell_2= -e_i-e_j$ which does not have a realization in $\Gama$). 

Suppose now that $\ell_2$ is red.
We next claim that if $|v_i|\leq R$ for all $i$ then $|x|< C R$ (where $C$ is a universal constant). Indeed since one of the edges is red then $x$ belongs to the circle of diameter $v_1,v_2$ (we are assuming without loss of generality  that the red edge is $\ell_2=-e_1-e_2$).

  Consider the set
$$A_U:\{v_1,\ldots,v_m,x\}\subset \Z^{2m+2},\ |v_i|\leq R,\quad x\quad\text{solves the equations given by}\  U:=\ell_1,\ell_2.$$  We claim that $|A_U|\ll R^{2m-1+\delta}$. Without loss of generality we may suppose that $\ell_1$ depends non trivially on $e_3$. 

We first use the equation given by $\ell_1$ to express $v_{3,1}$ in terms of the other parameters. This of course gives at most two solutions. 
Then the equation for $\ell_2$ is a circle in $x$ with diameter $\leq 2R$. 

Thus $\cup_UA_U$ has $\ll \binom{m}{2}^2 R^{2m-1+\delta}$  elements.  When $R$ is large $\ll \binom{m}{2}^2 R^{2m-1+\delta}<    R^{2m }$ thus the projection of this set on the first $m$ coordinates is not surjective and thus any point outside this projection is a set of tangential sites satisfying the condition of our proposition.
  
  To treat the case of $\ell_1,\ell_2$ both black we need to ensure that $|x|< CR$ provided that the $|v_i|<R$. We compute $x$ by Cramer's rule, the denominator is $\pi(\ell_1)\wedge  \pi(\ell_2)$ while the numerator is bounded proportionally to $|\pi(\ell_1)||\pi(\ell_2)|R$.  To obtain the desired bound we   restrict $v_1,\dots,v_m$ to the sector where $|\pi(\ell_1)\wedge  \pi(\ell_2) |\geq c(m) |\pi(\ell_1)|| \pi(\ell_2)|$ for all choices of $\ell_1,\ell_2\in X^0_1$, here $c(m)$ is some constant depending only on $m$. The set of $v_i$'s which satisfy this constraint and have $|v_i|<R$ is still of the order of $R^{2m}$.
  
  As done before, we use the equation given by $\ell_1$ to express $v_{3,1}$ in terms of the other parameters. 
Then the equation for $\ell_2$ is a circle in one of the variables $v_1,v_2$ with diameter $\leq 2 C R$. 
\smallskip

 {\bf Case $q>1$.} It is no longer always  true that there exists an index $i$ such that $e_i$ appears in $\ell_1$ but not in $\ell_2$.  If this restriction is satisfied then the previous proof applies, otherwise we claim that we can apply Theorem 5  in the paper of Bombieri and Pila \cite{BP}. In fact look at an equation $$|(x,y)|^2+((x,y),\sum_im_iv_i)=-1/2(|\sum_im_iv_i|^2+\sum_im_i|v_i|^2)$$ or equivalently
 $$ |2(x,y)+\sum_im_iv_i|^2:= |(x',y')|^2= -(|\sum_im_iv_i|^2+2\sum_im_i|v_i|^2).  $$
 since $\sum_im_i=-2$, either  $\sum_im_iv_i=-e_a-e_b$  (where the previous arguments apply)or  there is then an index  $j$ with $m_j>0$, write the equation  in terms of  $ (x' ,y'),z =v_{j,1}$ considering the other coordinates as parameters.  This defines an ellipsoid
$$(x')^2+(y')^2+(m_j^2+2m_j)z^2+az+b=0 $$ which, if the remaining coordinates of the   $v_j$ are bounded by some $R$  is contained in a cube  $[-CR,CR]^3$  with $C$  some fixed integer depending on the $m_i$. We now intersect with the other equation and claim that we have an absolutely irreducible curve, to its projection on one of the coordinate planes we apply the Theorem of Bombieri  and Pila.  The other equation is of the form $(\vec x,\sum_in_iv_i)=1/2(|\sum_in_iv_i|^2+\sum_in_i|v_i|^2)  $ if the other edge is black  or $(\vec x,\sum_i(n_i-m_i)v_i)=-1/2(|\sum_in_iv_i|^2+\sum_in_i|v_i|^2) +1/2(|\sum_im_iv_i|^2+\sum_im_i|v_i|^2)  $ if the edge is red.  
The equation is  of the form $n_jx'z+cz+dy'+ex'+fz^2+g=0$, where $d=\sum_i n_i v_{i,2}$.  If $d\neq 0$   we solve it for $y$ and see that we have a  plane quartic, otherwise we project it to the plane $y',z$ still getting  a  plane quartic. In either case the quartic does not contain a real line since its real points are bounded, the estimate on its integral points follows from Theorem of Bombieri  and Pila. 

\bigskip

Two black edges.  $\sum_im_ie_i,\sum_jn_je_j,\ \sum_im_i=\sum_jn_j=0$ the equations are  $$(\vec x,\sum_im_iv_i)=1/2(|\sum_im_iv_i|^2+\sum_im_i|v_i|^2)  ,  (\vec x,\sum_in_iv_i)=1/2(|\sum_in_iv_i|^2+\sum_in_i|v_i|^2)  $$ say that $m_1\neq n_1$   consider all the $v_i,i>1$ and $y$ as parameters, let $z,w$ denote the coordinates of $v_1$ which we consider as variables,  so write  the equations as
$$ x(a +m_1z) =1/2(m_1^2+m_1)[z^2+w^2]+m_1zw+b z+cw+d $$
$$ x(a' +n_1z) =1/2(n_1^2+n_1)[z^2+w^2]+n_1zw+b' z+c'w+d' $$ we may assume $n_1\neq 0$ otherwise we are in the previous case of an index appearing in $\ell_1$ and not in $\ell_2$.  Project to the $z,w$ plane and we see that we obtain the cubic
$$ [ 1/2(m_1^2+m_1)[z^2+w^2]+m_1zw+b z+cw+d](a' +n_1z) $$$$=(a +m_1z)[1/2 (n_1^2+n_1)[z^2+w^2]+n_1zw+b' z+c'w+d']$$  of equation
$$Az[z^2+w^2]+ Bz^2+Cw^2+Dzw+Ez+Fw+G=0$$ with $A=1/2 n_1m_1(m_1 - n_1)\neq 0.  $
Let us show  that this is absolutely irreducible.  Otherwise it factors through a linear and a quadratic term, and  we can always assume that the linear term is defined over $\R$ since with any factor we also have the conjugate factor.  This implies that if there is a factorization it is of the form
$$(Az+K)(z^2+w^2+Mz+Nw+P) $$ which implies that 
$$ Az ( Mz+Nw+P)  +K (z^2+w^2+ Mz+Nw+P)= Bz^2+Cw^2+Dzw+Ez+Fw+G$$ 
$$ AM+K=B,\ K=C,\ AN=D,\   AP+KM=E,\  KN=F,\  KP=G$$in particular $C=K$.
Now $C=(m_1^2+m_1)a'-(n_1^2+n_1)a$ where $a=\sum_{i>1}m_iv_{i,1},\ a'=\sum_{i>1}n_iv_{i,1}$.  
 
We have $AF=  CD$  and we claim that this imposes a  non trivial restriction to the parameters, thus for a large set of parameters we can apply the the method.

We have $F=ca'-c'a$  and $c=-m_1y+m_1\sum_{j>1}m_jv_{j,2},\ c'=-n_1y+n_1\sum_{j>1}n_jv_{j,2}$, while $D=a'm_1+cn_1-an_1-c'm_1$. We see that in $D$  the variable $y$ disappears while in $F$ it appears linearly with coefficient  $-m_1a'+n_1a=\sum_{i>1}(n_1m_i-m_1n_i)v_{i,1}$. We  cannot have $(n_1m_i-m_1n_i)=0$ for all $i>1$ unless $\ell_2$ is a multiple of $\ell_1$.  This case though we have excluded in Theorem \ref{ridma}.  \smallskip

We conclude that, for any $\delta>0$, the number of integral points are less that a constant (dependent only on $\delta$)  times $R^{1/2+\delta}$.  At this point the proof is identical to the previous argument. \end{proof}

 We denote the sets $S$ which do not contribute to any $A_U$ as {\em arithmetically generic} and think of the condition $\not\exists x\in\Z^2: (v_1,\dots,v_m,x)\in \cup_U A_U$ as an {\em arithmetic} constraint.
 
\begin{proposition}\label{blocchini}
Under the geometric and arithmetic constraint for $n=1$ or   $n=2$   all the non--diagonal blocks in $\tilde{ Q}$ are  two by two and given by Formula \eqref{24}.
\end{proposition}
 \begin{remark}

It  is unclear  what happens in higher dimension. One can use the same argument to exclude  graphs  of rank equal to the dimension, so   Dimension 3  could still behave in a special way.  On the other hand, for $q=1$  there is a different method using the second Melnikov condition  which we shall discuss elsewhere.

\end{remark}

\subsubsection{Real roots}\label{rero}  We ask if there are regions in the parameters $\xi$ where all blocks have real roots.  The issue is only for graphs containing red edges and the region is described by a finite set of inequalities given by Sylvester's Theory.

We discuss here the case in which all  graphs containing a red edge reduce to  this  edge.  As remarked in \S \ref{arcon}, one can  have this case in dimension $n=1,2$  for all $q$. 

The matrix associated to the graph $\GA$ consisting of a unique red edge $\ell$ is given by Formula \eqref{24} where:
\begin{equation}
\label{perco}a(\ell)=q\xi^{\frac{\ell^++\ell^-}{2}}\sum_{\alpha\in\N^m\atop{|\alpha+\ell^+|_1=q-1}}\binom{q+1}{\ell^-+\alpha}\binom{q-1}{ \ell^++\alpha}  \xi_i^\alpha,\ 
\end{equation}  $$\nabla_\xi  A_{q+1}(\xi).\ell -(q+1)^2A_{q }(\xi)\eta(\ell)= -(q+1)b(\ell) $$
 The characteristic polynomial of $C_\GA/(q+1)$   is $t(t-b(\ell))+a(\ell)^2$ with discriminant $b(\ell)^2-4a(\ell)^2$.
We want to show that there is a non empty open region in our parameter space where the roots of all these polynomials are distinct real, that is where all    these discriminants are strictly positive. For this, using the usual lexicographical order, let us compute the leading terms of all these polynomials.  Apply  Formula \eqref{perco} letting $d(\ell):=q-1-|\ell^+|_1$ we see that    the leading monomial of  $-4a(\ell)^2$ is   $\xi_1^{d(\ell)}\xi^{  \ell^++\ell^- }$. As for $b(\ell) $  it has the monomial $\xi_1^q$  appearing with  the following coefficient. In $\pd{A_{q+1}(\xi)}{\xi_1}$ 
the coefficient of $\xi_1^q$ comes from $\pd{ \xi_1^{q+1}}{\xi_1}=(q+1)\xi_1^q$  while for $i>1$ in $\pd{A_{q+1}(\xi)}{\xi_i}$ 
the coefficient of $\xi_1^q$ comes from $\pd{ (q+1)^2\xi_i\xi_1^{q }}{\xi_i}=(q+1)^2\xi_1^q$. If $\ell=(\ell_1,\ldots,\ell_m)$   the coefficient of  $\xi_1^q$ in $b(\ell) $ is $\ell_1+(q+1)\sum_{i>1}\ell_i$. Since $\sum_i\ell_i=-2$  we finally have $\ell_1+(q+1)\sum_{i>1}\ell_i= \ell_1-(q+1)\ell_1-2(q+1) =- q \ell_1-2(q+1)=-q(\ell_1+2)-2.$

In the term $-(q+1) A_{q }(\xi)\eta(\ell)=2(q+1) A_{q }(\xi)$ the  monomial  appears with coefficient  $2(q+1) $.
Thus we get a total contribution of  $-q\ell_1$. Thus if $\ell_1\neq 0$ the leading monomial of the discriminant is $\xi_1^{2q}$ with positive coefficient.  If $\ell_1=0$  let us look at the coefficient of $\xi_1^{q-1}$ in  $\pd{A_{q+1}(\xi)}{\xi_i},\ i\neq 1$. This comes from the terms  $\frac{q(q+1)}{2}\xi_1^{q-1}\xi_i^2$  giving $ q(q+1) \xi_1^{q-1}\xi_i $ and $ q(q+1) \xi_1^{q-1}\xi_i \xi_j,\ i\neq j\neq 1$  giving $ q(q+1) \xi_1^{q-1}\sum_{j\neq 1,i}\xi_j $.  Together we get $ q(q+\nobreak1) \xi_1^{q-1}\sum_{j\neq 1 }\xi_j $.  When we take the scalar product with $\ell$ we get  thus a total contribution of  $ -2q(q+1) \xi_1^{q-1}\sum_{j\neq 1 }\xi_j $. From $2(q+1) A_{q }(\xi)$ we get the  term $ 2q^2(q+\nobreak1) \xi_1^{q-1}\sum_{j\neq 1 }\xi_j $. Thus we get a leading term  of type $ 2(q^2-q)(q+1) \xi_1^{q-1} \xi_2$ unless $q=1$, in this case we need to do a different argument.

The leading term of  $b(\ell)^2$  is thus $\xi_1^{2q-2} \xi_2^2$ with positive coefficient. This gives the leading term in the discriminant unless $\ell^+=0$ hence $d(\ell)=0$ and $\ell=-2e_2$, but this is not possible by one of our first constraints.

Finally for $q=1$  the discriminants are all of type  $\xi_i^2+\xi_j^2-14\xi_i\xi_j$  so we see that in all cases we can apply the following Lemma.
\smallskip

   Given $j\in \N$ consider the list  $\mathcal M_j $  of monomials of degree $\leq j$ in the variables $\xi_i,\ i=1,\ldots, m$ ordered lexicographically, denote by $A\prec B$ this ordering.
 
Given a positive constant $D$ set
$$\mathcal A_D:=\{\xi\,|\,\xi_i>0,\ A(\xi)>DB(\xi)\},\ \forall B\prec A,\ A,B\in\mathcal M_j  \}.$$
\begin{lemma} \begin{enumerate}
\item  For every $D>0$  the open set $\mathcal A_D$ is non empty.
\item  Consider a list of polynomials $f_i(\xi)$    of degree $\leq j$ and with real coefficients. If,  for each $i$, the coefficient in $f_i$ of the leading monomial  is strictly positive then there is a positive constant $D$ so that in the region $\mathcal A_D$ we have $f_i(\xi)>0$ for all $i$.
\item Under the hypotheses of ii), if all the $f_i$ are homogeneous the non empty open set where $f_i(\xi)>0$ for all $i$ is a cone.\end{enumerate}

\end{lemma}
\begin{proof} {\em i)} Consider  the curve $\xi_i:=t^{(j+1)^{m+1-i}}$,  if $M=\prod_{i=1}^m\xi_i^{h_i}$ we have that  on this curve $M(t)=t^{\sum_i h_i (j+1)^{m+1-i} }$. It is clear that $B=\prod_{i=1}^m\xi_i^{k_i}\prec  A=\prod_{i=1}^m\xi_i^{h_i}$  if and only   if the sequence $(k_1,\ldots  k_m)\prec (h_1,\ldots  h_m)$.  But if $   \sum_ik_i\leq j, \sum_ih_i \leq j,\quad  (k_1,\ldots  k_m)\prec (h_1,\ldots  h_m)$ we have $\sum_i h_i (j+1)^{m+1-i}>\sum_i k_i (j+1)^{m+1-i}$ so that $\lim_{t\to\infty}A(t)/B(t)=\infty.$  For any $D>0$, for large $t$ the curve lies in $\mathcal A_D$.\smallskip

{\em ii)} The leading monomial is the maximum in the lexicographic order. Take a polynomial  $f=aM+ \sum_i^ka_iM_i$  with $M$ leading monomial and $a>0$. We have, in the quadrant $\xi_i>0$, that $aM+ \sum_i^ka_iM_i\geq aM-\sum_i^k|a_i|M_i $. If $\deg(f)\leq j$, in  $\mathcal A_D$ we further have:
$$aM-\sum_i^k|a_i|M_i\geq  \sum_i^k(\frac{a}{k}D-|a_i|)M_i .$$
Since $a>0$ it is enough to choose $D>\max  |a_i|\frac{k}{a} $.

{\em iii)} The set where an homogeneous polynomial is positive is a cone. 
\end{proof}
\begin{proof}[Proof of Theorem \ref{teo2}] This now follows from the previous Lemma and the discussion of the discriminants that we have performed.
\end{proof}

\begin{proof}[Proof of Corollary  \ref{ficof}] 
We use the notations and results of Remark \ref{scala}. We divide $S^c$ in the connected components of $\Gama$ and apply the standard theory of quadratic Hamiltonians to each  geometric block  $A$:$$H_A:=\sum_{k\in A}\tilde\Ome_k |z'_k|^2+ \tilde{\mathcal Q}_{A}(w'), $$ where now all the $\tilde\Ome_k$ in the block are equal in the case of $A$ black while they are opposite in the case of a red edge so that  $ad( \sum_{k\in A}\tilde\Ome_k |z'_k|^2)$ acts as a scalar matrix on the variables $Z_A$.
\smallskip

  If $ \ii\, ad(\tilde{\mathcal Q}_{A}(w))$ is semi--simple with real eigenvalues, it is a standard fact that there exists a real linear symplectic change of variables  
 $\psi_A$ under which 
$$H_A \circ \psi_A= \sum_{k\in A}\bar\Omega_k |z_k|^2, $$  where $\pm\bar\Omega_k$ are the eigenvalues of $\ii ad(H_A)$.

In particular for all geometric blocks $A$ which do not contain a red edge, this property holds for all positive $\xi$ by using Lemma \ref{autoag}.    Formula \eqref{hfin} follows with $\bar{\mathcal Q}= \sum_{A \in {\rm red}}\tilde{\mathcal Q}_A$, here $A \in {\rm red}$ means that $A$ contains red edges, and we have seen that this is a finite set. We have proved ii).

The change of variables $\psi_A:= w_A\to \mathcal L_A(\xi) w_A$ where $w_A= (z_k',\bar z_k')_{k\in A}$ and $\mathcal L_A$ is a matrix which diagonalizes  $\tilde Q_A= \ii ad(\tilde{\mathcal Q}_A)$. Since there are only a finite number of distinct matrices $\tilde Q_A$, we only need a finite number of distinct $\psi_A$.

We have $\bar\Ome_k=\tilde \Ome_k +\lambda_k(\tilde Q_A)$, where $\lambda_k(\tilde Q_A)$ runs over the eigenvalues of $\tilde Q_A$, this proves i). 

 Each change of variables $\psi_A$ is locally analytic (real algebraic) in $\xi$ for those $\xi$ where the  eigenvalues which are not identically equal are distinct. This condition identifies an algebraic hypersurface (where the non--identically equal eigenvalues of a combinatorial block coincide) which we have to remove. 
 The algebraic  hypersurface that we remove is  a cone, so we can remove a conic neighborhood of this hypersurface arbitrarily determined by its intersection with the unit sphere.   Since the choice of the neighborhood on the unit sphere is arbitrary, we easily see that  in the domain $A_\ro$ this is equivalent to removing a tubular neighborhood of order $\ro$. The bounds  iii) follow by homogeneity of the functions.

To be more explicit, from Theorem \ref{teo1} we have decomposed our space as infinite 
 direct   orthogonal sum of  symplectic spaces, each      decomposed explicitly as direct sum of two   Lagrangian subspaces  in duality each stable under $ad(N)$.  
 
 It is a standard fact  that, if for a  given symplectic block the matrix is semisimple with real eigenvalues then on that block there is a symplectic change of variables which makes it diagonal. In fact if the matrix preserves the decomposition into two Lagrangian subspaces, as in our case, we can take the change of variables preserving the two subspaces. 
 
 Under the geometric constraint all blocks relative to only black edges give rise to symmetric matrices for a positive quadratic form which   thus are semisimple with real eigenvalues and can be diagonalize. 
 
 Consider next the cases   in which,  under the arithmetic constraint,  we have   remaining 
   $2\times 2$ blocks associated to red edges. For these  we can apply Theorem \ref{teo2}.  It remains to prove that the global symplectic transformation defines as direct sum of all the ones diagonalizing each block is indeed continuous and preserves the domain. This follows from the fact that, up to a scalar summand,  we have only finitely many types of blocks in $ad(\mathcal Q)$.

\end{proof}
\appendix
 \section{Marked graphs\label{ilG}}
 {\em This section is independent of the previous parts of the paper. It has the only purpose to establish the correct algebraic language.This is a  standard material in Group Theory.}

\subsection{The Cayley  graphs\label{Cg}}
Let $G$ be a group and $X=X^{-1}\subset G$ a subset.\subsubsection{Marked graphs}
\begin{definition}
An $X$--marked graph is an oriented graph $\GA$ such that each oriented edge is marked with an element $x\in X$.
$$\xymatrix{ &a\ar@{->}[r]^{x} &b&  &a\ar@{<-}[r]^{\ x^{-1}} &b&   }$$
We mark the same edge, with opposite orientation, with $x^{-1}$.

A morphism of marked graphs $j:\GA_1\to \GA_2$ is a map between the vertices, which preserves the oriented edges and their markings.

A morphism which is also injective is called an  {\em embedding}.

\end{definition}

Recall that \begin{definition}\label{ilpa}\begin{enumerate}
\item   A  {\em path} $p$ of length $f$, from a vertex $a$ to a vertex $b$  in a graph is a sequence of vertices
$p=\{a=a_0,a_1,\ldots,a_f=b\}$  such that $a_{i-1},a_i$ form an   edge for all $i=1,\ldots, f$.

The vertex $a$ is called the {\em source} and $b$ the {\em target} of the path.
\smallskip

\item    A {\em circuit}  is a  path from a vertex $a$ to itself. 
\smallskip

We always exclude the presence in a path of {\em trivial steps}  that is $a_{i-1}=a_{i+1}$.

\item   A    graph  without circuits is called a {\em tree}.

\item    If we have an oriented path $p:=\{a_0,a_1,\ldots,a_f\}$ marked  $a_{i-1}\stackrel{g_{i }}\rightarrow a_i,\ i=1,\ldots,f$ in an $X$--marked graph, then we set  $g(p):=g_fg_{f-1}\ldots g_1 $.

\item    If $g^2=1$  then an edge marked $g$ has both orientations so we consider it as {\em unoriented}.

\end{enumerate}

 \end{definition} \subsubsection{Cayley graphs}
 A typical way to construct an $X$--marked graph is the following. Consider an action $G\times A\to A$ of $G$ on a set $A$, we then define.
 \begin{definition}[Cayley graph] The graph $A_X$ has as  vertices   the elements of $A$ and, given $a,b\in A$ we join them by an oriented edge $a\stackrel{x}\rightarrow b$, marked $x$, if $b=xa,\ x\in X$.  
 \end{definition}

 If $G$ acts on two sets $A_1$ and $A_2$ and $\pi:A_1\to A_2$ is a map compatible with the $G$ action then $\pi$ is also a morphism of marked graphs.
 
 A special case is obtained when $G$ acts on itself by left (resp. right) multiplication and we have the Cayley graph $G_X^l$ (resp. $G_X^r$).   We  concentrate on $G_X^l$ which we just denote by $G_X$.
One then immediately sees that \begin{lemma}
\label{propp}  If $G$ acts on a set $A$ and $a\in A$ the {\em orbit map}  $g\mapsto ga$ is  compatible with the graph structure. 

The graph $G_X$ is  preserved by right multiplication by elements of $G$, that is if $a,b$ are joined by an edge marked $g$ then also $ ah, bh$ are so joined, for all $h\in G$.

The graphs $G_X^l,\ G_X^r$ are isomorphic  with opposite orientations under the map $g\mapsto g^{-1}$.

 The graph $G_X $ is connected if and only if $X$ generates $G$, otherwise its connected components are the right cosets in $G$ of the subgroup $H$ generated by $X$. 
\end{lemma} 

\medskip

  \begin{definition}\label{mare} Given an   $X$--marked graph $\GA$.
 We say that $\GA$ is {\em compatible} with $G_X$  if it can be embedded $j:\GA\to G_X$ in $G_X$.  
 \end{definition} Note: \  two embeddings of $\GA$ in $G_X$  differ by a right multiplication by an element of $G$.
  
 Let us understand the conditions under which a connected graph $\GA$ is compatible.  Take two vertices $h,k$  in   $\GA$  and join them by a path $p:=k=k_0,k_1,\ldots, k_t=h$. Assume that $k_{i-1},k_{i },\ i=1,\ldots, t$ is marked by the element $g_i\in X$. Then define $g(p):=g_t g_{t-1}\ldots g_1$.  We can fix an element $r\in \GA$ which we call {\em the root}  and lift it for instance to 1.  Given any other element $h\in \GA$ choose a path $p$ from $r$ to $h$ and set   $g_h :=g(p)$.  In order for this to be well defined we need  that if $h$ is joined by two distinct paths $p_1,p_2$ then $g(p_1)=g(p_2)$. In other words
 \begin{lemma}
\label{lacomp} $\GA$ is compatible if and only if  given any circuit $p$ from $r$ to $r$  we  have $g(p)=Id$.
\end{lemma}  If this condition is fulfilled we have the special {\em lift }
  $j:a\mapsto g_a  $ under which $r\mapsto 1$.

In particular suppose that $G$ acts on a set $A$ and $\GA\subset A$ is a connected subgraph  of $A_X$ with $f$  vertices.   Then
\begin{corollary}\label{respro}
A sufficient condition for $\GA$ to be embedded in $G_X$ is that, for any   $a\in \GA$, if an element $g\in G$ is a product $g=x_1x_2\ldots x_d$ of $d\leq f$ elements we have that $ga=a$ implies $g=1$.
\end{corollary}
\section{Proof of Lemma \ref{spegra}}\label{prspegra}
\begin{proof} Consider a graph with $r+1$  vertices and of rank $r\geq 2$. We distinguish the elements $a_i,\ i=1,\ldots, u$ corresponding to black vertices from the $b_j,\ j=1,\ldots, v$  of red vertices, we are assuming that both colors appear. We have $a_i(\underline 1)=0, b_j(\underline 1)=-2$. 
 
 We have the equations
 $$(x,\pi(a_i))=K(a_i),\quad |x|^2+(x,\pi(b_j))=K(b_j)$$
 If the  solution $x$ is polynomial in the $v_i$,   it is linear by a simple degree computation. Since it is also  equivariant under the orthogonal group, it follows that it has the form $x=\sum_sc_sv_s$ for some numbers $c_s$.   Let now $a=-\sum_sc_se_s$ so $x=-\pi(a)$. The fact that the given system of equations is satisfied for all $v_i$ (this since they are now polynomials)  is equivalent to the equations.
  \begin{equation}
-a a_i=C(a_i),\  a^2-ab_j=C(b_j).
\end{equation}By changing root if necessary we    may always assume that there are black vertices different from the root. For such a vertex $a_i\neq 0$ we have   an equation
   $ -2a a_i =a_i^2+a_i^{(2)}$  which implies that $a_i$ divides $a_i^{(2)}$.
   
   If  $a_i=\sum_jp_je_j$ we have that $a_i^{(2)}=\sum_jp_je_j^2$ is an irreducible polynomial unless  $a_i=p(e_h-e_k)$ (recall that $\sum_jp_j=0$).
    Then $-2a-a_i= (e_h+e_k)$ which implies $-2a=(1+p)e_h+(1-p)e_k $, namely $a=\alpha e_h+(-\alpha -1)e_k$, for some $\alpha$.
    
    Now we exploit the fact that $a$ satisfies also all the other equations. If it satisfies another black equation--say with a vertex $a_j$-- by linear independence of the vertices we must have $\alpha =0$ or $\alpha=-1$   and $a= - e_h$ for some $h$. Hence the only case to exclude is 1 black and one or more red equations. For a red equation we have:      $$2a^2-2ab_j= -b_j^2-b_j^{(2)}\iff a^2+(a-b_j)^2=-b_j^{(2)}.$$
By comparing the coefficients of the quadratic terms we see  that $b_j=\sum_l q_l e_l$  cannot have any positive coefficient $q_l$, since $\eta(b_j)=-2$.  Hence we must have $b_j=-e_h-e_k$ for the same $h,k$ appearing in $a$.  Now substitute in the equation
$$  (\alpha e_h+(-\alpha -1)e_k)^2+((\alpha +1)e_h-\alpha e_k)^2=e_h^2+e_k^2 $$  get $\alpha^2+(\alpha+1)^2=1$ with solutions $\alpha=0,-1$ hence $x=v_h,v_k$ as desired.  
\end{proof}
\section{Determinantal varieties}\label{detvar} In this section we think of a marking $\ell=\sum_{i=1}^ma_iv_i$  coming from the edges  (for $q=1$ we have $\pm v_i\pm v_j$)   as a {\em map} from $V^{\oplus m}$ to $V$.  Here $V$ is a vector space where the $v_i$ belong. Thus a list of $k$ markings is thought of as a map $\rho:  V^{\oplus m}=V\otimes\C^m\to V^{\oplus k}=V\otimes\C^k$. Such a map is given by a $k\times m$ matrix $A$ and $\rho= 1_V\otimes A$ so that $Im(\rho)=V\otimes Im(A),\ \ker(\rho)=V\otimes\ker(A)$.

When $\dim(V)=n$  we shall be interested in particular in $n$--tuples of markings. In this case we have
\begin{lemma}
An  $n$--tuples of markings $m_i:= \sum_j a_{ij}v_j$ is formally linearly independent -- that is the $n\times m$ matrix of the $a_{ij}$ has rank $n$-- if and only if the associated   map  $\rho :V^{\oplus m}\to V^{\oplus n}$ is surjective. 
\end{lemma} We may identify  $V^{\oplus n}$ with $n\times n$ matrices and  we have   the determinantal variety  $D_n$ of  $V^{\oplus n},$ defined by the vanishing of the determinant  $\det$ (an irreducible polynomial),  and formed by all the $n$--tuples of vectors $v_1,\ldots,v_n$ which are linearly dependent. The variety $D_n$ defines a similar irreducible determinantal variety $D_\rho:= \rho ^{-1}(D_n)$ in $V^{\oplus m}$ which   depends on  the map $\rho$. This is a proper hypersurface if and only if $\rho$ is surjective. We have already remarked that, in this case, $D_\rho$ is an irreducible hypersurface  with equation the irreducible polynomial $\det\circ\rho$. We need to see when different lists of markings give rise to different determinantal  varieties in $V^{\oplus m}$.  
\begin{lemma}\label{tras} Given a surjective map $\rho :V^{\oplus m}\to V^{\oplus n},$
a vector $a\in V^{\oplus m}$ is such that $a+b\in D_\rho,\ \forall b\in D_\rho$ if and only if $\rho(a)=0$.
\end{lemma}
\begin{proof}
Clearly  if $\rho(a)=0$ then $a$ satisfies the condition. Conversely if $\rho(a)\neq 0$, we think of $\rho(a)$ as a non zero matrix  $A$ and it is easily seen that there is a matrix $B=\rho(b)\in D_n$  such that  $A+B=\rho(a+b)\notin D_n$.
\end{proof}
Let $\rho_1,\rho_2:V^{\oplus m}\to V^{\oplus n}$ be two surjective maps,  given by two $n\times m$ matrices $ A=(a_{i,j}),\ B=(b_{i,j});\ a_{i,j},b_{i,j}\in\mathbb C$ .
\begin{proposition}\label{kke}
$\rho_1^{-1}(D_n)=\rho_2^{-1}(D_n)$ if and only if  the two matrices $A,B$ have the same kernel.
\end{proposition}
\begin{proof} The two matrices $A,B$ have the same kernel if and only if $\rho_1,\rho_2 $   have the same kernel.
By Lemma \ref{tras}, if  $\rho_1^{-1}(D_n)=\rho_2^{-1}(D_n)$ then  the two matrices $A,B$ have the same kernel. Conversely if  the two matrices $A,B$ have the same kernel we can write $B=CA$ with $C$ invertible.  Clearly $CD_n=D_n$ and the claim follows.
\end{proof}
We shall also need the following well known fact:
\begin{lemma}\label{zade}
Consider the determinantal variety $D,$ given by  $d(X)=0,$ of $n\times n$  complex  matrices of determinant zero. The real points of $D$ are Zariski dense in $D$.\footnote{this means that a polynomial vanishing on the real points of $D$ vanishes also on the complex points.}
\end{lemma}
\begin{proof}
Consider in    $D$   the set of matrices    of rank exactly $n-1$. This set is dense in $D$ and  obtained from a  fixed matrix (for instance the diagonal matrix $I_{n-1}$ with all 1 except one 0)  by multiplying $AI_{n-1}B$  with $A,B$ invertible matrices.  If a polynomial $f$  vanishes on the real points of  $D$ then $F(A,B):=f(AI_{n-1} B)$  vanishes  for all $A,B$  invertible matrices and real. This set is the set of  points in $\mathbb R^{2n^2}$  where a polynomial (the product of the two determinants) is non zero. But a polynomial which vanishes in all the  points of any space $\mathbb R^{s}$  where another  polynomial   is non zero is necessarily the zero polynomial. So $f$ vanishes also on complex points. This is the meaning of  Zariski dense.
\end{proof}\bibliographystyle{plain}

\bibliography{bibliografia}
\end{document}